\documentclass[12pt]{amsart}
\usepackage{ amsmath, amsthm, amsfonts, amssymb, color}
 \usepackage{mathrsfs}
\usepackage{amsfonts, amsmath}
 \usepackage{amsmath,amstext,amsthm,amssymb,amsxtra}
 \usepackage{txfonts} 
 \usepackage[colorlinks, citecolor=blue,pagebackref,hypertexnames=false]{hyperref}
 \allowdisplaybreaks
 \usepackage{pgf,tikz}
\begin{document}

 \baselineskip 16.6pt
\hfuzz=6pt

\widowpenalty=10000

\newtheorem{cl}{Claim}
\newtheorem{theorem}{Theorem}[section]
\newtheorem{proposition}[theorem]{Proposition}
\newtheorem{corollary}[theorem]{Corollary}
\newtheorem{lemma}[theorem]{Lemma}
\newtheorem{definition}[theorem]{Definition}
\newtheorem{assum}{Assumption}[section]
\newtheorem{example}[theorem]{Example}
\newtheorem{remark}[theorem]{Remark}
\renewcommand{\theequation}
{\thesection.\arabic{equation}}

\def\SL{\sqrt H}

\newcommand{\mar}[1]{{\marginpar{\sffamily{\scriptsize
        #1}}}}

\newcommand{\as}[1]{{\mar{AS:#1}}}

\newcommand\R{\mathbb{R}}
\newcommand\RR{\mathbb{R}}
\newcommand\CC{\mathbb{C}}
\newcommand\NN{\mathbb{N}}
\newcommand\ZZ{\mathbb{Z}}
\newcommand\hDelta{{\bf L}}
\def\RN {\mathbb{R}^n}
\renewcommand\Re{\operatorname{Re}}
\renewcommand\Im{\operatorname{Im}}

\newcommand{\mc}{\mathcal}
\newcommand\D{\mathcal{D}}
\def\hs{\hspace{0.33cm}}
\newcommand{\la}{\alpha}
\def \l {\alpha}
\def\ls{\lesssim}
\def\su{{\sum_{i\in\nn}}}
\def\lz{\lambda}
\newcommand{\eps}{\varepsilon}
\newcommand{\pl}{\partial}
\newcommand{\supp}{{\rm supp}{\hspace{.05cm}}}
\newcommand{\x}{\times}
\newcommand{\lag}{\langle}
\newcommand{\rag}{\rangle}

\newcommand\wrt{\,{\rm d}}
\newcommand{\botimes}{\bar{\otimes}}
\def\nn{{\mathbb N}}
\def\bx{{\mathbb X}}
\def\fz{\infty}
\def\r{\right}
\def\lf{\left}
\def\cm{{\mathcal M}}
\def\cs{{\mathcal S}}
\def\rr{{\mathbb R}}
\def\lm{\mathcal{N}}
\def\rn{{{\rr}^n}}
\def\zz{{\mathbb Z}}
\def\cl{{\mathcal L}}
\def\cq{{\mathcal Q}}
\def\cd{{\mathcal D}}

\theoremstyle{assumption}
\newtheorem{assumption}[theorem]{Assumption}

\title[Operator-Valued Hardy spaces and BMO Spaces on Spaces of Homogeneous Type]{Operator-Valued Hardy spaces and BMO spaces on Spaces of Homogeneous Type}

\author{Zhijie Fan}
\address{Zhijie Fan, School of Mathematics and Information Science,
Guangzhou University, Guangzhou 510006, China}
\email{fanzhj3@mail2.sysu.edu.cn}

\author{Guixiang Hong}
\address{Guixiang Hong, Institute for Advanced Study in Mathematics, Harbin Institute of Technology, Harbin 150001, China}
\email{gxhong@hit.edu.cn}

\author{Wenhua Wang}
\address{Wenhua Wang, Institute for Advanced Study in Mathematics, Harbin Institute of Technology, Harbin 150001, China}
\email{whwangmath@whu.edu.cn}


  \date{\today}

\begin{abstract}
Let $\mathcal{M}$ be a von Neumann algebra equipped with a normal semifinite faithful trace, $(\mathbb{X},\,d,\,\mu)$ be a space of homogeneous type in the sense of Coifman and Weiss, and $\mathcal{N}=L_\infty(\mathbb{X})\overline{\otimes}\mathcal{M}$. In this paper, we introduce and then conduct a systematic study on the operator-valued Hardy space $\mathcal{H}_p(\mathbb{X},\,\mathcal{M})$ for all $1\leq p<\infty$ and operator-valued BMO space $\mathcal{BMO}(\mathbb{X},\,\mathcal{M})$. The main results of this paper include  $H_1$--$BMO$ duality theorem, atomic decomposition of $\mathcal{H}_1(\mathbb{X},\,\mathcal{M})$, interpolation between these Hardy spaces and BMO spaces, and equivalence between mixture Hardy spaces and $L_p$-spaces. 
In particular, without the use of non-commutative martingale theory as in Mei's seminal work \cite{m07}, we provide a direct proof for the interpolation theory.
Moreover, under our assumption on Calder\'{o}n representation formula, these results are even new when going back to the commutative setting for spaces of homogeneous type which fails to satisfy reverse doubling condition. As an application,  we obtain the $L_p(\mathcal{N})$-boundedness of operator-valued Calder\'{o}n-Zygmund operators.
\end{abstract}

\maketitle


\tableofcontents\newpage

\section{Introduction}
\setcounter{equation}{0}
\subsection{Background and motivation}

The theory of Hardy spaces and BMO spaces on the Euclidean space $\rr^n$, which was originated from the seminal works of Stein--Weiss \cite{s60} and John--Nirenberg \cite{jn61}, respectively, has been developed rapidly and plays an essential role in harmonic analysis and \textbf{PDE}s   (see e.g. \cite{c76,c74,cr80,cw77,fs72,g14,s60}). As endpoint substitutes of Lebesgue spaces $L_1(\mathbb{R}^n)$ and $L_\infty(\mathbb{R}^n)$, respectively, Hardy space and BMO space have been proven widely applicable  to obtain $L_p$-boundedness ($1<p<\infty$) of many singular integral operators ({\bf SIO}s) on $\mathbb{R}^n$ by interpolation. However, it is less effective to investigate those ones falling beyond Euclidean world. Motivated by the study of $L_p$-boundedness of {\bf SIO}s beyond Euclidean framework, generalizing the function space theory from Euclidean setting to more general settings has became a central theme in modern harmonic analysis.

In the early 1970s, Coifman and Weiss introduced spaces of homogeneous type \cite{cw71,cw77}, which generalizes  Euclidean space and includes many important spaces beyond standard Euclidean setting in analysis, such as Lie group with polynomial growth \cite{VSC}, Riemannian manifold with non-negative Ricci curvature \cite[Section 5]{MR1103113} and Euclidean space equipped with Bessel measure \cite{MS1965}. During the past half century, this generalization has achieved lots of significant progress in the theory of singular integrals and function spaces  (see e.g. \cite{c77,dy03,dy05,hmy06,HS1994,HHLLYY,MS} and the reference therein). In particular, R. Coifman and G. Weiss \cite{cw77} introduced the BMO space on space of homogeneous type  $\bx$, and obtain its predual space---the atomic Hardy space $H_1^{\mathrm{at}}(\bx)$. In 2006, Y. Han et al. \cite{hmy06} further introduced the Hardy space $H_1(\bx)$, defined via the Lusin area integral function, and proved that its dual space is the BMO space $\mathrm{BMO}(\bx)$. Recently, motivated by the celebrated construction of orthonormal wavelet basis on $L_2(\bx)$ due to P. Auscher and T. Hyt\"{o}nen \cite{AH}, Z. He et al. \cite{HHLLYY} established a real-variable theory of Hardy spaces on
spaces of homogeneous type.

On the other hand, motivated by the works on matrix-valued harmonic analysis (operator-weighted norm inequalities, operator Hilbert transform), the development of the non-commutative martingale inequalities (see e.g. \cite{j02,jx03,jx08,MR2235948,px97,MR1929141,MR2319715}) and the Littlewood-Paley-Stein theory of quantum Markov semigroup (see e.g.  \cite{MR3920831,jl06,jm10,jm12}), T. Mei \cite{m07} initiated a remarkable work on the theory of operator-valued Hardy space, defined via the operator-valued Lusin area integral function associated to the Poisson semigroup. To further enrich this theory, some classical characterizations, including wavelet characterization \cite{HWW,hy13} and Lusin area function characterization associated to test function \cite{xxx16}, and the theory of several important classical function spaces \cite{MR4525615,xx18,MR4015958} were extended to this operator-valued setting. These operator-valued function spaces have been proven to be very useful in operator-valued Calder\'{o}n-Zygmund theory (see e.g. \cite{MR3734984,hx21,MR2476951,MR4525615,xxx16}) and some central topics (function space theory and Fourier multiplier theory) in several completely non-commutative models (see e.g. \cite{FHW,MR3679616,MR4320770,MR3283931,xxx16}), including matrix algebra, group von Neumann algebra, quantum tori and quantum Euclidean space.

However, although the results in harmonic analysis on spaces of homogeneous type in the commutative setting are fruitful, it is still almost a blank in the non-commutative setting except for the non-commutative maximum inequality established in \cite{hlw21}. Motivated by the development of non-commutative {\bf SIO}s and the subsequent problems about non-commutative {\bf PDE}s on spaces of homogeneous type, the main goal of this paper is to fill in this blank by initiating the study of operator-valued Hardy spaces and BMO spaces on spaces of homogeneous type.

%

%



\subsection{Organization and our main results}
Let $\cm$ be a von Neumann algebra equipped with a normal semifinite faithful trace $\tau$ and $\mathcal{N}=L_\infty(\bx)\overline{\otimes}\mathcal{M}$ be the von Neumann algebra tensor product equipped with tensor trace ${\rm Tr}=\int\otimes\tau$.
 Moreover, $L_p(\mathcal{N})$ denotes the  non-commutative $L_p$-space associated with $(\mathcal{N},\,{\rm Tr})$ (see Section \ref{nclp} and \cite{PX} for more details about non-commutative integration theory). Let $\dagger\in\{c,\,r,\,cr\}$ (the set represents column, row and mixture), then the organization of our paper is as follows.

In Section \ref{s2}, we first recall some  definitions and properties of spaces of homogeneous type and non-commutative $L_p$-space. We then introduce operator-valued BMO spaces $\mathcal{BMO}^\dagger(\bx,\,\cm)$, and the Hardy spaces $\mathcal{H}_p^\dagger(\bx,\,\cm)$ for $1\leq p<\infty$, via suitable Lusin area integral function.

In Section \ref{s3}, we establish $\cm$-valued Carleson measure characterization for $\mathcal{BMO}^\dagger(\bx,\,\cm)$, and the Fefferman-Stein duality theorem between the operator-valued Hardy spaces and BMO spaces introduced in Section \ref{s2}. To establish this duality theorem, our strategy is to combine the techniques in \cite{m07} with \cite{xxx16}. This combination sometimes is quite straightforward, but sometimes requires significantly extra efforts due to the lack of analogue of harmonicity and geometric structure on spaces of homogeneous type.

In Section \ref{s31}, as an application of Theorem \ref{duua}, we introduce  $\cm^\dagger$-atom Hardy spaces
$\mathcal{H}_{1}^{\dagger,\mathrm{at}}(\bx,\,\cm)$ and then establish
the atomic characterization of operator-valued Hardy spaces $\mathcal{H}_1^{\dagger}(\bx,\,\cm)$. As a byproduct, we obtain an analogue of Coifman-Weiss Theorem \cite{cw77} in the operator-valued setting.
%

In Section \ref{s4}, we introduce $L_p$-space analogues of the BMO spaces $L_p\mathcal{MO}^{\dagger}(\bx,\,\cm)$ for  $p\in(2,\,\infty)$. Then we combine the argument in the proof of Theorem \ref{duua} with some techniques on non-commutative maximal inequalities and collection of adjacent systems of dyadic cubes developed in \cite{hk12} to establish the duality between $\mathcal{H}_{p'}^{\dagger}(\bx,\,\cm)$ and $L_p\mathcal{MO}^{\dagger}(\bx,\,\cm)$ for $p\in(2,\infty)$.

In Section \ref{s5}, we apply a recent breakthrough in \cite{hlw21} about non-commutative Hardy-Littlewood maximal inequality on spaces of homogeneous type to establish the equivalence between $\mathcal{H}_p^{\dagger}(\bx,\,\cm)$
 and $L_p\mathcal{MO}^{\dagger}(\bx,\,\cm)$ for $p\in(2,\,\infty)$, and the duality between $\mathcal{H}_p^{\dagger}(\bx,\,\cm)$ and $\mathcal{H}_{p'}^{\dagger}(\bx,\,\cm)$ for $p\in(1,\,\infty)$.

In Section \ref{s6}, we first establish the relationship between $L_p\mathcal{MO}^{\dagger}(\bx,\,\cm)$ (resp. $\mathcal{H}_p^{\dagger}(\bx,\,\cm)$) and its dyadic version. Then we use the above relationship to show the equivalence between mixture Hardy space $\mathcal{H}_p^{cr}(\bx,\,\cm)$ and $L_p(\lm)$ for $p\in(1,\,\infty)$.

In Section \ref{s7}, based on several duality theorems and equivalence theorems established in the previous sections,
we establish the interpolation theorem between operator-valued BMO spaces and Hardy spaces. This implies that $\mathcal{H}_{1}^{\dagger}(\bx,\,\cm)$ and $\mathcal{BMO}^{\dagger}(\bx,\,\cm)$ are suitable substitutes of $L_1(\mathcal{N})$ and $L_\infty(\mathcal{N})$, respectively, and therefore, they are naturally expected to be useful in obtaining the $L_p(\mathcal{N})$-boundedness ($1<p<\infty$) of operator-valued {\bf SIO}s as in the commutative setting. It is worthwhile to mention that non-commutative martingale theory usually play crucial roles in the proof of interpolation theorem between non-commutative Hardy spaces and BMO spaces (see e.g. \cite{jm12,m07}), but we provide a more direct proof without involving the use of non-commutative martingale.

In Section \ref{s8}, as an application of operator-valued Hardy spaces and BMO spaces studied in this paper, we obtain the $L_p(\mathcal{N})$ boundedness of non-commutative Calder\'{o}n-Zygmund operators, with operator-valued kernel, by establishing $L_\infty(\mathcal{N})\rightarrow \mathcal{BMO}^{\dagger}(\bx,\,\cm)$ boundedness of these operators.

Let $(\mathcal{X}_0,\,\mathcal{X}_1)$ be a
compatible couple of quasi-normed spaces. For $0<\theta<1$ and $q\in(0,\,\infty]$, we let $[\mathcal{X}_0,\,\mathcal{X}_1]_{\theta}$ and $[\mathcal{X}_0,\,\mathcal{X}_1]_{\theta,\,q}$ be the {\it complex interpolation space} and {\it real interpolation space}, respectively. See e.g. \cite{bl76} for these definitions. Then for the convenience of the readers, we summarize our main results as follows.
\begin{theorem}\label{duua}
Let $\dagger \in \{c,r,cr\}$, then
\begin{enumerate}
  \item $(\mathcal{H}_1^{\dagger}(\bx,\,\cm))^{\ast}\backsimeq \mathcal{BMO}^{\dagger}(\bx,\,\cm)$;
   \item $\mathcal{H}_1^{\dagger}(\bx,\,\cm)=\mathcal{H}_{1}^{\dagger,\mathrm{at}}(\bx,\,\cm)$;
  \item $(\mathcal{H}_p^{\dagger}(\bx,\,\cm))^{\ast}\backsimeq \mathcal{H}_{p'}^{\dagger}(\bx,\,\cm),$ where $1<p<\infty$;
  \item $\mathcal{H}_{p}^{\dagger}(\bx,\,\cm)= L_{p}\mathcal{MO}^{\dagger}(\bx,\,\cm),$ where $2< p<\infty$;

  \item $\mathcal{H}_p^{cr}(\bx,\,\cm)= L_p(\mathcal{N})$, where $1<p<\infty$.
\end{enumerate}

\end{theorem}

\begin{theorem}
Let $\dagger \in \{c,r,cr\}$ and $1\leq q<p<\infty$, then
\begin{enumerate}
  \item $\lf[\mathcal{BMO}^{\dagger}(\bx,\,\cm),\,
\mathcal{H}_q^{\dagger}(\bx,\,\cm)\r]_{\frac{q}{p}}=\lf[\mathcal{BMO}^{\dagger}(\bx,\,\cm),\,\mathcal{H}_q^\dagger(\bx,\,\cm)\r]_{\frac{q}{p},\,p}=
\mathcal{H}_p^{\dagger}(\bx,\,\cm)$;
  \item $
\lf[\mathcal{X},\,\mathcal{Y}\r]_{\frac{1}{p}}=\lf[\mathcal{X},\,\mathcal{Y}\r]_{\frac{1}{p},\,p}=
L_p(\lm),\\
$
where
$\mathcal{X}=\mathcal{BMO}^{cr}(\bx,\,\cm)$ or $L_{\infty}(\lm)$,
$\mathcal{Y}=\mathcal{H}_{1}^{cr}(\bx,\,\cm)$ or $L_1(\lm)$.
\end{enumerate}
\end{theorem}
There are two comments about our results.
\begin{remark}{\rm
(1) From Remark \ref{a2.5} one shall see that via appropriate choices of approximation of identity, our result can be regarded as a starting point of studying operator-valued function spaces and singular integrals associated with general differential operators beyond classical Euclidean framework, which has achieved a lot of milestone progress in the commutative setting (see e.g. \cite{dy05,dy05b,HLMMY,HM2009}) but is still a blank in the non-commutative setting.

(2) The reverse doubling condition is sometimes needed for considering the real variable theory of Hardy space on spaces of homogeneous type (see e.g. \cite{gly08,hmy06,hmy08}) even when going back to the commutative setting, but the novel strategy developed in the more general non-commutative setting enables us to get rid of this assumption.}
\end{remark}


\subsection{Notation}
Conventionally, we set $\nn:=\{1,\, 2,\,\ldots\}$, $\zz_+:=\{0\}\cup\nn$ and $\rr_+:=(0,\,\infty)$. Throughout the whole paper, we denote by $C$ a positive constant which is independent of the main parameters, but it may
vary from line to line. We use $A\lesssim B$ to denote the statement that $A\leq CB$ for some constant $C>0$, and $A\thicksim B$ to denote the statement that $A\lesssim B$ and $B\lesssim A$. If $D\ls F$ and $F\ls D$, then we write $D\sim F$. For any $\lambda>0$ and ball $B:=B_d(x_B,\,r_B)\subset \bx$, we use the notation $\lambda B$ to denote the ball centred at $x_B$ with radius $\lambda r_B$. For any $1\leq p\leq\infty$, we denote by $p'$ the conjugate of $p$, which satisfies $\frac{1}{p}+\frac{1}{p'}=1$. We denote the average of a function $f$ over a ball $B$ by
\begin{align*}
f_B:=\fint_{B}f(x)d\mu(x):=\frac{1}{\mu(B)}\int_{B}f(x)d\mu(x).
\end{align*}
We denote
 \begin{align}\label{UjB}
 U_j(B)=\left\{\begin{array}{ll}2B, &j=0,\\2^{j+1}B\setminus2^jB, &j\geq 1.\end{array}\right.
 \end{align}
For a measurable set $E \subset\bx$, let $E^\complement:=\bx\setminus E$ be the complement of $E$ and we denote by $\chi_E$ its characteristic function.

\bigskip




\section{Preliminaries \label{s2}}
\setcounter{equation}{0}

\subsection{Space of homogeneous type}\label{homogene}
A function $d:\bx\times\bx\rightarrow[0,\,\fz)$ is said to be a quasi-metric function if it satisfies
\begin{enumerate}
  \item $d(x,\,y)=d(y,\,x)\geq 0$ for all $x,\,y\in\bx$;
  \item $d(x,\,y)=0$ if and only if $x=y$;
  \item There is a constant $A_d\in[1,\infty)$ such that for all $x,\,y,\,z\in\bx$,
 \begin{align}\label{noimp}d(x,\,y)\leq A_d[d(x,\,z)+d(z,\,y)].\end{align}
\end{enumerate}
Moreover, $\mu$ is said to be a doubling measure if there is a constant $C>0$ such that
\begin{align}\label{e2.1}
\mu(B_d(x,\,2r))\leq C\mu(B_d(x,\,r)).
\end{align}
$(\bx,\, d,\, \mu)$ is said to be a {\it space of
homogeneous type} in the sense of Coifman and Weiss \cite{cw71,cw77} if
$\bx$ is a metric space with a quasi-metric function $d$ and a nonnegative, Borel, doubling measure $\mu$. Note that the constant $A_d$ on the right-hand side of inequality \eqref{noimp} plays no essential role in our proof below, so we shall only present our proof for the case of $A_d=1$.

In what follows, for any ball $B_d(x,\,r):=\{y\in\bx:d(x,\,y)<r\}$, we define the volume functions
\begin{align*}
V_r(x):=\mu(B_d(x,\,r)) \ \ \mathrm{and} \ \ V(x,\,y):=\mu(B_d(x,\,d(x,\,y))).
\end{align*}
It follows from \eqref{e2.1} that the following strong homogeneity property holds:
\begin{align}\label{e2.2}
V_{\lambda r}(x)\leq \widetilde{C_0}\lambda^nV_{r}(x)
\end{align}
for some constants $\widetilde{C_0},n>0$  independent of all $\lambda\geq 1$ and $x\in\bx$. The smallest constant $n$ satisfying the above inequality is said to be the homogeneous dimension of $\bx$.

\begin{lemma}\label{l2.0}
Assume that $(\bx,\, d,\, \mu)$ is a space of
homogeneous type. Then
\begin{enumerate}
\item[\rm{(i)}]{\rm \cite[Lemma 2.1]{hmy08}} For any $x\in\bx$ and $r>0$, one has
$V(x,\,y)\thicksim V(y,\,x)$ and
$$V_r(x)+V_r(y)+V(x,\,y)\thicksim V_r(y)+V(x,\,y)\thicksim V_r(x)+V(x,\,y)
\thicksim \mu(B(x,\,r+d(x,\,y))).$$
\item[\rm{(ii)}] {\rm\cite[Proposition 3.2]{ghl09}} There exist two constants $C>0$ and $0\leq\gamma\leq n$ such that
$$V_{r_1}(x)\lesssim \lf[\frac{r_1+d(x,\,y)}{r_2}\r]^{\gamma}V_{r_2}(y).$$
Here all the implicit constants are independently of $x,\,y\in\bx$ and $r>0$.
\end{enumerate}
\end{lemma}
The following lemma can be regarded as a substitution of dyadic cube system in the Euclidean space $\rn$.
\begin{lemma} \cite[Corollary 7.4]{hk12}\label{l6.1}
Let $(\mathbb{X},\,d,\,\mu)$ be a space of homogeneous type. Then there exists a finite collection of families  $\mathcal{Q}^{(1)}, \mathcal{Q}^{(2)},\cdots, \mathcal{Q}^{(I_0)}$ such that $\mathcal{Q}^{(i)}:=\cup_{k\in\mathbb{Z}}\mathcal{Q}^{(i)}_k$ and that each $\{\mathcal{Q}^{(i)}_k\}_{k\in\zz}$ is a sequence of partitions of $\bx$ and the following
conditions hold:
\begin{enumerate}
\item[\rm{(i)}] for each integer $1 \leq i \leq I_0$ and for each $k\in\zz$, the partition $\mathcal{Q}^{(i)}_{k+1}$ is a refinement of
the partition $\mathcal{Q}^{(i)}_{k}$;
\item[\rm{(ii)}] there exists a constant $C>0$ such that for any ball $B:=B_d(x,\,r)\subset\bx$ with $\delta^{k+3}<r\leq \delta^{k+2}$, $k\in\mathbb{Z}$, there exist integers
$i\in\{1,2,\cdots,I_0\}$ and $Q\in\mathcal{Q}^{(i)}_{k}$ such that
$B_d(x,\,r)\subset Q\subset B_d(x,\,Cr)$.
Conversely, for every
dyadic cube $Q\in \mathcal{Q}^{(i)}_k$, there exists a ball $B\subset\bx$ such that $Q\subset B$ and $\mu(B)\leq C\mu(Q)$.
\end{enumerate}
\end{lemma}
For each integer $1 \leq i \leq I_0$ and $k\in\zz$, denote by $\sigma^{(i)}_k$
 the $\sigma$-algebra generated by $\mathcal{Q}^{(i)}_k$. Then
$ L_{\infty}(\bx,\,\sigma^{(i)}_k,\,d\mu(x))\overline{\otimes} \cm$ is a von
Neumann subalgebra of $\lm$.
Moreover, $ \{L_{\infty}(\bx,\,\sigma^{(i)}_k,\,d\mu(x))\overline{\otimes} \cm\}_{k\in\zz}$ is a sequence of increasing von Neumann subalgebras such that its union is $w^*$-dense in $\lm$.
\subsection{Non-commutative $L_p$ space}\label{nclp}
Let us recall the definition and some properties of non-commutative $L_p$-spaces (see \cite{PX} for more details about non-commutative integration theory). To begin with, let $S_{\cm}^+$ be the set of all positive element $x\in\cm$ such that
$$\tau(s(x))<\infty,$$
where $s(x)$ denotes the least projection $e\in\cm$, called the support of $x$, such that
$exe=x.$
Let $S_\cm$ be the linear span of $S_{\cm}^+$. For any $p\in(0,\,\infty)$, we define
$$\|x\|_{L_p(\cm)}:=(\tau(|x|^p))^{1/p}, \ \ \ x\in S_\cm,$$
where $|x|:=(x^*x)^{1/2}$. We define the {\it non-commutative $L_p$-space} associated with $(\cm,\,\tau)$, denoted by $L_p(\cm)$, be the completion
of $(S_\cm, \|\cdot\|_{L_p(\cm)})$. For convenience, we usually set $L_{\infty}(\cm)=\cm$ equipped
with the operator norm $\|\cdot\|_{\cm}$.

Let $(\Omega,\,d\nu)$ be a measurable space, then $f$ is said to be an $S_{\cm}$-valued simple function on $\Omega$ if
 $$f=\sum_{j=1}^Nm_j\cdot\chi_{E_j},$$
where $m_j\in S_{\cm}$ and $E_j$'s are measurable disjoint subsets of $\Omega$
with $\nu(E_j)<\infty$.
For any $S_{\cm}$-valued simple function $f$, we define
$$\|f\|_{L_p(\cm;\,L_2^c(\Omega))}:=\lf\|\lf(\int_{\Omega}f(x)^*f(x)\,d\nu(x)\r)^{1/2}
\r\|_{L_p(\cm)}$$
and
$$\|f\|_{L_p(\cm;\,L_2^r(\Omega))}:=\lf\|\lf(\int_{\Omega}f(x)f(x)^*\,d\nu(x)\r)^{1/2}\r\|_{L_p(\cm)}.$$
Furthermore, we define the {\it  column space} $L_p(\cm;\,L_2^c(\Omega))$ (resp. {\it row space} $L_p(\cm;\,L_2^r(\Omega))$) be the completion of the space of all $S_{\cm}$-valued simple functions with respect to the topology of $L_p(\cm;\,L_2^c(\Omega))$ (resp. $L_p(\cm;\,L_2^r(\Omega))$).

The following lemma establish the duality between the above spaces.
\begin{lemma}{\rm\cite[p.9]{m07}}
For any $p\in[1,\,\infty)$, we have
$$\lf(L_{p}(\cm;\,L_2^c(\Omega))\r)^*=
L_{p'}(\cm;\,L_2^c(\Omega)),$$
where the duality bracket is given by
$$\langle f,\,g\rangle:=\tau\int f(x)g^*(x)\,d\mu(x),$$
 for any $f\in L_p(\cm;\,L_2^c(\Omega))$ and $g\in L_{p'}(\cm;\,L_2^c(\Omega))$.
Similar result also holds for row spaces.
\end{lemma}


Recall the following Kadison-Schwarz inequality, which follows from the operator convexity of $t\rightarrow |t|^2$ (see e.g. \cite[p.12]{m07}).
\begin{lemma}\label{e2.5}
Let $(\Omega,\,d\nu)$  be a measure space. Assume that $f:\Omega\rightarrow \mathbb{R}_+$ and $g:\Omega\rightarrow\cm$ such that all members of the below inequality make sense. Then we have
\begin{align*}
\lf|\int_{\Omega}f(x)g(x)\,d\nu(x)\r|^2\leq\int_{\Omega}f(x)^2\,d\nu(x)\int_{\Omega}|g(x)|^2\,d\nu(x),
\end{align*}
where `$\leq$' is understood as the partial order in the positive cone of $\mathcal{M}$.
\end{lemma}

\subsection{Non-commutative $H_p$ and $BMO$ spaces}
Now we introduce the operator-valued column, row and mixture BMO spaces on spaces of homogeneous type. In what follows,
let $ L_\infty(\mathcal{M};\,L_2^{{\rm loc},c}(\bx))$ denote the set of all locally $L_2$-integrable functions on $\bx$ with values in $\cm$.

\begin{definition}\label{d2.4}
The {\it{column $\mathrm{BMO}$ space}} $\mathcal{BMO}^{c}(\bx,\,\cm)$ is defined as
\begin{eqnarray*}
\mathcal{BMO}^{c}(\bx,\,\cm):=\lf\{g\in L_\infty(\mathcal{M};\,L_2^{{\rm loc},c}(\bx)):\|g\|_{\mathcal{BMO}^{c}(\mathbb{X},\,\cm)}<\infty\r\},
\end{eqnarray*}
where
$$\|g\|_{\mathcal{BMO}^{c}(\mathbb{X},\,\cm)}:=\sup_{x\in \mathbb{X},\,r>0}\lf\|\lf(
\fint_{B_d(x,\,r)}|g(y)-g_{B_d(x,\,r)}|^2\,d\mu(y)\r)^{1/2}\r\|_{\cm}.$$
Similarly, we define the {\it  row $\mathrm{BMO}$ space} $\mathcal{BMO}^r(\bx,\,\cm)$ as the space of all $g\in L_\infty(\mathcal{M};\,L_2^{{\rm loc},c}(\bx))$ such that $g^{\ast}\in\mathcal{BMO}^c(\bx,\,\cm)$ with the norm
$\|g\|_{\mathcal{BMO}^r(\bx,\,\cm)}:=\|g^{\ast}\|_{\mathcal{BMO}^c(\bx,\,\cm)}$, and the {\it mixture {\rm BMO} space}
$$\mathcal{BMO}^{cr}(\bx,\,\cm):=\mathcal{BMO}^c(\bx,\,\cm)\cap\mathcal{BMO}^r(\bx,\,\cm)$$ with the norm
$$\|g\|_{\mathcal{BMO}^{cr}(\bx,\,\cm)}:=\max\lf\{\|g\|_{\mathcal{BMO}^c(\bx,\,\cm)},\,
\|g\|_{\mathcal{BMO}^r(\bx,\,\cm)}\r\}.$$
\end{definition}

\begin{remark}{\rm
\begin{enumerate}
\item[\rm{(i)}] When $(\bx,\,d,\,\mu):=(\rr^d,\,|\,\cdot\,|,\,dx)$, these $\mathrm{BMO}$ spaces goes back to the $\mathrm{BMO}$ spaces studied in \cite[p.14]{m07}.
\item[\rm{(ii)}]  When
 $\cm:=\mathbb{C}$, these spaces coincide with the $\rm BMO$ space $\mathrm{BMO}(\bx)$ studied in \cite{cw77}.
\end{enumerate}}
\end{remark}
Next, we shall introduce the definition of operator-valued Hardy space via Lusin area function. To begin with, we impose the following assumptions throughout the whole paper.
\begin{assumption}\label{a2.5}
\rm Let $(\bx,\,d,\,\mu)$ be a space of homogeneous type (see Section \ref{homogene} for definition) with $\mu(\bx)=+\infty$.
Assume that there exists a {\it Calder\'{o}n reproducing formula of order $(\epsilon_1,\,\epsilon_2)$} on $L_2({\bx})$ for some $\epsilon_1\in(0,\,1]$ and $\epsilon_2\in(0,\,\fz)$, that is, there exists a family of self-adjoint bounded linear operators
$\{\mathbf{D}_t\}_{t>0}$ on $L_2(\bx)$ such that
for all $f\in L_2(\bx)$,
\begin{align}\label{e2.50}
f=\int_0^{\infty}\mathbf{D}_t^2(f)\frac{dt}{t},
\end{align}
and moreover, for all $f\in L_2({\bx})$ and $x\in\bx$,
$$\mathbf{D}_t(f)(x)=\int_{\bx}
\mathbf{D}_t(x,\,y)f(y)\,d\mu(y),$$
where
$\mathbf{D}_t(\cdot,\,\cdot)$ is a measurable function from $\bx\times\bx$ to $\mathbb{C}$ satisfying the following conditions:
there exists a positive constant $C_1$ such that, for all $t\in\rr_+$ and all $x,\,x',\,y\in\bx$ with
$d(x,\,x')\leq\frac{[t + d(x, y)]}{2}$,
\begin{enumerate}
\item[\rm $(\mathbf{H}_1)$]
 $\lf|\mathbf{D}_t(x,\,y)\r|\leq C_1
\frac{1}{V_t(x)+V_t(y)+V(x,\,y)}\lf[\frac{t}{t+d(x,\,y)}\r]^{\epsilon_2}$;\smallskip

\item[\rm $(\mathbf{H}_2)$] $\lf|\mathbf{D}_t(x,\,y)-\mathbf{D}_t(x',\,y)\r|\leq C_1
\lf[\frac{d(x,\,x')}{t+d(x,\,y)}\r]^{\epsilon_1}
\frac{1}{V_t(x)+V_t(y)+V(x,\,y)}\lf[\frac{t}{t+d(x,\,y)}\r]^{\epsilon_2}$;\smallskip

\item[\rm $(\mathbf{H}_3)$]
Property ($\textbf{H}_2$) still holds true with the roles of $x$ and $y$ interchanged;\smallskip

\item[\rm $(\mathbf{H}_4)$] $\int_{\bx}\mathbf{D}_t(x,\,y)\,d\mu(x)=\int_{\bx}\mathbf{D}_t(x,\,y)\,d\mu(y)=0.$
\end{enumerate}
\end{assumption}

\begin{remark}{\rm In the following, we list several ways for the construction of Calder\'{o}n reproducing formula:

(1) Let $L$ be a non-negative self-adjoint operator on $L_2(\bx)$. Then imposing suitable conditions on the heat kernel associated with $L$ yields {\bf CRF} (see e.g. \cite[Proposition 2.6]{hmy06}). A typical non-Euclidean example is the Lie group $G$ with polynomial growth, in which $L$ is taken to be the sub-Laplacian operator and $\mathbf{D}_t$ is taken to be $ct^2Le^{-t^2L}$ or $ct^2Le^{-t\sqrt{L}}$ (see \cite[Theorem VIII.2.7]{VSC}), where $c>0$ is an absolute normalized constant. Furthermore, if $G$ is a nilpotent Lie group, one may also take $\mathbf{D}_t:=c\Phi(t^2L)$, whose convolution kernel is a Schwartz function (see \cite{Hulanicki}), for any smooth function $\Phi:\mathbb{R}_+\rightarrow \mathbb{R}$ supported in $[1/2,2]$;

(2) In Ahlfors $n$-regular metric spaces, which include some fractals and Lipschitz manifolds (see \cite{TH2011}), a Calder\'{o}n reproducing formula was constructed in \cite{Han1994} via some geometry methods;

(3) In the Euclidean space $\mathbb{R}^n$, the Fourier transform can be also applied to construct Calder\'{o}n reproducing formula.}
\end{remark}
For any $S_{\cm}$-valued simple function $f$, the column and row {\it Lusin area functions} of $f$ are defined by
$$\cs^c(f)(x):=\lf(\iint_{\Gamma_x}|\mathbf{D}_t(f)(y)|^2\,\frac{d\mu(y)dt}
{V_t(x)t}\r)^{1/2} \ \
\mathrm{and} \ \ \cs^r(f)(x):=\lf(\iint_{\Gamma_x}\lf|(\mathbf{D}_t(f)(y))^*\r|^2\,\frac{d\mu(y)dt}
{V_t(x)t}\r)^{1/2},$$
respectively, where $\Gamma_x:=\{(y,\,t)\in\bx\times\rr_+:d(x,\,y)<t\}$ with $x\in\bx$. Here and in what follows, we abuse the notation $\mathbf{D}_t$ frequently to denote $\mathbf{D}_t\otimes id_\mathcal{M}$ for simplicity.

For any $p\in[1,\,\infty)$, we define the $\mathcal{H}_p^{c}(\bx,\,\cm)$ and $\mathcal{H}_p^{r}(\bx,\,\cm)$ norms of $f$ by
$$\|f\|_{\mathcal{H}_p^{c}(\bx,\,\cm)}:=
\lf\|\cs^c(f)\r\|_{L_p(\lm)}\ \ \  {\rm and}\ \ \ \|f\|_{\mathcal{H}_p^{r}(\bx,\,\cm)}:=
\lf\|\cs^r(f)\r\|_{L_p(\lm)},$$
respectively.
\begin{definition}
We define the {\it column Hardy space} $\mathcal{H}_p^{c}(\bx,\,\cm)$ (resp. {\it row Hardy space} $\mathcal{H}_p^{r}(\bx,\,\cm)$) to be the completion of the space of all $S_{\cm}$-valued simple functions with finite $\mathcal{H}_p^{c}(\bx,\,\cm)$ (resp. $\mathcal{H}_p^{r}(\bx,\,\cm)$) norm. Moreover, we define the {\it mixture space} $\mathcal{H}^{cr}_p(\bx,\,\cm)$ as follows: when $p\in[1,\,2)$,
$$\mathcal{H}^{cr}_p(\bx,\,\cm):=\mathcal{H}^{c}_p(\bx,\,\cm)+
\mathcal{H}^{r}_p(\bx,\,\cm),$$
equipped with the norm
\begin{align*}
\|f\|_{\mathcal{H}^{cr}_p(\bx,\,\cm)}:=&\inf\lf\{\|f_1\|_{\mathcal{H}^{c}_p(\bx,\,\cm)}+
\|f_2\|_{\mathcal{H}^{r}_p(\bx,\,\cm)}:\,f=f_1+f_2,\,
f_1\in\mathcal{H}^{c}_p(\bx,\,\cm),\,f_2\in\mathcal{H}^{r}_p(\bx,\,\cm)\r\},
\end{align*}
where the infimum is taken over all the decompositions of $f$ as above.
When $p\in[2,\,\infty)$, define
$$\mathcal{H}^{cr}_p(\bx,\,\cm):=\mathcal{H}^{c}_p(\bx,\,\cm)
\cap\mathcal{H}^{r}_p(\bx,\,\cm),$$
equipped with the intersection norm
$$\|f\|_{\mathcal{H}^{cr}_p(\bx,\,\cm)}:=
\max\lf\{\|f\|_{\mathcal{H}^{c}_p(\bx,\,\cm)},\,
\|f\|_{\mathcal{H}^r_p(\bx,\,\cm)}\r\}.$$
\end{definition}
\begin{remark}\label{r2.9}{\rm
\begin{enumerate}
\item[\rm{(i)}]
When it comes back to the commutative setting, i.e.,
 $\cm:=\mathbb{C}$, these spaces go back to the Hardy space $H_p(\bx)$ studied in \cite{hmy06}, where $p\in[1,\,\infty)$.
\item[\rm{(ii)}]
 When $\bx:=\rr^d$,
 the operator-valued Hardy spaces $\mathcal{H}_{p}^c(\bx,\,\cm)$, $\mathcal{H}_{p}^r(\bx,\,\cm)$ and $\mathcal{H}_p^{cr}(\bx,\, \cm)$ are equivalent to the operator-valued Hardy spaces $\mathcal{H}_p^{c}(\rr^d,\,\cm)$, $\mathcal{H}_p^{r}(\rr^d,\,\cm)$ and $\mathcal{H}^{cr}_p(\rr^d,\,\cm)$ in \cite[p.11]{m07}, respectively, for $p\in[1,\,\infty)$, which can be seen from the combination of Theorem \ref{t3.1} with \cite[Theorem 2.4]{m07}.
\item[\rm{(iii)}] It follows from Fubini's theorem, Lemma \ref{l2.0} and representation formula \eqref{e2.50} that for any $S_{\cm}$-valued simple function $f$, there holds:
    $$\mathcal{H}^{c}_2(\bx,\,\cm)=\mathcal{H}^{r}_2(\bx,\,\cm)=\mathcal{H}^{cr}_2(\bx,\,\cm)=L_2(\lm).$$

\end{enumerate}}
\end{remark}

\bigskip

\section{Duality between $H_1$ and $BMO$}\label{s3}
\setcounter{equation}{0}

This section is devoted to establishing the dualities between $\mathcal{H}_1^{\dagger}(\bx,\,\cm)$ and $\mathcal{BMO}^\dagger(\bx,\,\cm)$ for $\dagger\in\{c,r,cr\}$, which can be stated specifically as follows:
\begin{theorem}\label{t3.1}
Let $\dagger\in\{c,r,cr\}$, then
$$(\mathcal{H}_1^{\dagger}(\bx,\,\cm))^{\ast}\backsimeq\mathcal{BMO}^\dagger(\bx,\,\cm)$$
in the following sense:
\begin{enumerate}
\item[\rm{(i)}]
Each $g\in \mathcal{BMO}^\dagger(\bx,\,\cm)$ defines a continuous linear
functional $\mathcal{L}_g$ on $\mathcal{H}_1^{\dagger}(\bx,\,\cm)$ given by
\begin{align}\label{dua}
\mathcal{L}_g(f):=\tau\int_{\bx}f(x)g^{*}(x)\,d\mu(x),\ \ \ for\ any \ S_{\cm}{\rm -}valued\ simple\ function\ f.
\end{align}
\item[\rm{(ii)}] For any $\mathcal{L}\in(\mathcal{H}_1^{\dagger}(\bx,\,\cm))^{*}$, there exists
$g\in\mathcal{BMO}^\dagger(\bx,\,\cm)$ such that $\mathcal{L}=\mathcal{L}_g$.
\end{enumerate}
Moreover, there exists a positive constant $C$ such that
$$C^{-1}\|g\|_{\mathcal{BMO}^\dagger(\bx,\,\cm)}
\leq\|\mathcal{L}_g\|_{(\mathcal{H}_1^{\dagger}(\bx,\,\cm))^{*}}\leq C\|g\|_{\mathcal{BMO}^\dagger(\bx,\,\cm)}.$$

\end{theorem}
We will give the proof of Theorem \ref{t3.1} after showing ${\rm (i)}\Rightarrow {\rm (ii)}$ in Theorem \ref{carl}.

For any set $E\subset\bx$, we denote by $T(E):=\{(y,\,t)\in\bx\times\rr_+:B(y,\,t)\subset E\}$ the tent over $E$. Now we recall the definition of $\cm$-valued Carleson measure.
\begin{definition}
An $\cm$-valued measure $ dm $ on $\bx\times\rr_+$ is called an {\it{$\cm$-valued Carleson measure}}, if
\begin{eqnarray*}
\|dm\|_{{\bf C}}:=\sup_{\mathrm{ball} \, B\subset\bx}\lf\|\frac{1}{\mu(B)}\int_{ T(B)}dm(x,\,t)\r\|_{\cm}<\infty,
\end{eqnarray*}
where the supremum is taken over all balls $B\subset\bx$.
\end{definition}
The $\cm$-valued Carleson measure connects closely to the associated operator-valued BMO spaces. For the space $\mathcal{BMO}^{\dagger}(\bx,\,\cm)$, $\dagger\in\{c,r\}$, one has the following Carleson measure characterization.

\begin{theorem}\label{carl}
The following conditions are equivalent:
\begin{enumerate}
\item[\rm{(i)}]
$g\in\mathcal{BMO}^{c}(\bx,\,\cm)$;
\item[\rm{(ii)}] $dm_g(x,\,t)=
|\mathbf{D}_t(g)(x)|^2\,\frac{d\mu(x)dt}{t}$ is an $\cm$-valued Carleson measure on $\bx\times\rr_+$.
\end{enumerate}
Moreover, there exists a positive constant $C$ such that  $$\frac{1}{C}\|dm_g\|_{ \mathrm{C}}\leq\|g\|_{\mathcal{BMO}^{c}(\bx,\,\cm)}^2\leq{C}\|dm_g\|_{\rm C}.$$
Similar result also holds for the row space.
\end{theorem}
The proof of Theorem \ref{carl} is a combination of Proposition \ref{l2.1} and Proposition \ref{a111}. Now we provide a proof for ${\rm (i)}\Rightarrow {\rm (ii)}$ in Theorem \ref{carl}.
\begin{proposition}\label{l2.1}
For any $g\in\mathcal{BMO}^c(\bx,\,\cm)$, $dm_g:=|\mathbf{D}_t(g)(x)|^2\,\frac{d\mu(x)dt}{t}$ is an $\cm$-valued Carleson measure on $\bx\times\rr_+$. Moreover,  there exists a constant $C>0$ such that
$$\lf\|dm_g\r\|_{\mathrm{C}}\leq C\|g\|_{\mathcal{BMO}^c(\bx,\,\cm)}^2.$$
\end{proposition}
\begin{proof}
Let $g\in\mathcal{BMO}^c(\bx,\,\cm)$ and  $B:=B_d(x_B,\,r_B)\subset\bx$ with $x_B\in\bx,\,r_B>0$. We write
\begin{align*}
g&=(g-g_{4B})\chi_{4B}
+(g-g_{4B})\chi_{(4B)^{\complement}}
+g_{4B}\\
&=:g_1+g_2+g_3.
\end{align*}
By the condition ($\textbf{H}_4$), we obtain that, for any $x\in\bx$ and $t>0$,
$$\mathbf{D}_t(g_3)(x)=\int_{\bx}\mathbf{D}_t(x,\,y)g_3(y)\,d\mu(y)
=g_{4B}\int_{\bx}\mathbf{D}_t(x,\,y)\,d\mu(y)=0.$$
Therefore, we have, for any $t>0$,
$\mathbf{D}_t(g)=\mathbf{D}_t(g_1)+\mathbf{D}_t(g_2).$
Define $$dm_{g_1}:=|\mathbf{D}_t(g_1)|^2\,\frac{d\mu(x)dt}{t} \ \ \mathrm{and} \ \ dm_{g_2}:=|\mathbf{D}_t(g_2)(x)|^2\,\frac{d\mu(x)dt}{t}.$$
Then
$$\lf\|dm_g\r\|_{\mathrm{C}}\lesssim \lf\|dm_{g_1}\r\|_{\mathrm{C}}+\lf\|dm_{g_2}\r\|_{\mathrm{C}}.$$

For the term $\lf\|dm_{g_1}\r\|_{\mathrm{C}}$, by the Fubini theorem and \eqref{e2.50}, we have
\begin{align}\label{rep1}
&\lf\|\frac{1}{\mu(B)}\iint_{T(B)}|\mathbf{D}_t(g_1)(x)|^2\,
\frac{dtd\mu(x)}{t}\r\|_{\cm}\nonumber\\
\leq& \lf\|\frac{1}{\mu(B)}\int_{\bx}\int_0^{\fz}\mathbf{D}_t(g_1^*)(x)
\mathbf{D}_t(g_1)(x)\,\frac{dtd\mu(x)}{t}\r\|_{\cm}\nonumber\\
=& \lf\|\frac{1}{\mu(B)}\int_{\bx}\int_0^{\fz}g_1^*(x)
\mathbf{D}_t^2(g_1)(x)\,\frac{dtd\mu(x)}{t}\r\|_{\cm}\nonumber\\
=& \lf\|\frac{1}{\mu(B)}\int_{\bx}|g_1(x)|^2\,d\mu(x)\r\|_{\cm}\nonumber\\
=& \lf\|\frac{1}{\mu(B)}\int_{4B}|g(x)-g_{4B}|^2\,d\mu(x)\r\|_{\cm}\\
\lesssim&\|g\|_{\mathcal{BMO}^c(\bx,\,\cm)}^2\nonumber.
\end{align}
Therefore,
$$\lf\|dm_{g_1}\r\|_{\mathrm{C}}\lesssim\|g\|_{\mathcal{BMO}^c(\bx,\,\cm)}^2.$$

To estimate $\lf\|dm_{g_2}\r\|_{\mathrm{C}}$, we apply  Lemma
 \ref{e2.5} to deduce that for any
$x\in B_d(x_B,\,r_B)$ and $t>0$,
\begin{align}\label{e2.6}
|\mathbf{D}_t(g_2)(x)|^2&=\lf|\int_{(4B)^{\complement}}
\mathbf{D}_t(x,\,y)g_2(y)\,d\mu(y)\r|^2\nonumber\\
&\leq \int_{(4B)^{\complement}}
\lf|\mathbf{D}_t(x,\,y)\r|^2\mu(B_d(x_B,\,d(y,\,x_B)))d(x,\,y)^{\epsilon_2}\,d\mu(y)\nonumber\\
&\hspace{1.0cm} \times\int_{(4B)^{\complement}}\lf|g_2(y)\r|^2\frac{1}
{\mu(B_d(x_B,\,d(y,\,x_B)))d(x,\,y)^{\epsilon_2}}\,d\mu(y)\nonumber\\
&=:\mathrm{I}_1\times\mathrm{I}_2.
\end{align}
It follows from Lemma \ref{e2.5} and \eqref{e2.1} that for all $j\geq2$,
\begin{align}\label{e2.6x}
\lf|g_{2^{j+1}B}-g_{4B}\r|^2
\leq& j\sum_{k=2}^j\lf|g_{2^{k+1}B}-g_{2^{k}B}\r|^2\nonumber\\
\lesssim&j\sum_{k=2}^j
\fint_{2^{k+1}B}\lf|g(y)-g_{2^{k+1}B}\r|^2\,d\mu(y)\\
\lesssim &j^2\|g\|_{\mathcal{BMO}^c(\bx,\,\cm)}^2.\nonumber
\end{align}
From the above inequality and \eqref{e2.1}, we deduce that
\begin{align}\label{rep2}
\mathrm{I}_2=&\int_{(4B)^{\complement}}
|g(y)-g_{4B}|^2\frac{1}{\mu(B_d(x_B,\,d(y,\,x_B)))d(x,\,y)^{\epsilon_2}}\,d\mu(y)\nonumber\\
\lesssim&\sum_{j=2}^{\infty}\frac{1}{(2^{j}r_B)^{\epsilon_2}}\fint_{2^{j+1}B}
\lf|g(y)-g_{4B}\r|^2\,d\mu(y)\nonumber\\
\lesssim&\sum_{j=2}^{\infty}\frac{1}{(2^{j}r_B)^{\epsilon_2}}
\fint_{2^{j+1}B}
\lf|g(y)-g_{2^{j+1}B}\r|^2+\lf|g_{2^{j+1}B}-g_{4B}\r|^2\,d\mu(y)\\
\lesssim&\sum_{j=2}^{\infty}\frac{1+j^2}{(2^{j}r_B)^{\epsilon_2}}
\|g\|_{\mathcal{BMO}^c(\bx,\,\cm)}^2\nonumber\\
\thicksim& \frac{1}{r_B^{\epsilon_2}}
\|g\|_{\mathcal{BMO}^c(\bx,\,\cm)}^2.\nonumber
\end{align}

Next we estimate the term $\mathrm{I}_1$. From the condition $(\mathbf{H}_1)$, \eqref{e2.1} and Lemma \ref{l2.0}, we deduce that
\begin{align} \label{e2.7}
\mathrm{I}_1=&\int_{(4B)^{\complement}}
\lf|\mathbf{D}_t(x,\,y)\r|^2d(x,\,y)^{\epsilon_2}\mu(B_d(x_B,\,d(y,\,x_B)))\,d\mu(y)\nonumber\\
\lesssim&\sum_{j=2}^{\fz}\int_{U_j(B)}\frac{1}{\lf[V_t(x)+V_t(y)+V(x,\,y)\r]^2}\lf[\frac{t}{t+d(x,\,y)}\r]^{2\epsilon_2}
d(x,\,y)^{\epsilon_2}\mu(B_d(x_B,\,d(y,\,x_B)))\,d\mu(y)\nonumber\\
\lesssim&t^{2\epsilon_2}\sum_{j=2}^{\fz}\lf(\frac{1}{2^j r_B}\r)^{\epsilon_2}\int_{U_j(B)}\frac{1}{\lf[V(x,\,y)\r]^2}
\mu(B_d(x_B,\,d(y,\,x_B)))\,d\mu(y)\nonumber\\
\lesssim&r_B^{\epsilon_2}\lf(\frac{t}{r_B}\r)^{2\epsilon_2}.
\end{align}

From \eqref{e2.6} and the estimates of $\mathrm{I}_1$ and $\mathrm{I}_2$, we conclude that
\begin{align*}
\lf\|dm_{g_2}\r\|_{\mathrm{C}}
\lesssim&\sup_{\mathrm{ball} \,B\subset\bx}\frac{1}{\mu(B)}\int_{B}\int_0^{r_B}
\lf(\frac{t}{r_B}\r)^{2\epsilon_2}
\frac{dtd\mu(x)}{t}\|g\|_{\mathcal{BMO}^c(\bx,\,\cm)}^2
\lesssim\|g\|_{\mathcal{BMO}^c(\bx,\,\cm)}^2.
\end{align*}
Therefore, we complete the proof of Proposition \ref{l2.1}.
\end{proof}
To prove Theorem \ref{t3.1},
we need the following construction given by Christ \cite{c90}, which provides an analogue of the grid of
Euclidean dyadic cubes on spaces of homogeneous type.
\begin{lemma}\label{l3.1}
Let $\bx$ be a space of homogeneous type. Then there exists a collection
$$\mathcal{Q}:=\lf\{Q^k_\alpha\subset\bx:k\in\zz,\,\alpha\in I_k\r\}$$
of open subsets, where $I_k$ is certain index set, and constants $\delta\in(0,\,1)$ and $C_3,\,C_4>0$ such that
\begin{enumerate}
\item[\rm{(i)}] for any fixed $k\in\zz$, $\mu(\bx \backslash \bigcup_\alpha Q^k_\alpha)=0$ and $Q^k_\alpha\cap Q^k_\beta=\emptyset$ for any  $\alpha\neq\beta$;
\item[\rm{(ii)}] for any $\alpha$, $\beta,k$, $\ell$ with $\ell\geq k$, either $Q^k_\alpha\cap Q^\ell_\beta=\emptyset$ or
$Q^\ell_\alpha\subset Q^k_\beta$;
\item[\rm{(iii)}] for each $(\ell,\,\beta)$ and each $k<\ell$, there exists a unique $\alpha$ such that
$Q^\ell_\beta\subset Q^k_\alpha$;
\item[\rm{(iv)}] for any $k\in\zz,\,\alpha\in I_k$, diam$(Q_{\alpha}^k)\leq C_4\delta^k$;
\item[\rm{(v)}] for any
$k\in\zz$ and $\alpha\in I_k$, $Q^k_\alpha$ contains some ball $B(z_{\alpha}^k,\,C_3\delta^k)$,
where $z_{\alpha}^k\in Q^k_\alpha$.
\end{enumerate}
\end{lemma}

Next, we shall apply the one-sided $\cm$-valued Carleson measure characterization theorem (Proposition \ref{l2.1}) to provide a proof of (i) in Theorem \ref{t3.1}.
\begin{proof}[Proof of (i) in Theorem \ref{t3.1}]
By the density argument and approximation, it suffices to show that for any $S_{\cm}$-valued simple function $f$ and $g\in\mathcal{BMO}^c(\bx,\,\cm)$,
\begin{align}\label{duugoa}
\lf|\cl_g(f)\r|\lesssim\|f\|_{\mathcal{H}_1^{c}(\bx,\,\cm)}\|g\|_{\mathcal{BMO}^{c}(\bx,\,\cm)}.
\end{align}
By \eqref{e2.1}, \eqref{e2.50}, the Fubini theorem, Lemma \ref{l2.0} and the Cauchy-Schwarz inequality, we have
\begin{align}\label{e3.1}
\lf|\cl_g(f)\r|^2&=\lf|\tau\int_{\bx}f(x)g^{*}(x)\,d\mu(x)\r|^2\nonumber\\
&=\lf|\tau\int_{\bx}\int_0^{\fz}\mathbf{D}_t\mathbf{D}_t(f)(x)[g(x)]^{*}\,\frac{dtd\mu(x)}{t}\r|^2 \nonumber\\
&=\lf|\tau\int_{\bx}\int_0^{\fz}\mathbf{D}_t(f)(z)\lf[\mathbf{D}_t(g)(z)\r]^{*}
\,\frac{dtd\mu(z)}{t}\r|^2\nonumber\\
&=\lf|\tau\int_{\bx}\int_0^{\fz}\int_{B_d(x,\,\frac{t}{2})}\mathbf{D}_t(f)(z)\lf[\mathbf{D}_t(g)(z)\r]^{*}
\,\frac{d\mu(z)dt}{V_{\frac{t}{2}}(z)t}d\mu(x)\r|^2.
\end{align}
By the Cauchy-Schwartz inequality, the right-hand side above is dominated by
\begin{align}\label{jjj1}
&\leq \tau\int_{\bx}\int_0^{\fz}\lf[S_{\mathbf{D}}^c(f)(x,\,t)\r]^{-1}
\int_{B_d(x,\,\frac{t}{2})}\lf|\mathbf{D}_t(f)(z)\r|^2
\,\frac{d\mu(z)dt}{V_{\frac{t}{2}}(x)t}d\mu(x)\nonumber\\
&\hspace{1.0cm} \times \tau\int_{\bx}\int_0^{\fz}S_{\mathbf{D}}^c(f)(x,\,t)\int_{B_d(x,\,\frac{t}{2})}
\lf|\mathbf{D}_t(g)(z)\r|^2
\,\frac{d\mu(z)dt}{V_{\frac{t}{2}}(x)t}d\mu(x)\nonumber\\
&=:\mathrm{I} \times \mathrm{J},
\end{align}
where the definition of $S_{\mathbf{D}}^c(f)(x,\,t)$ is as follows: for $x\in\bx$ and $t>0$,
\begin{align}\label{ggg1}
S_{\mathbf{D}}^c(f)(x,\,t):=\lf(\int_t^{\infty}\int_{B_{d}
(x,\,s-\frac{t}{2})}|\mathbf{D}_s(f)(z)|^2\,\frac{d\mu(z)ds}{V_{\frac{s}{2}}(x)s}\r)^{1/2}.
\end{align}

For the term $\mathrm{I}$, we introduce another auxiliary square function $\widetilde{S}_{\mathbf{D}}^c(f)$ defined by
\begin{align}\label{ggg2}
\widetilde{S}_{\mathbf{D}}^c(f)(x,\,t):=\lf(\int_t^{\infty}
\int_{B_{d}(x,\,\frac{s}{2})}|\mathbf{D}_s(f)(z)|^2\,\frac{d\mu(z)ds}
{V_{\frac{s}{2}}(x)s}\r)^{1/2}\ {\rm for}\ {\rm any}\ (x,\,t)\in\bx\times\mathbb{R}_+.
\end{align}
It can be verified that these auxiliary square functions satisfy the following property:
\begin{enumerate}
\item[\rm{(i)}] $S_{\mathbf{D}}^c(f)(x,\,t)$ and $\widetilde{S}_{\mathbf{D}}^c(f)(x,\,t)$ are decreasing in $t$;
 \item[\rm{(ii)}] for any $(x,\,t)\in\bx \times\rr_+$,
$\widetilde{S}_{\mathbf{D}}^c(f)(x,\,t)\leq C_n\cs^c(f)(x)$ and
$\widetilde{S}_{\mathbf{D}}^c(f)(x,\,t)\leq S_{\mathbf{D}}^c(f)(x,\,t)$;
\item[\rm{(iii)}] $S_{\mathbf{D}}^c(f)(x,\,t)$ converges to $0$ in the weak-$*$ topology as $t\rightarrow+\infty$.
\end{enumerate}
By approximation, we can assume that $\tau$ is finite and
$S_{\mathbf{D}}^c(f)(x,\,t)$, $\widetilde{S}^c_{\mathbf{D}}(f)(x,\,t)$ are invertible for any $(x,\,t)\in\bx \times\rr_+$. It follows from the fact (ii) that $S_{\mathbf{D}}^{-1}(f)(x,\,t)\leq\widetilde{S}_{\mathbf{D}}^{-1}(f)(x,\,t)$ for any $(x,\,t)\in\bx \times\rr_+$. This, in combination with the fact (ii), yields that
\begin{align}\label{jjj2}
\mathrm{I}&\leq\tau\int_{\bx}\int_0^{\fz}\lf[\widetilde{S}_{\mathbf{D}}^{c}
(f)(x,\,t)\r]^{-1}\int_{B_{d}(x,\,\frac{t}{2})}|\mathbf{D}_t(f)(z)|^2\,
\frac{d\mu(z)dt}{V_{\frac{t}{2}}(x)t}d\mu(x)\nonumber\\
&=-\tau\int_{\bx}\int_0^{\fz}\lf[\widetilde{S}_{\mathbf{D}}^{c}
(f)(x,\,t)\r]^{-1}\frac{\partial}{\partial t}\lf[\widetilde{S}_{\mathbf{D}}^{c}
(f)(x,\,t)\r]^2\,dtd\mu(x)\nonumber\\
&\thicksim-\tau\int_{\bx}\int_0^{\fz}\frac{\partial}{\partial t}\widetilde{S}_{\mathbf{D}}^{c}
(f)(x,\,t)\,dtd\mu(x)\nonumber\\
&\thicksim\tau\int_{\bx}\widetilde{S}_{\mathbf{D}}^{c}
(f)(x,\,0)\,d\mu(x)\nonumber\\
&\lesssim\tau\int_{\bx}\cs^c(f)(x)\,d\mu(x)\nonumber\\
&\thicksim\|f\|_{\mathcal{H}_1^{c}(\bx,\,\cm)}.
\end{align}

Now we deal with the term $\mathrm{J}$. We have
\begin{align}\label{e3.3}
\mathrm{J}&=\tau\int_{\bx}\int_0^{\fz}S_{\mathbf{D}}^c(f)(x,\,t)
\int_{B_d(x,\,\frac{t}{2})}\lf|\mathbf{D}_t(g)(z)\r|^2
\,\frac{d\mu(z)}{V_{\frac{t}{2}}(x)t}dtd\mu(x)\\
&=\tau\sum_{k\in\zz}\sum_{\alpha\in I_k}\int_{Q_{\alpha}^k}\int_{2C_4c_\delta\delta^k}
^{2C_4c_\delta\delta^{k-1}}S_{\mathbf{D}}^c(f)(x,\,t)\int_{B_d(x,\,\frac{t}{2})}
\lf|\mathbf{D}_t(g)(z)\r|^2
\,\frac{d\mu(z)}{V_{\frac{t}{2}}(x)t}dtd\mu(x),\nonumber
\end{align}
where $Q_{\alpha}^k\in\mathcal{Q}$, $I_k$, $\mathcal{Q}$, $C_4$, $\delta$ are given in Lemma \ref{l3.1} and $c_\delta:=1+(1-\delta)^{-1}$.

To estimate \eqref{e3.3}, we define
\begin{align}\label{mathS}
\mathbb{S}_{\mathbf{D}}^c(f)(x,\,k):=\lf(\int_{2C_4c_\delta\delta^k}^{\fz}
\int_{B_{d}(z_{\alpha}^k,\,s-\frac{C_4}{1-\delta}\delta^k)}
|\mathbf{D}_t(f)(z)|^2\,\frac{d\mu(z)ds}{V_{\frac{s}{2}}(z)s}\r)^{1/2},\ {\rm for}\ {\rm any}\ (x,\,k)\in Q_\alpha^k\times \mathbb{Z},
\end{align}
where $z_\alpha^k$ is given in Lemma \ref{l3.1}. Note that
\begin{align*}
B_{d}(x,\,s-\frac{t}{2})\subset B_{d}(z_{\alpha}^k,\,s-\frac{C_4}{1-\delta}\delta^k),\
\mathrm{for}\ \mathrm{any}\ x\in Q_{\alpha}^k \ \mathrm{and}\ s\geq t\geq 2C_4c_\delta\delta^k.
\end{align*}
This implies that
\begin{align*}
S_{\mathbf{D}}^c(f)(x,\,t)\leq \mathbb{S}_{\mathbf{D}}^c(f)(x,\,k),\ \ \
\mathrm{for}\ \mathrm{any}\ x\in Q_{\alpha}^k \ \mathrm{and}\ t\geq 2C_4c_\delta\delta^k.
\end{align*}
This, together with \eqref{e3.3}, implies that
\begin{align}\label{e3.5}
\mathrm{J}&\leq\tau\sum_{k\in\zz}\sum_{\alpha\in I_k}\int_{Q_{\alpha}^k}\int_{2C_4c_\delta\delta^k}
^{2C_4c_\delta\delta^{k-1}}\mathbb{S}_{\mathbf{D}}^c(f)(x,\,k)\int_{B_d(x,\,\frac{t}{2})}
\lf|\mathbf{D}_t(g)(z)\r|^2
\,\frac{d\mu(z)}{V_{\frac{t}{2}}(x)t}dtd\mu(x)\nonumber\\
   &=\tau\int_{\bx}\sum_{k\in\zz}\mathbb{S}_{\mathbf{D}}^c(f)(x,\,k)\int_{2 C_4c_\delta\delta^k}
^{2C_4c_\delta\delta^{k-1}}\int_{B_d(x,\,\frac{t}{2})}\lf|\mathbf{D}_t(g)(z)\r|^2
\,\frac{d\mu(z)}{V_{\frac{t}{2}}(x)t}dtd\mu(x)\\
&=\tau\int_{\bx}\sum_{k\in\zz}\sum_{j\leq k}\mathbb{W}(x,\,j)\int_{2C_4c_\delta \delta^k}
^{2C_4c_\delta\delta^{k-1}}\int_{B_d(x,\,\frac{t}{2})}\lf|\mathbf{D}_t(g)(z)\r|^2
\,\frac{d\mu(z)}{V_{\frac{t}{2}}(x)t}dtd\mu(x)\nonumber\\
&=\tau\int_{\bx}\sum_{j\in\zz}\mathbb{W}(x,\,j)\sum_{k\geq j}\int_{2 C_4c_\delta\delta^k}
^{2C_4c_\delta\delta^{k-1}}\int_{B_d(x,\,\frac{t}{2})}\lf|\mathbf{D}_t(g)(z)\r|^2
\,\frac{d\mu(z)}{V_{\frac{t}{2}}(x)t}dtd\mu(x), \nonumber
\end{align}
where $\mathbb{W}(x,\,j):=\mathbb{S}_{\mathbf{D}}^c(f)(x,\,j)-\mathbb{S}_{\mathbf{D}}^c(f)(x,\,j-1)$ for any $x\in\bx$ and $j\in\zz$. Now we claim that for any $x\in\bx$ and $j\in\mathbb{Z}$, one has $\mathbb{W}(x,\,j)\geq 0$. To show this, for any $x\in Q_\alpha^j\subset Q_\beta^{j-1}$ and $y\in B(z_\beta^{j-1},s-\frac{C_4}{1-\delta}\delta^{j-1})$, one has
\begin{align*}
d(y,\,z_\alpha^j)\leq d(y,\,z_\beta^{j-1})+d(z_\alpha^j,\,z_\beta^{j-1})\leq s-\frac{C_4}{1-\delta}\delta^{j-1}+C_4\delta^{j-1}=s-\frac{C_4}{1-\delta}\delta^j.
\end{align*}
This indicates that $B(z_\beta^{j-1},\,s-\frac{C_4}{1-\delta}\delta^{j-1})\subset B(z_\alpha^j,\,s-\frac{C_4}{1-\delta}\delta^j)$, which finishes the proof of the claim.

By the above claim and the fact that $\mathbb{W}(x,\,j)$ is a constant on
$Q_{\alpha}^j$ for any $\alpha\in I_j$, together with Proposition \ref{l2.1}, we obtain that
\begin{align}\label{rept1}
\mathrm{J}&\leq\tau\sum_{j\in\zz}\sum_{\alpha\in I_k}\mathbb{W}(z_{\alpha}^j,\,j)
\int_{Q_{\alpha}^j}\int_{0}
^{2C_4c_\delta\delta^{j-1}}\int_{B_d(x,\,\frac{t}{2})}\lf|\mathbf{D}_t(g)(z)\r|^2
\,\frac{d\mu(z)}{V_{\frac{t}{2}}(x)t}dtd\mu(x) \nonumber\\
&\lesssim\tau\sum_{j\in\zz}\sum_{\alpha\in I_k}\mathbb{W}(z_{\alpha}^j,\,j)
\int_{0}^{2C_4c_\delta\delta^{j-1}}\int_{B_d(z_{\alpha}^j,\,2C_4c_\delta\delta^{j-1})}\lf|\mathbf{D}_t(g)(z)\r|^2
\,\frac{d\mu(z)dt}{t} \\ \nonumber
&\lesssim\tau\sum_{j\in\zz}\sum_{\alpha\in I_k}\mathbb{W}(z_{\alpha}^j,\,j)
\mu(B_d(z_{\alpha}^j,\,2C_4c_\delta\delta^{j-1}))\,\|g\|_{\mathcal{BMO}^{c}(\bx,\,\cm)}^2.\nonumber
\end{align}
From this, \eqref{e2.1}, Lemma \ref{l3.1} and the fact that $\mathbb{W}(x,\,j)$ is a constant on $Q_{\alpha}^j$, we deduce that
\begin{align*}
\mathrm{J}
&\lesssim\|g\|_{\mathcal{BMO}^{c}(\bx,\,\cm)}^2
\tau\sum_{j\in\zz}\sum_{\alpha\in I_j}\mathbb{W}(z_{\alpha}^j,\,j)
\mu(B_d(z_{\alpha}^j,\,C_3\delta^{j}))\nonumber\\
&\lesssim\|g\|_{\mathcal{BMO}^{c}(\bx,\,\cm)}^2
\tau\sum_{j\in\zz}\sum_{\alpha\in I_j}\int_{Q_{\alpha}^j}\mathbb{W}(x,\,j)\,d\mu(x)\nonumber\\
&\thicksim\|g\|_{\mathcal{BMO}^{c}(\bx,\,\cm)}^2
\tau\int_{\bx}\sum_{j\in\zz}\mathbb{W}(x,\,j)\,d\mu(x)\nonumber\\
&\thicksim\|g\|_{\mathcal{BMO}^{c}(\bx,\,\cm)}^2
\tau\int_{\bx}\mathbb{S}_{\mathbf{D}}^c(f)(x,\,+\infty)\,d\mu(x)\nonumber\\
&\thicksim\|g\|_{\mathcal{BMO}^{c}(\bx,\,\cm)}^2\tau\int_{\bx}
\cs^c(f)(x)\,d\mu(x)\nonumber\\
&\thicksim\|g\|_{\mathcal{BMO}^{c}(\bx,\,\cm)}^2\|f\|_{\mathcal{H}_1^{c}(\bx,\,\cm)}.
\end{align*}
Combining the estimates of $\mathrm{I}$ and $\mathrm{J}$ completes the proof of  \eqref{duugoa} and then Theorem \ref{t3.1} (i).
\end{proof}

The following Lemma provides a connection between $\mathcal{H}^{c}_1(\bx,\,\cm)$ and $L_1(\lm;\,L_2^c(\bx\times \mathbb{R}_+,\,d\mu(y)dt))$ via two maps $\Phi$ and $\Psi$ introduces as follows.
\begin{lemma}\label{d3.1}
(1) The operator
$$\Phi(f)(x,\,y,\,t):=\frac{\mathbf{D}_t(f)(y)}{[V_t(x)t]^{1/2}}\chi_{\Gamma_x}(y,\,t),$$
initially defined on the set of $S_{\cm}$-valued  simple function, extends to an isometry operator from $\mathcal{H}_p^{c}(\bx,\,\cm)$ to $L_{p}(\lm;$ $\,L^c_2(\bx\times \mathbb{R}_+,\,d\mu(y)dt))$ for any $1\leq p<\infty$.

(2) The operator
$$\Psi(g)(z):=\int_{\bx}\iint_{\Gamma_x}g(x,\,y,\,t)
\overline{\mathbf{D}_t(y,\,z)}\,\frac{d\mu(y)dt}{[V_t(x)t]^{1/2}}d\mu(x),$$
initially defined on the set of $S_{\lm}$-valued  simple function, extends to a bounded operator from  $ L_{2}(\lm;\,L_2^c(\bx\times \mathbb{R}_+,\,d\mu(y)dt))$ to $\mathcal{H}^{c}_2(\bx,\,\cm)$.

(3) Let $f\in\mathcal{H}^{c}_2(\bx,\,\cm)$. Then we have
$$\Psi\Phi(f)=f.$$
\end{lemma}
\begin{proof}
It is a direct consequence of the definitions of these two maps.
\end{proof}

\begin{lemma}\label{l3.2}
Let $\Psi$ be the map given in Lemma \ref{d3.1}. Then $\Psi$ is bounded from $L_{\infty}(\lm;\,$ $L_2^c(\bx\times \mathbb{R}_+,\,d\mu(y)dt))$ to $\mathcal{BMO}^c(\bx,\,\cm)$. Similar result also holds for the row space.
\end{lemma}
\begin{proof}
By a density argument, it suffices to show that for any $S_{\lm}$-valued simple function $g$,
$$\lf\|\Psi(g)\r\|_{\mathcal{BMO}^c(\bx,\,\cm)}\lesssim\|g\|_
{L_{\infty}(\lm;\,L_2^c(\bx\times \mathbb{R}_+,\,d\mu(y)dt))}.$$
Let $B:=B_d(x_B,\,r_B)\subset\bx$ with $x_B\in\bx$ and $r_B\in\rr_+$.
Then we write
\begin{align*}
g(x,\,y,\,t)&=g(x,\,y,\,t)\chi_{4B}(x)+g(x,\,y,\,t)
\chi_{(4B)^{\complement}}(x)\\
&=:g_1(x,\,y,\,t)+g_2(x,\,y,\,t).
\end{align*}
For simplicity, we define
$$C_{B}:=\int_{\bx}\iint_{\Gamma_x}
g_2(x,\,y,\,t)\fint_{B}\overline{\mathbf{D}_t(y,\,z)}\,d\mu(z)
\frac{d\mu(y)dt}{\lf[V_t(x)t\r]^{1/2}}d\mu(x).$$
Then
\begin{align}\label{iopiopiop}
&\fint_{B}\lf|\Psi(g)(z)-C_{B}\r|^2\,d\mu(z)\nonumber\\
\leq&2\fint_{B}\lf|
\int_{(4B)
^{\complement}}
\iint_{\Gamma_x}g_2(x,\,y,\,t)
\lf[\mathbf{D}_t(z,\,y)-Q_{B}(y)\r]\,\frac{d\mu(y)dt}{\lf[V_t(x)t\r]^{1/2}}d\mu(x)\r|^2\,d\mu(z)\nonumber\\
&\ \ +2\fint_{B}\lf|\int_{\bx}\iint_{\Gamma_x}g_1(x,\,y,\,t)
\overline{\mathbf{D}_t(y,\,z)}\,\frac{d\mu(y)dt}{\lf[V_t(x)t\r]^{1/2}}d\mu(x)\r|^2\,d\mu(z)\nonumber\\
=&:\mathrm{II}_1+\mathrm{II}_2,
\end{align}
where $Q_{B}(y):=\fint_{B}\mathbf{D}_t(z,\,y)\,d\mu(z)$ for $y\in\bx$.

To estimate $\mathrm{II}_1$, we apply Lemma \ref{e2.5} to deduce that for any $z\in B$,
\begin{align}\label{eeeee000}
&\lf|\int_{(4B)
^{\complement}}
\iint_{\Gamma_x}g_2(x,\,y,\,x)
\lf[\overline{\mathbf{D}_t(y,\,z)}-Q_{B}(y)\r]\,\frac{d\mu(y)dt}
{\lf[V_t(x)t\r]^{1/2}}d\mu(x)\r|^2\nonumber\\
\leq&\int_{(4B)^{\complement}}\frac{1}{V(x,\,z)[d(x,\,z)]^{\epsilon_1}}\,d\mu(x)\int_{(4B)^{\complement}}V(x,\,z)[d(x,\,z)]^{\epsilon_1}\nonumber\\
&\hspace{1.0cm} \times
 \iint_{\Gamma_x}
\lf|\overline{\mathbf{D}_t(y,\,z)}-Q_{B}(y)\r|^2\,\frac{d\mu(y)dt}
{V_t(x)t}\iint_{\Gamma_x}|g(x,\,y,\,t)|^2\,d\mu(y)dtd\mu(x)\nonumber\\
\lesssim&r_B^{-\epsilon_1}\int_{(4B)^{\complement}}V(x,\,z)[d(x,\,z)]^{\epsilon_1}
\fint_B\iint_{\Gamma_x}
\lf|\overline{\mathbf{D}_t(y,\,z)}-\overline{\mathbf{D}_t(y,\,s)}\r|^2
\,\frac{d\mu(y)dt}{V_t(x)t}d\mu(s)\nonumber\\
&\hspace{1.0cm} \times\iint_{\Gamma_x}|g(x,\,y,\,t)|^2\,d\mu(y)dtd\mu(x)\\
\lesssim&r_B^{-\epsilon_1}\|g\|_{L_{\infty}(\lm;\,L_2^c(\bx\times \mathbb{R}_+,\,d\mu(y)dt)}^2\int_{(4B)^{\complement}}V(x,\,z)[d(x,\,z)]^{\epsilon_1}\nonumber\\
&\hspace{1.0cm} \times
\fint_B\iint_{\Gamma_x}
\lf|\overline{\mathbf{D}_t(y,\,z)}-\overline{\mathbf{D}_t(y,\,s)}\r|^2
\,\frac{d\mu(y)dt}{V_t(x)t}d\mu(s)d\mu(x).\nonumber
\end{align}
To continue, it remains to show that for any $z\in B$,
\begin{align}\label{3.1a}
\int_{(4B)^{\complement}}V(x,\,z)[d(x,\,z)]^{\epsilon_1}
\fint_B\iint_{\Gamma_x}
\lf|\mathbf{D}_t(y,\,z)-\mathbf{D}_t(y,\,s)\r|^2
\,\frac{d\mu(y)dt}{V_t(x)t}d\mu(s)d\mu(x)
\lesssim r_B^{\epsilon_1}.
\end{align}
To show this, it follows from \eqref{e2.1} and the condition $(\mathbf{H}_2)$ that for any $z\in B$,
\begin{align*}
&\int_{(4B)^{\complement}}V(x,\,z)[d(x,\,z)]^{\epsilon_1}
\fint_B\iint_{\Gamma_x}
\lf|\mathbf{D}_t(y,\,z)-\mathbf{D}_t(y,\,s)\r|^2
\,\frac{d\mu(y)dt}{V_t(x)t}d\mu(s)d\mu(x)\\
\lesssim&\sum_{k=2}^{+\fz}\int_{U_k(B)}V(x,\,z)[d(x,\,z)]^{\epsilon_1}\\
&\hspace{1.0cm} \times\iint_{\Gamma_x}
\lf[\frac{r_B}{t+d(y,\,z)}\r]^{2\epsilon_1}
\frac{1}{\lf[V_t(z)+V_t(y)+V(z,\,y)\r]^2}
\lf[\frac{t}{t+d(y,\,z)}\r]^{2\epsilon_2}
\,\frac{d\mu(y)dt}{V_t(x)t}d\mu(x)\\
\lesssim&\sum_{k=2}^{+\fz}\int_{U_k(B)}V(x,\,z)[d(x,\,z)]^{\epsilon_1}
\lf(\int_0^{\frac{d(x,\,z)}{4}}+\int_{\frac{d(x,\,z)}{4}}^{\fz}\r)\\
&\hspace{1.0cm} \times\int_{B_d(x,\,t)}
\lf[\frac{r_B}{t+d(z,\,y)}\r]^{2\epsilon_1}
\frac{1}{\lf[V_t(z)+V_t(y)+V(z,\,y)\r]^2}
\lf[\frac{t}{t+d(z,\,y)}\r]^{2\epsilon_2}
\,\frac{d\mu(y)dt}{V_t(x)t}d\mu(x)\\
=&:\mathrm{II}_{11}+\mathrm{II}_{12}.
\end{align*}

 For the term $\mathrm{II}_{11}$, we first note that if $0<t<\frac{d(x,\,z)}{4}$ and $d(x,\,y)<t$, then
$$d(x,\,z)\leq d(x,\,y)+d(y,\,z)<t+d(y,\,z)\leq
\frac{d(x,\,z)}{4}+d(y,\,z).$$
Thus, $\frac{3d(x,\,z)}{4}\leq d(y,\,z)$, which implies that
$$V(y,\,z)\thicksim V(z,\,y)\gtrsim\mu(B_d(z,\,\frac{3d(x,\,z)}{4}))\gtrsim
V(z,\,x).$$
By the above facts, we obtain
\begin{align}\label{ana1}
\mathrm{II}_{11}
&\lesssim\sum_{k=2}^{+\fz}\int_{U_k(B)}V(x,\,z)[d(x,\,z)]^{\epsilon_1}
\lf[\frac{r_B}{d(x,\,z)}\r]^{2\epsilon_1}\nonumber\\
&\hspace{1.0cm} \times\int_0^{\frac{d(x,\,z)}{4}}\int_{B_d(x,\,t)}
\frac{1}{\lf[V(z,\,y)\r]^2}
\lf[\frac{t}{d(x,\,z)}\r]^{2\epsilon_2}
\,\frac{d\mu(y)dt}{V_t(x)t}d\mu(x)\nonumber\\
&\lesssim\sum_{k=2}^{+\fz}\int_{U_k(B)}V(x,\,z)[d(x,\,z)]^{\epsilon_1}
\lf[\frac{r_B}{d(x,\,z)}\r]^{2\epsilon_1}
\int_0^{\frac{d(x,\,z)}{4}}\frac{1}{\lf[V(z,\,x)\r]^2}
\lf[\frac{t}{d(x,\,z)}\r]^{2\epsilon_2}
\,\frac{dt}{t}d\mu(x)\nonumber\\
&\lesssim r_B^{\epsilon_1}\sum_{k=2}^{+\fz}2^{-k\epsilon_1}\int_{U_k(B)}
\frac{1}{V(x,\,z)}\,d\mu(x)\\
&\lesssim r_B^{\epsilon_1}.\nonumber
\end{align}

For the term $\mathrm{II}_{12}$,
we deduce that
\begin{align}\label{ana2}
\mathrm{II}_{12}
&\lesssim\sum_{k=2}^{+\fz}
\int_{U_k(B)}V(x,\,z)[d(x,\,z)]^{\epsilon_1}
\int_{\frac{d(x,\,z)}{4}}^{\fz}
\frac{1}{\lf[V_t(z)\r]^2}\lf(\frac{r_B}{t}\r)^{2\epsilon_1}
\,\frac{dt}{t}d\mu(x)\nonumber\\
&\lesssim\sum_{k=2}^{+\fz}2^{-2k\epsilon_1}(2^kr_B)
^{\epsilon_1}\int_{U_k(B)}
V(x,\,z)\lf[\frac{1}{V_{\frac{d(x,\,z)}{4}}(z)}\r]^2
\int_{\frac{d(x,\,z)}{4}}^{\fz}
\lf(\frac{2^kr_B}{t}\r)^{2\epsilon_1}
\,\frac{dt}{t}d\mu(x)\nonumber\\
&\lesssim r_B^{\epsilon_1}\sum_{k=2}^{+\fz}2^{-k\epsilon_1}\int_{U_k(B)}
\frac{1}{V(x,\,z)}\,d\mu(x)\\
&\lesssim r_B^{\epsilon_1}.\nonumber
\end{align}
This together with the estimate of $\mathrm{II}_{11}$ implies \eqref{3.1a}.
Therefore,
$$
\mathrm{II}_{1}\lesssim\|g\|
_{L_{\infty}(\lm;\,L_2^c(\bx\times \mathbb{R}_+,\,d\mu(y)dt)}^2.
$$

To deal with $\mathrm{II}_2$, by the duality theory, the Cauchy-Schwarz inequality,
Lemma \ref{l2.0}, \eqref{e2.1} and the $L_2(\lm)$-boundedness of $\Psi$, we have
\begin{align*}
\mathrm{II}_2
\lesssim&\lf\|\frac{1}{\mu(B)}\int_{\bx}\lf|\int_{\bx}\iint_{\Gamma_x}
g_1(x,\,y,\,t)
\overline{\mathbf{D}_t(y,\,z)}\,\frac{d\mu(y)dt}{\lf[V_t(x)t\r]^{1/2}}d\mu(x)\r|^2\,d\mu(z)\r\|_{\cm}\\
\lesssim&\frac{1}{\mu(B)}\sup_{\|a\|_{L_1(\cm)}\leq1}
\tau\lf(|a|\int_{\bx}\lf|\int_{\bx}\iint_{\Gamma_x}g_1(x,\,y,\,t)
\overline{\mathbf{D}_t(y,\,z)}\,\frac{d\mu(y)dt}{\lf[V_t(x)t\r]^{1/2}}d\mu(x)\r|^2\,d\mu(z)\r)\\
\thicksim&\frac{1}{\mu(B)}\sup_{\|a\|_{L_1(\cm)}\leq1}
\tau\lf(\int_{\bx}\lf|\int_{\bx}\iint_{\Gamma_x}g_1(x,\,y,\,t)
\overline{\mathbf{D}_t(y,\,z)}|a|^{{1}/{2}}\,\frac{d\mu(y)dt}{\lf[V_t(x)t\r]^{1/2}}
d\mu(x)\r|^2\,d\mu(z)\r)\\
\thicksim&\frac{1}{\mu(B)}\sup_{\|a\|_{L_1(\cm)}\leq1}
\sup_{\|h\|_{L_2(\lm)}\leq1}
\lf|\tau\int_{\bx}h(z)\r.\\
&\hspace{1.0cm} \times\lf.\int_{\bx}\iint_{\Gamma_x}g_1(x,\,y,\,t)
\overline{\mathbf{D}_t(y,\,z)}|a|^{{1}/{2}}\,\frac{d\mu(y)dt}{\lf[V_t(x)t\r]^{1/2}}d\mu(x)\,d\mu(z)\r|^2\\
\lesssim&\frac{1}{\mu(B)}\sup_{\|a\|_{L_1(\cm)}\leq1}\sup_{\|h\|_{L_2(\lm)}\leq1}
\tau\int_{\bx}\iint_{\Gamma_x}\lf|\mathbf{D}_t(h^*)(y)\r|^2
\,\frac{d\mu(y)dt}{V_t(x)t}d\mu(x)\\
&\hspace{1.0cm}\times\tau\lf(|a|\int_{\bx}\iint_{\Gamma_x}\lf|g_1(x,\,y,\,t)\r|^2\,d\mu(y)dtd\mu(x)\r)\\
\lesssim&
\lf\|\fint_{4B}\iint_{\Gamma_x}\lf|g_1(x,\,y,\,t)\r|^2\,d\mu(y)d\mu(x)\r\|_{\cm}\\
\lesssim&\|g\|_{L_{\infty}(\lm;\,L_2^c(\bx\times \mathbb{R}_+,\,d\mu(y)dt))}^2.
\end{align*}
Combining the estimates of $\mathrm{II}_1$ and $\mathrm{II}_2$, we complete the proof of Lemma \ref{l3.2}.
\end{proof}

\begin{proof}[Proof of (ii) in Theorem \ref{t3.1}] Let $\cl\in(\mathcal{H}_1^{c}(\bx,\,\cm))^{\ast}$.
It suffices to show that there exists a $g\in\mathcal{BMO}^{c}(\bx,\,\cm)$ such that
\begin{align*}
\cl(f)=\cl_g(f):=\tau\int_{\bx}f(x)g^{\ast}(x)\,d\mu(x).
\end{align*}
Recall that $\mathcal{H}_1^{c}(\bx,\,\cm)$ can be embedded into $(L_1(\lm;\,L_2^c(\bx\times \mathbb{R}_+,\,d\mu(y)dt)))$ via the embedding map $\Phi$ defined in Lemma \ref{d3.1}. This, together with Hahn-Banach theorem, indicates that $\cl$ can be extended to a continuous functional on $L_1(\lm;\,L_2^c(\bx\times \mathbb{R}_+,\,d\mu(y)dt))$ with the same norm.
Moreover, since
$$(L_1(\lm;\,L_2^c(\bx\times \mathbb{R}_+,\,d\mu(y)dt)))^{\ast}
=L_{\infty}(\lm;\,L_2^c(\bx\times \mathbb{R}_+,\,d\mu(y)dt)),$$
there exists $h\in L_{\infty}(\lm;\,L_2^c(\bx\times \mathbb{R}_+,\,d\mu(y)dt))$ such that
\begin{align*}
\|h\|_{L_{\infty}(\lm;\,L_2^c(\bx\times \mathbb{R}_+,\,d\mu(y)dt))}
=\|\cl\|_{(\mathcal{H}_1^{c}(\bx,\,\cm))^{\ast}}
\end{align*}
and that for $f\in\mathcal{H}_1^{c}(\bx,\,\cm)$,
\begin{align}\label{euqa}
\cl(f)&=\tau\int_{\bx}\iint_{\Gamma_x}h^{\ast}(x,\,y,\,t)
\frac{\mathbf{D}_t(f)(y)}{[V_t(x)t]^{1/2}}\,d\mu(y)dtd\mu(x)\nonumber\\
&=\tau\int_{\bx}(\Psi(h))^{\ast}(z)f(z)\,d\mu(z)=\tau\int_{\bx}g^{\ast}(z)f(z)\,d\mu(z),
\end{align}
where $g:=\Psi(h)$ and $\Psi$ is as in Lemma \ref{d3.1}. Then, it follows from Lemma \ref{l3.2} that
\begin{align*}
\|g\|_{\mathcal{BMO}^{c}(\bx,\,\cm)}
=\|\Psi(h)\|_{\mathcal{BMO}^{c}(\bx,\,\cm)}\lesssim
\|h\|_{L_{\infty}(\lm;\,L_2^c(\bx\times \mathbb{R}_+,\,d\mu(y)dt))}.
\end{align*}
This completes the proof of Theorem \ref{t3.1} (ii) and hence of Theorem \ref{t3.1}.
\end{proof}
At the end of this section, we provide a proof for ${\rm (ii)}\Rightarrow {\rm (i)}$ in Theorem \ref{carl}.
\begin{proposition}\label{a111}
Let $dm_g:=|\mathbf{D}_t(g)(x)|^2\,\frac{d\mu(x)dt}{t}$ be an $\cm$-valued Carleson measure on $\bx\times\rr_+$. Then $g\in \mathcal{BMO}^c(\bx,\,\cm)$. Moreover,  there exists a constant $C>0$ such that
$$\|g\|_{\mathcal{BMO}^c(\bx,\,\cm)}^2\leq C\lf\|dm_g\r\|_{\mathrm{C}}.$$
\end{proposition}
\begin{proof}
By repeating the proof of Theorem \ref{t3.1}, we know from inequalities \eqref{jjj1}, \eqref{jjj2} and \eqref{rept1} that $g$ induces a linear functional $\mathcal{L}_g$ with the form \eqref{dua}, initially defined on the set of all $S_{\cm}$-valued simple functions, satisfying
\begin{align}\label{stidu}
\lf|\cl_g(f)\r|^2&\leq C(\mathrm{I} \times \mathrm{J}),
\end{align}
where $\mathrm{I}\lesssim \|f\|_{\mathcal{H}_1^{c}(\bx,\,\cm)}$ and
\begin{align*}
 \mathrm{ J}&\lesssim\tau\sum_{j\in\zz}\sum_{\alpha\in I_j}\mathbb{W}(z_{\alpha}^j,\,j)
\int_{0}^{2C_4c_\delta\delta^{j-1}}\int_{B_d(z_{\alpha}^j,\,2C_4c_\delta\delta^{j-1})}\lf|\mathbf{D}_t(g)(z)\r|^2
\,\frac{d\mu(z)dt}{t}\\
&\lesssim \|dm_g\|_{{\bf C}}\ \tau\sum_{j\in\zz}\sum_{\alpha\in I_j}\mathbb{W}(z_{\alpha}^j,\,j)
\mu(B_d(z_{\alpha}^j,\,2C_4c_\delta\delta^{j-1}))\,\nonumber\\
&\lesssim \|dm_g\|_{{\bf C}}\|f\|_{\mathcal{H}_1^{c}(\bx,\,\cm)}.
 \end{align*}
Substituting the estimates of $\mathrm{I}$ and $\mathrm{J}$ into inequality \eqref{stidu}, we conclude that
$$|\cl_g(f)|^2\lesssim \|dm_g\|_{{\bf C}}\|f\|_{\mathcal{H}_1^{c}(\bx,\,\cm)}^2.$$
Therefore, it follows from Theorem \ref{t3.1} that there exists $\widetilde{g}\in\mathcal{BMO}^{c}(\bx,\,\cm)$ such that
\begin{align*}
\|\widetilde{g}\|_{\mathcal{BMO}^{c}(\bx,\,\cm)}^2=\lf\|\cl_g\r\|^2\lesssim \|dm_g\|_{\rm C}
\end{align*}
and
\begin{align*}
\cl_g(f)=\tau\int_{\bx}f(x)g^{*}(x)\,d\mu(x)=\tau\int_{\bx}f(x)\widetilde{g}^{*}(x)\,d\mu(x),
\end{align*}
which implies that $g=\widetilde{g}$, $g\in\mathcal{BMO}^{c}(\bx,\,\cm)$ and
$\|g\|_{\mathcal{BMO}^{c}(\bx,\,\cm)}^2\lesssim \|dm_g\|_{\rm C}$.
\end{proof}

\bigskip
\section{Atomic characterization of $H_1$}\label{s31}
\hskip\parindent
In this section, as an application of Theorem \ref{t3.1}, we establish
the atomic characterization of operator-valued Hardy spaces $\mathcal{H}_1^{\dagger}(\bx,\,\cm)$, where $\dagger\in \{c,r\}$.
\begin{definition}
Let $\dagger\in\{c,r\}$. A  function  $a\in L_1(\cm;\,L^\dagger_2(\bx))$ is called an  {\it $\cm^\dagger$-atom} if there is a ball $B\subset \bx$ such that
\begin{enumerate}
\item[\rm{(i)}]  $\supp a:=\overline{\{x\in\bx:a(x)\neq0\}}\subset B$;
\item[\rm{(ii)}]  $\|a\|_{L_1(\cm;\,L_2^\dagger(\bx))}\leq \mu(B)^{-1/2}$;
\item[\rm{(iii)}]  $\int_\bx a(x)\, d\mu(x)=0$.
\end{enumerate}
\end{definition}
\begin{lemma}\label{l3.21}
Let $\dagger\in\{c,r\}$ and $a$ be an $\cm^\dagger$-atom. Then $a\in\mathcal{H}_{1}^{\dagger}(\bx,\,\cm)$.
\end{lemma}
\begin{proof}
Let $a$ be an $\cm^c$-atom supported in $B:={B_d(x_B,\,r_B)}$ with $x_B\in\bx$ and $r_B>0$. Then
\begin{align*}
\|a\|_{\mathcal{H}_{1}^{c}(\bx,\,\cm)}
&=\tau\int_{4B}\lf(\iint_{\Gamma_x}|\mathbf{D}_t(a)(y)|^2\,\frac{d\mu(y)dt}{V_t(x)t}\r)
^{1/2}\,d\mu(x)\\
&\hspace{1.0cm}+\tau\int_{(4B)^{\complement}}\lf(\iint_{\Gamma_x}|\mathbf{D}_t(a)(y)|^2\,\frac{d\mu(y)dt}{V_t(x)t}\r)
^{1/2}\,d\mu(x)\\
&=:\mathrm{K}_1+\mathrm{K}_2.
\end{align*}
From
the Fubini theorem, Lemma \ref{e2.5}, \eqref{e2.1} and the size condition of $a$, we conclude that
\begin{align}\label{keymodify}
\mathrm{K}_1
\leq&[\mu(4B)]^{1/2}\tau\lf(\int_{4B}\iint_{\Gamma_x}
|\mathbf{D}_t(a)(y)|^2\,\frac{d\mu(y)dt}{V_t(x)t}
\,d\mu(x)\r)^{1/2}\nonumber\\
\lesssim&[\mu(4B)]^{1/2}\tau\lf(\int_0^{\fz}\int_{\bx}|\mathbf{D}_t(a)(y)|
^2\,\frac{d\mu(y)dt}{t}\r)^{1/2}\nonumber\\
\lesssim&[\mu(4B)]^{1/2}\tau\lf(\int_{\bx}a^*(x)\int_0^{\fz}\mathbf{D}^2_t(a)(x)
\,\frac{dtd\mu(x)}{t}\r)^{1/2}\nonumber\\
\lesssim&[\mu(4B)]^{1/2}\tau\lf(\int_{B}|a(x)|^2\,d\mu(x)\r)^{1/2}\\
\lesssim&1.\nonumber
\end{align}

For the term $\mathrm{K}_2$, we apply a duality argument to deduce that
\begin{align*}
\mathrm{K}_2&=\tau\int_{(4B)^{\complement}}
\lf(\iint_{\Gamma_x}|\mathbf{D}_t(a)(y)|^2\,\frac{d\mu(y)dt}{V_t(x)t}\r)
^{1/2}\,d\mu(x)\\
&=\sup_{\|g\|_{L_{\infty}(\lm; \,L_2^c(\bx\times \mathbb{R}_+,\,d\mu(y)dt))}\leq1}
\lf|\tau\int_{(4B)^{\complement}}
\iint_{\Gamma_x}\mathbf{D}_t(a)(y)g^*(x,\,y,\,t)\,\frac{d\mu(y)dt}{[V_t(x)t]^{1/2}}d\mu(x)\r|.\nonumber
\end{align*}
For any $g\in L_{\infty}(\lm; \,L_2^c(\bx\times \mathbb{R}_+,\,d\mu(y)dt))$ with $\|g\|_{L_{\infty}(\lm; \,L_2^c(\bx\times \mathbb{R}_+,\,d\mu(y)dt))}\leq1$, it follows from the cancellation condition of $a$, Lemma \ref{e2.5} and Fubini's theorem that
\begin{align}\label{e3.6}
&\lf|\tau\int_{(4B)^{\complement}}
\iint_{\Gamma_x}\mathbf{D}_t(a)(y)g^*(x,\,y,\,t)\,\frac{d\mu(y)dt}
{[V_t(x)t]^{1/2}}d\mu(x)\r| \\ \nonumber
=&\lf|\tau\int_{(4B)^{\complement}}
\iint_{\Gamma_x}\int_{B}a(z)[\mathbf{D}_t(y,\,z)-\mathbf{D}_t(y,\,x_B)]\,
d\mu(z)g^*(x,\,y,\,t)\,\frac{d\mu(y)dt}{[V_t(x)t]^{1/2}}d\mu(x)\r|\\ \nonumber
=&\tau\int_{B}|a(z)|\int_{(4B)^{\complement}}
\lf\|\iint_{\Gamma_x}[\mathbf{D}_t(y,\,z)-\mathbf{D}_t(y,\,x_B)]
g(x,\,y,\,t)\,\frac{d\mu(y)dt}
{[V_t(x)t]^{1/2}}\r\|
_{\cm}d\mu(x)d\mu(z)\\ \nonumber
\leq&\|g\|_{L_{\infty}(\lm; \,L_2^c(\bx\times \mathbb{R}_+,\,d\mu(y)dt))}\tau\int_{B}|a(z)|\int_{(4B)^{\complement}}
\lf(\iint_{\Gamma_x}|\mathbf{D}_t(y,\,z)-\mathbf{D}_t(y,\,x_B)|^2\,\frac{d\mu(y)dt}{V_t(x)t}\r)
^{1/2}d\mu(x)d\mu(z). \nonumber
\end{align}
Now we claim that for any $z\in B$ and $x\in(4B)^{\complement}$, one has
\begin{align}\label{e3.6x}
\lf(\iint_{\Gamma_x}|\mathbf{D}_t(y,\,z)-\mathbf{D}_t(y,\,x_B)|^2\,
\frac{d\mu(y)dt}{V_t(x)t}\r)^{1/2}
\lesssim\frac{[d(z,\,x_B)]^{\epsilon_1}}{V(x,\,x_B)[d(x,\,x_B)]^{\epsilon_1}}.
\end{align}
To show this inequality, we first
 write
\begin{align*}
\lf(\iint_{\Gamma_x}|\mathbf{D}_t(y,\,z)-\mathbf{D}_t(y,\,x_B)|^2\,
\frac{d\mu(y)dt}{V_t(x)t}\r)^{1/2}
\leq& \lf(\int_0^{\frac{d(x,\,x_B)}{4}}\int_{B_d(x,\,t)}|\mathbf{D}_t(y,\,z)-\mathbf{D}_t(y,\,x_B)|^2\,
\frac{d\mu(y)dt}{V_t(x)t}\r)^{1/2}\\
&\hspace{1.0cm} +\lf(\int_{\frac{d(x,\,x_B)}{4}}^{\fz}\int_{B_d(x,\,t)}|\mathbf{D}_t(y,\,z)-\mathbf{D}_t(y,\,x_B)|^2\,
\frac{d\mu(y)dt}{V_t(x)t}\r)^{1/2}\\
=&:\mathrm{K}_{21}+\mathrm{K}_{22}.
\end{align*}
For $\mathrm{K}_{21}$, it follows from $0<t<\frac{d(x,\,x_B)}{4}$ and $d(x,\,y)<t$ that
$$d(x,\,x_B)\leq d(x,\,y)+d(y,\,x_B)<t+d(y,\,x_B)\leq
\frac{d(x,\,x_B)}{4}+d(y,\,x_B),$$
and hence, $\frac{3d(x,\,x_B)}{4}\leq d(y,\,x_B)$, which implies that
$$V(y,\,x_B)\thicksim V(x_B,\,y)\gtrsim\mu(B_d(x_B,\,\frac{3d(x,\,x_B)}{4}))\gtrsim
V(x_B,\,x).$$
This, in combination with condition $(\mathbf{H}_3)$, yields
\begin{align*}
\mathrm{K}_{21}&\lesssim\lf(\int_0^{\frac{d(x,\,x_B)}{4}}\int_{B_d(x,\,t)}
\lf[\frac{d(z,\,x_B)}{t+d(y,\,x_B)}\r]^{2\epsilon_1}
\frac{1}{\lf[V(y,\,x_B)\r]^2}
\lf[\frac{t}{t+d(y,\,x_B)}\r]^{2\epsilon_2}\,
\frac{d\mu(y)dt}{V_t(x)t}\r)^{1/2}\\
&\lesssim\lf[\frac{d(z,\,x_B)}{d(x,\,x_B)}\r]^{\epsilon_1}
\lf(\int_0^{\frac{d(x,\,x_B)}{4}}\int_{B_d(x,\,t)}
\frac{1}{\lf[V(x,\,x_B)\r]^2}
\lf[\frac{t}{d(x,\,x_B)}\r]^{2\epsilon_2}\,
\frac{d\mu(y)dt}{V_t(x)t}\r)^{1/2}\\
&\lesssim\frac{r_B^{\epsilon_1}}{V(x,\,x_B)[d(x,\,x_B)]^{\epsilon_1}}.
\end{align*}
For $\mathrm{K}_{22}$, it follows from \eqref{e2.2} that for $t\geq\frac{d(x,\,x_B)}{4}$,
\begin{align*}
\frac{1}{V_t(x_B)}\lesssim\frac{1}{V(x_B,\,x)}
\lesssim\frac{1}{V(x,\,x_B)}.
\end{align*}
This, in combination with condition $(\mathbf{H}_3)$ and the fact that
$$t+d(y,\,x_B)>d(x,\,y)+d(y,\,x_B)\geq d(x,\,x_B),$$
implies that
\begin{align*}
\mathrm{K}_{22}&\lesssim\lf(\int_{\frac{d(x,\,x_B)}{4}}^{\fz}\int_{B_d(x,\,t)}
\lf[\frac{d(z,\,x_B)}{t+d(y,\,x_B)}\r]^{2\epsilon_1}
\frac{1}{\lf[V_t(x_B)\r]^2}
\lf[\frac{t}{t+d(y,\,x_B)}\r]^{2\epsilon_2}\,
\frac{d\mu(y)dt}{V_t(x)t}\r)^{1/2}\\
&\lesssim
\lf(\int_{\frac{d(x,\,x_B)}{4}}^{\fz}\int_{B_d(x,\,t)}
\lf[\frac{d(z,\,x_B)}{t}\r]^{2\epsilon_1}
\frac{1}{\lf[V(x,\,x_B)\r]^2}\,
\frac{d\mu(y)dt}{V_t(x)t}\r)^{1/2}\\
&\lesssim\frac{r_B^{\epsilon_1}}{V(x,\,x_B)[d(x,\,x_B)]^{\epsilon_1}},
\end{align*}
which completes the proof of \eqref{e3.6x}.

Substituting \eqref{e3.6x} into \eqref{e3.6} and then applying Lemma \ref{e2.5}, we conclude that
\begin{align*}
\mathrm{K}_2
\lesssim &\tau\int_{B}|a(z)|\int_{(4B)^{\complement}}
\frac{r_B^{\epsilon_1}}{V(x,\,x_B)[d(x,\,x_B)]^{\epsilon_1}}\,d\mu(x)d\mu(z)\\
\lesssim& [\mu(B)]^{1/2}\tau\lf(\int_{B}|a(z)|^2\,d\mu(z)\r)^{1/2}\\
\lesssim&1.
\end{align*}
This completes the proof of Lemma \ref{l3.21}.
\end{proof}

\begin{definition}
The {\it $\cm^c$-atom Hardy space}
$\mathcal{H}_{1}^{c,\mathrm{at}}(\bx,\,\cm)$
is defined to be the set of all  $f$ admitting a representation of the form
\begin{align*}
f=\sum_{j\in\nn} \lz_ja_j,
\end{align*}
where the $\{a_j\}_{j\in\nn}$ is a sequence of $\cm^c$-atoms and $\{\lambda_j\}_{j\in\nn}\subset\mathbb{C}$ satisfies $\sum_{j\in\nn} |\lz_j|<\infty$.
Moreover, for $f\in \mathcal{H}_{1}^{c,\mathrm{at}}(\bx,\,\cm)$, define
$$\|f\|_{\mathcal{H}_{1}^{c,\mathrm{at}}(\bx,\,\cm)}
:=\inf\lf\{\sum_{j\in\nn} |\lambda_j|:f=\sum_{j\in\nn}\lambda_j a_j, a_j \ \mathrm{is\ an} \ \cm^c\mathrm{-atom}\ \mathrm{and}\ \lambda_j\in\mathbb{C}\r\},$$
where the infimum is taken over all the decompositions of $f$ as above.

Similarly, one can define the Hardy spaces $\mathcal{H}_{1}^{r,\mathrm{at}}(\bx,\,\cm)$ and $\mathcal{H}_{1}^{cr,\mathrm{at}}(\bx,\,\cm)$.
\end{definition}
\begin{theorem}\label{t3.2}
Let $\dagger\in\{c,r,cr\}$, then
$$\mathcal{H}_1^{\dagger}(\bx,\,\cm)=\mathcal{H}_{1}^{\dagger,\mathrm{at}}(\bx,\,\cm)$$
with equivalent norms.
\end{theorem}
\begin{proof}[Proof of Theorem \ref{t3.2}]
By Theorem \ref{t3.1}, it remains to show an analogue of Coifman--Weiss theorem in the operator-valued setting, which refers to the following duality:
\begin{align}\label{CWatom0}
(\mathcal{H}_{1}^{c,\mathrm{at}}(\bx,\,\cm))^{*}\backsimeq\mathcal{BMO}^{c}(\bx,\,\cm).
\end{align}
We first show that $\mathcal{BMO}^{c}(\bx,\,\cm)\subset(\mathcal{H}_{1}^{c,\mathrm{at}}(\bx,\,\cm))^{*}.$
Let $g\in\mathcal{BMO}^{c}(\bx,\,\cm)$ and $f\in \mathcal{H}_{1}^{c,\mathrm{at}}(\bx,\,\cm)$. By the definition of $\mathcal{H}_{1}^{c,\mathrm{at}}(\bx,\,\cm)$, we know that there exist  $\{\lambda_j\}_{j\in\mathbb{N}}\subset\mathbb{C}$ and a sequence of  $\cm^c$-atoms $\{a_j\}_{j\in\mathbb{N}}$ supported in the balls $\{B_j\}_{j\in\nn}\subset\bx$ such that
$f=\sum_{j\in\nn} \lz_{j}a_j,$
where
\begin{align*}
\|f\|_{\mathcal{H}_{1}^{c,\mathrm{at}}(\bx,\,\cm)}
\thicksim\sum_{j\in\nn} |\lambda_j|<\infty.
\end{align*}
This, in combination with Lemma \ref{e2.5}, yields
\begin{align*}
\lf|\mathcal{L}_g(f)\r|&:=\lf|\tau\int_{\bx}f(x)g^*(x)\,d\mu(x)\r|\\
&\leq\sum_{j\in\nn}|\lambda_j|\lf|\tau\int_{\bx}
a_j(x)g^*(x)\,d\mu(x)\r|\\
&=\sum_{j\in\nn}|\lambda_j|\lf|\tau\int_{B_j}
a_j(x)\lf[g(x)-g_{B_j}\r]^*\,d\mu(x)\r|\\
&\leq\sum_{j\in\nn}|\lambda_j|
\|a_j\|_{L_1(\cm;\,L_2^c(B_j))}
\lf\|\lf(\int_{B_j}
\lf|g(x)-g_{B_j}\r|^2\,d\mu(x)\r)^{1/2}\r\|_{\cm}\\
&\leq\sum_{j\in\nn}|\lambda_j|
\lf\|\lf(\fint_{B_j}
\lf|g(x)-g_{B_j}\r|^2\,d\mu(x)\r)^{1/2}\r\|_{\cm}\\
&\lesssim\|f\|_{\mathcal{H}_{1}^{c,\mathrm{at}}(\bx,\,\cm)}
\|g\|_{\mathcal{BMO}^c(\bx,\,\cm)},
\end{align*}
which implies that $(\mathcal{H}_{1}^{c,\mathrm{at}}(\bx,\,\cm))^{*}
\supset\mathcal{BMO}^{c}(\bx,\,\cm)$.

Next we prove that $(\mathcal{H}_{1}^{c,\mathrm{at}}(\bx,\,\cm))^{*}
\subset\mathcal{BMO}^{c}(\bx,\,\cm)$.
For any ball $B\subset\bx$, let
$$L_0^1(\cm;\,L_2^c(B)):=\lf\{f\in L_1(\cm;\,L_2^c(B)):\int_{\bx} f(x)\,d\mu(x)=0\r\}.$$
Then we have
$$L_0^1(\cm;\,L_2^c(B))\subset\mathcal{H}_{1}^{c,\mathrm{at}}(\bx,\,\cm).$$
In fact, if $f\in L_0^1(\cm;\,L_2^c(B))$, then
$a:=\mu(B)^{-1/2}\|f\|_{L_1(\cm;\,L_2^c(B))}^{-1}f$
is an $\cm^c$-atom with $\supp a\subset B$, and
$$\|f\|_{\mathcal{H}_{1}^{c,\mathrm{at}}(\bx,\,\cm)}
\leq\mu(B)^{1/2}\|f\|_{L_1(\cm;\,L_2^c(B))}.$$
This indicates that for any $\cl\in(\mathcal{H}_{1}^{c,\mathrm{at}}(\bx,\,\cm))^*$, $\cl$ induces a continuous linear functional on $L_0^1(\cm;\,L_2^c(B))$ with norm bounded by $\mu(B)^{1/2}\|\cl\|_{(\mathcal{H}_{1}^{c,\mathrm{at}}(\bx,\,\cm))^*}$.
Fix $x_0\in\bx$, since $\bx=\bigcup_{k\in\zz_+}B_d(x_0,\,2^k)$, one may choose a sequence $\{g_k\}_{k\in\nn}\subset L_{\infty}(\cm;\,L_2^c(B_d(x_0,\,2^k)))$ such that
$$\cl(a)=\tau\int_{\bx} a(x)g_k(x)\,d\mu(x),\ \ \mathrm{for\ any}\ \cm^c\mathrm{-atom}\ a \ \mathrm{with\ \supp} a\subset B_d(x_0,\,2^k),$$
$$\|g_k\|_{L_{\infty}(\cm;\,L_2^c(B_d(x_0,\,2^k)))}\leq \lf[\mu(B_d(x_0,\,2^k))\r]^{1/2}\|\cl\|_{(\mathcal{H}_{1}^{c,\mathrm{at}}(\bx,\,\cm))^*},$$
and
$$g_k|_{B_d(x_0,\,2^l)}=g_l, \ \ \mathrm{for\ any}\ k>l\in\nn.$$
Define
$$g(x):=g_k(x), \ \ \mathrm{for\ any}\ x\in B_d(x_0,\,2^k)\backslash B_d(x_0,\,2^{k-1}),\ \ k\in\nn.$$
Therefore, for any $\cm^c$-atom $a$, one has
$$\cl(a)=\tau\int_{\bx}a(x)g^*(x)\,d\mu(x).$$

It remains to show that $g\in \mathcal{BMO}^c(\bx,\,\cm)$.
Let $H$ be the Hilbert space on which $\cm$ acts.
The standard inner product and the norm on $H$ will be denoted by $\langle\cdot,\,\cdot\rangle_H$ and $\|\cdot\|_H$, respectively.
 Recall that $\cm_{*}$ is a
quotient space of $B(H)_{*}$ by the preannihilator of $\cm$. Denote the quotient map
by $q$. For every $a,\,b\in H$, denote $$[a\otimes b]: = q(a\otimes b).$$
Notice that, for any $m\in\cm$,
$$\tau(m^*
[a\otimes b])=\tau([m^*
(a\otimes b)]) = \langle m(b),\, \overline{a}\rangle_H.$$
Let $\mathcal{BMO}(\bx,\,H)$ and $\mathcal{H}_1(\bx,\,H)$ denote the $H$-valued BMO and Hardy spaces on $\bx$, respectively.
From the classical duality between
$\mathcal{BMO}(\bx,\,H)$ and $\mathcal{H}_1(\bx,\,H)$, we deduce that
\begin{align*}
\|g\|_{\mathcal{BMO}^{c}(\bx,\,\cm)}
&=\sup_{e\in H,\,\|e\|_H=1}\|ge\|_{\mathrm{BMO}(\bx,\,H)}\\
&=\sup_{e\in H,\,\|e\|_H=1}\sup_{\|h\|_{\mathcal{H}_1(\bx,\,H)}=1}\lf|\int_{\bx}\lf\langle
g(x)e,\,\overline{h(x)}\r\rangle_H\,d\mu(x)\r|\\
&=\sup_{e\in H,\,\|e\|_H=1}\sup_{\|h\|_{\mathcal{H}_1(\bx,\,H)}=1}\lf|\tau\int_{\bx}
g^*(x)[h(x)\otimes e]\,d\mu(x)\r|.
\end{align*}
Denote by $f(x)=[h(x)\otimes e]$, $x\in \bx$. Since for any $t>0$ and $y\in\bx$,
$$|\mathbf{D}_t(f)(y)|^2=\langle \mathbf{D}_t(h)(y),\,\mathbf{D}_t(h)(y)\rangle_H[e\otimes e]=\|\mathbf{D}_t(h)(y)\|_H^2[e \otimes e],$$
we have
$$\tau(\cs^c(f))(x)=\lf(\iint_{\Gamma_x}\|\mathbf{D}_t(h)(y)\|_H^2\frac{d\mu(y)dt}{V_t(x)t}\r)^{1/2}.$$
Thus,
 $\|f\|_{\mathcal{H}_{1}^{c}(\bx,\,\cm)}= 1$ if $\|h\|_{\mathcal{H}_1(\bx,\,H)}=1$ and $\|e\|_H=1.$ This, together with the density of $\mathcal{H}_{1}^{c,\mathrm{at}}(\bx,\,\cm)$ in $\mathcal{H}_{1}^{c}(\bx,\,\cm)$, yields
$$\|g\|_{\mathcal{BMO}^{c}(\bx,\,\cm)}\lesssim\sup_{\|f\|
_{\mathcal{H}_{1}^{c,\mathrm{at}}(\bx,\,\cm)}\leq1}
\lf|\tau\int_{\bx}f(x)g^*(x)\,d\mu(x)\r|\\
\lesssim\|\cl\|
_{(\mathcal{H}_{1}^{c,\mathrm{at}}(\bx,\,\cm))^*}.$$
This completes the proof of Theorem \ref{t3.2}.
\end{proof}

\bigskip
\section{Duality between $H_p$ and $L_{p'}MO$}{\label{s4}}
\setcounter{equation}{0}

At the beginning of  this section, we will introduce $L_p$-space analogues of the BMO spaces $L_p\mathcal{MO}^{c}(\bx,\,\cm)$, $L_p\mathcal{MO}^{r}(\bx,\,\cm)$ and  $L_p\mathcal{MO}^{cr}(\bx,\,\cm)$ with $p\in(2,\,\infty)$. Then, we obtain their  predual spaces, which
are the operator-valued Hardy spaces $\mathcal{H}_{p'}^{c}(\bx,\,\cm)$, $\mathcal{H}_{p'}^{r}(\bx,\,\cm)$ and $\mathcal{H}_{p'}^{cr}(\bx,\,\cm)$ introduced in Section \ref{s3}, respectively.

Let us recall the definition of non-commutative maximal norms. Let $0<p\leq \infty$ and $x=\{x_i\}_{i}$ be a sequence of elements in $L_p(\cm)$. Define $L_p(\cm;\,\ell_{\infty})$-norm of $x$ by
$$\lf\|x\r\|_{L_p(\cm;\,\ell_{\infty})}:=\inf_{x_i=ay_ib}\lf\{\|a\|_{L_{2p}(\cm)}
\|b\|_{L_{2p}(\cm)}\sup_{i}\|y_i\|_{\cm}\r\},$$
where the infimum is taken over all $a,\,b\in L_{2p}(\cm)$ and $\{y_i\}_{i}\subset\cm$ such that
$x_i=ay_ib.$ Conventionally,
$\|x\|_{L_p(\cm;\,\ell_{\infty})}$ is usually denoted by $\lf\|\sup_{i}^{+}x_i\r\|_{L_p(\cm)}$. However, it is worthwhile to mention that $\sup_{i}^+x_i$
 is just a notation without making any specific
sense in the non-commutative setting. We just use this notation for convenience.
Moreover,
if $\Theta$ is an uncountable set, we define
$$\lf\|\{x_{\lambda}\}_{\lambda\in\Theta}\r\|_{L_p(\cm;\,\ell_{\infty})}:=\lf\|\sup_{\lambda\in\Theta}^{\ \ \ \ \  +}x_{\lambda}\r\|_{L_p(\cm)}:=\sup_{\{\lambda_i\}_{i}\subset\Theta}
\lf\|\sup_{i}^{\ \ \ \ \  +}x_{\lambda_i}\r\|_{L_p(\cm)}.$$
It was proved by Junge \cite{j02} that for any $p\in(1,\,\infty)$  and any sequence of positive operators $\{x_i\}_{i}$,
\begin{align}\label{e4.0}
\lf\|\sup_{i}^{\ \ \ \ \  +}x_i\r\|_{L_p(\cm)}=\sup\lf\{\sum_{i}\tau(x_iy_i):
y_i\in L_q(\cm),\,y_i\geq0,\,\lf\|\sum_{i}y_i\r\|_{L_q(\cm)}\leq1\r\}.
\end{align}

We also need the column and row non-commutative maximal norms, which was introduced by \cite{dj04}. Let $p\in [2,\infty)$. Then the space $L_p(\cm;\,\ell^c_{\infty})$ is defined to be the set of all sequences $\{x_k\}_{k\in\mathbb{N}}\subset L_p(\mathcal{M})$ which can be decomposed as $$x_k=y_ka\ {\rm for}\ {\rm some}\ \ a\in L_p(\mathcal{M})\ {\rm and}\ \{y_k\}_k\subset L_\infty(\mathcal{M})$$ satisfying $\sup\limits_{k\in\mathbb{N}}\|y_k\|_{\mathcal{M}}<\infty$. Furthermore, the $L_p(\cm;\,\ell^c_{\infty})$-norm of $\{x_k\}_k$ is defined to be
 $$\|\{x_k\}_k\|_{L_p(\cm;\,\ell_{\fz}^c)}:
=\inf_{x_k=y_ka}\sup\limits_{k\in\mathbb{N}}\|y_k\|_{\mathcal{M}}\|a\|_{L_p(\cm)},$$
where the infimum is taken over all factorization of $\{x_k\}_k$ as above. This norm is related closely to the $L_p(\cm;\,\ell_{\infty})$ space via the following equality:
$$\|\{x_k\}_k\|_{L_p(\cm;\,\ell_{\fz}^c)}=\|\{x_k^*x_k\}_k\|_{L_{\frac p2}(\cm;\,\ell_{\fz})}^{1/2}.$$
Similarly, one can define $L_p(\cm;\,\ell^r_{\infty})$ and connects its norm with $L_p(\cm;\,\ell_{\infty})$.

In what follows, for $p\in[1,\,\fz]$, let $L_p(\mathcal{M};\,L_2^{{\rm loc},c}(\bx))$ denote the set of all locally $L_2$-integrable functions on $\bx$ with values in $L_p(\cm)$. For any $g\in L_p(\mathcal{M};\,L_2^{{\rm loc},c}(\bx))$ and ball $B\subset\bx$, set
\begin{align*}
g_B^{\sharp}(x):=\fint_{B}\lf|g(y)
-g_{B}\r|^2\,d\mu(y),\ \  x\in B.
\end{align*}

\begin{definition}\label{d4.1}
Let $p\in(2,\,\infty)$. We define column $L_p\mathcal{MO}$ space
\begin{align*}
L_p\mathcal{MO}^{c}(\bx,\,\cm):=\lf\{g\in L_p(\mathcal{M};\,L_2^{{\rm loc},c}(\bx)):\|g\|_{L_q\mathcal{MO}^{c}(\bx,\,\cm)}<\infty\r\},
\end{align*}
where
\begin{align*}
\|g\|_{L_p\mathcal{MO}^{c}(\bx,\,\cm)}:=\lf\|\sup_{x\in B\subset\bx}^{\ \ \ \ \  +}g_B^{\sharp}\r\|
_{L_{\frac{p}{2}}(\lm)}^{1/2}.
\end{align*}
In a similar way as extending the definition of column BMO space to the row and mixture ones, one can define
the row space $L_p\mathcal{MO}^{r}(\lm)$ and mixture space $L_p\mathcal{MO}^{cr}(\lm)$, respectively.
\end{definition}
Now we state the main result of this section as follows.
\begin{theorem}\label{t4.1}
Let $p\in(1,\,2)$ and $\dagger\in\{c,r,cr\}$. Then
$$(\mathcal{H}_{p}^\dagger(\bx,\,\cm))^{\ast}\backsimeq L_{p'}\mathcal{MO}^{\dagger}(\bx,\,\cm)$$
in the following sense:
\begin{enumerate}
\item[\rm{(i)}]
 Each $g\in L_{p'}\mathcal{MO}^{\dagger}(\bx,\,\cm)$ defines a continuous linear
functional $\mathcal{L}_g$ on $\mathcal{H}_{p}^\dagger(\bx,\,\cm)$ by
\begin{align}\label{formmm}
\mathcal{L}_g(f):=\tau\int_{\bx}f(x)g^{*}(x)\,d\mu(x),\ \ \ \mathrm{for\ any}\ S_{\cm}\mathrm{-valued\ simple\ function}\ f.
\end{align}
\item[\rm{(ii)}] For any $\mathcal{L}\in(\mathcal{H}_{p}^\dagger(\bx,\,\cm))^{*}$, there exists some
$g\in L_{p'}\mathcal{MO}^{\dagger}(\bx,\,\cm)$ such that $\mathcal{L}=\mathcal{L}_g$.
\end{enumerate}
Moreover, there exists an universal positive constant $C$ such that
$$C^{-1}\|g\|_{L_{p'}\mathcal{MO}^{\dagger}(\bx,\,\cm)}
\leq\|\mathcal{L}_g\|_{(\mathcal{H}_{p}^\dagger(\bx,\,\cm))^{*}}\leq C\|g\|_{L_{p'}\mathcal{MO}^{\dagger}(\bx,\,\cm)}.$$
\end{theorem}
To prove Theorem \ref{t4.1}, we shall combine the argument of showing Theorem \ref{t3.1} with non-commutative maximal inequality on space of homogeneous type. Firstly, for any $x\in B$, let
$$\mu_{g,\,B}^{\sharp}(x):=\frac{1}{\mu(B)}\int_{ T(B)}\lf|\mathbf{D}_t(g)(y)\r|^2\,d\mu(y)\frac{dt}{t}.$$

\begin{proposition}\label{p4.00}
For any $p\in(2,\,\infty)$, there exists a constant $C>0$ such that
$$\lf\|\sup_{x\in B\subset\bx}^{\ \ \ \ \  +}\mu_{g,\,B}^{\sharp}\r\|_{L_{\frac{p}{2}}(\lm)}
^{1/2}\leq C\|g\|_{L_p\mathcal{MO}^{c}(\bx,\,\cm)}.$$
\end{proposition}
\begin{proof}
Let $g\in L_p\mathcal{MO}^{c}(\bx,\,\cm)$ and the ball $B:=B(x_B,\,r_B)$ with $x_B\in\bx,\,r_B>0$. Then we write
\begin{align*}
g&=(g-g_{4B})\chi_{4B}
+(g-g_{4B})\chi_{(4B)^{\complement}}
+g_{4B}\\
&=:g_1+g_2+g_3.
\end{align*}
Thus,
\begin{align*}
\lf\|\sup_{x\in B\subset\bx}^{\ \ \ \ \  +}\mu_{g,\,B}^{\sharp}\r\|_{L_{\frac{p}{2}}(\lm)}
\lesssim&\lf\|\sup_{x\in B\subset\bx}^{\ \ \ \ \  +}\frac{1}{\mu(B)}\int_{T(B)}\lf|\mathbf{D}_t(g_1)(y)\r|^2\,d\mu(y)\frac{dt}{t}\r\|
_{L_{\frac{p}{2}}(\lm)}\\
&\ \  +\lf\|\sup_{x\in B\subset\bx}^{\ \ \ \ \  +}\frac{1}{\mu(B)}\int_{T(B)}\lf|\mathbf{D}_t(g_2)(y)\r|^2\,d\mu(y)\frac{dt}{t}\r\|
_{L_{\frac{p}{2}}(\lm)}\\
&\ \  +\lf\|\sup_{x\in B\subset\bx}^{\ \ \ \ \  +}
\frac{1}{\mu(B)}\int_{T(B)}\lf|\mathbf{D}_t(g_3)(y)\r|^2\,\frac{d\mu(y)dt}{t}\r\|
_{L_{\frac{p}{2}}(\lm)}\\
=&:\mathrm{K}_1+\mathrm{K}_2+\mathrm{K}_3.
\end{align*}

For the term $\mathrm{K}_1$, similar to the proof of \eqref{rep1}, we deduce that
\begin{align*}
\mathrm{K}_1&\leq
\lf\|\sup_{x\in B\subset\bx}^{\ \ \ \ \  +}\frac{1}{\mu(B)}\int_{4B}\lf|g(z)-g_{4B}\r|^2
\,d\mu(z)\r\|_{L_{\frac{p}{2}}(\lm)}\lesssim\|g\|_{L_p\mathcal{MO}^{c}(\bx,\,\cm)}^2.
\end{align*}

For the term $\mathrm{K}_2$, substituting inequalities \eqref{rep2} and \eqref{e2.7}  into \eqref{e2.6}, we conclude that, for any $y\in B_d(x_B,\,r_B)$ and $t>0$,
\begin{align*}
\lf|\mathbf{D}_t(g_2)(y)\r|^2&\lesssim \lf(\frac{t}{r_B}\r)^{2\epsilon_2} \sum_{j=2}^{\infty}\frac{1}{2^{j}} \fint_{2^{j+1}B}
\lf|g(z)-g_{2^{j+1}B}\r|^2
+\lf|g_{2^{j+2}B}-g_{4B}\r|^2\,d\mu(z).
\end{align*}
This, in combination with \eqref{e2.6x}, yields
\begin{align*}
\mathrm{K}_2
&=\lf\|\sup_{x\in B\subset\bx}^{\ \ \ \ \  +}\frac{1}{\mu(B)}\int_0^{r_B}\int_{B}
\lf|\mathbf{D}_t(g_2)(y)\r|^2\,d\mu(y)\frac{dt}{t}\r\|_{L_{\frac{p}{2}}
(\lm)}\\
&\lesssim\lf\|\sup_{x\in B\subset\bx}^{\ \ \ \ \  +}\sum_{j=2}^{\infty}\frac{1}{2^j}
\fint_{2^{j+1}B}\lf|g(z)-g_{2^{j+1}B}\r|^2\,d\mu(z)\r\|
_{L_{\frac{p}{2}}(\lm)}\\
&\hspace{1.0cm}+\lf\|\sup_{x\in B\subset\bx}^{\ \ \ \ \  +}\sum_{j=2}^{\infty}\frac{1}{2^{j}}
\lf|g_{2^{j+1}B}-g_{4B}\r|^2\r\|_{L_{\frac{p}{2}}
(\lm)}\\
&\lesssim\|g\|_{L_p\mathcal{MO}^{c}(\bx,\,\cm)}^2.
\end{align*}

For the term $\mathrm{K}_3$, we apply the condition $(\mathbf{H}_4)$ to obtain
that  $\mathbf{D}_t(g_3)(x)\equiv0$ for any $x\in\bx$ and $t>0$. Therefore, $\mathrm{K}_3=0$.

Combining the estimates of $\mathrm{K}_1$, $\mathrm{K}_2$ and $\mathrm{K}_3$, we end the proof of Proposition \ref{p4.00}.
\end{proof}

\begin{lemma}\label{hlmi}
Let $p\in(1,\,\fz)$. Then there exists a positive constant $c_p$ such that
for any positive sequence $\{a_k\}_{k\in\zz}\subseteq L_p(\lm)$,
$$\lf\|\sum_{k\in\zz}\fint_{B_d(x,\,\delta^k)} a_k(y)\,d\mu(y)\r\|
_{L_p(\lm)}
\leq c_p
\lf\|\sum_{k\in\zz}a_k\r\|
_{L_p(\lm)}
.$$
\end{lemma}
\begin{proof}
From Lemma \ref{l6.1},
for any $x\in\bx$ and $k\in\zz$, there exists $Q\in\mathcal{Q}^{(i)}_{k-2}$ for some $1 \leq i \leq I_0$ such that
$B_d(x,\,\delta^{k})\subset Q$ and $\mu(Q)\lesssim\mu(B_d(x,\,\delta^k))$.
Therefore,
$$
\fint_{B_d(x,\,\delta^k)} a_k(y)\,d\mu(y)
\lesssim
\fint_{Q} a_k(y)\,d\mu(y)=\mathbb{E}_{k-2}^{(i)}(a_k)(x).
$$
From this, we conclude that
\begin{align*}
\lf\|\sum_{k\in\zz}\fint_{B_d(x,\,\delta^k)} a_k(y)\,d\mu(y)\r\|
_{L_p(\lm)}
\lesssim\lf\|\sum_{k\in\zz}\lf(\sum_{i=0}^{I_0}\mathbb{E}_{k-2}^{(i)}(a_k)\r)\r\|
_{L_p(\lm)}
\lesssim\sum_{i=0}^{I_0}\lf\|\sum_{k\in\zz}\mathbb{E}_{k-2}^{(i)}(a_k)\r\|
_{L_p(\lm)}.
\end{align*}
where $\mathbb{E}_k^{(i)}$($0\leq i\leq I_0$) denotes the conditional expectation with respect to $\mathcal{Q}^{(i)}_{k}$. This, in combination with Doob's maximal inequality given in \cite[Theorem 0.1]{j02}, ends the proof of Lemma \ref{hlmi}.
\end{proof}

The following lemma
 plays an important role in the proof of Theorem \ref{t4.1}, which can be regarded as an extension of Lemma \ref{l3.2} to the case of $p\in (2,\,\infty)$, but it requires some extra techniques from non-commutative maximal inequality and non-commutative martingale.

\begin{lemma}\label{l4.1}
Let  $\Psi$ be the map given in Lemma \ref{d3.1}. Then for any $p\in(2,\,\infty)$,
 $\Psi$ is bounded from $L_{p}(\lm;\,L_2^c(\bx\times \mathbb{R}_+,\,d\mu(y)dt))$ to $L_p\mathcal{MO}^{c}(\bx,\,\cm)$. Similar result also holds for the row space.
\end{lemma}
\begin{proof}
It can be seen from Lemma \ref{l2.0} that for any $f\in L_p\mathcal{MO}^{c}(\bx,\,\cm)$,
\begin{align*}
\|f\|_{L_p\mathcal{MO}^{c}(\bx,\,\cm)}\sim\left\|
\sup_{k\in\mathbb{Z}}\fint_{B_d(\omega,\,\delta^k)}\lf|f(y)-f_{B_d(\omega,\,\delta^k)}\r|
^2\,d\mu(y)\right\|_{L_{\frac{p}{2}}(\mathcal{N})}^{1/2}.
\end{align*}
Then, it suffices to show that for any $g\in L_{p}(\lm;\,L_2^c(\bx\times \mathbb{R}_+,\,d\mu(y)dt))$,
\begin{align*}
\left\|\sup_{k\in\mathbb{Z}}\fint_{B_d(\omega,\,\delta^k)}
\lf|\Psi(g)(z)-\Psi(g)_{B_d(\omega,\,\delta^k)}\r|^2\,d\mu(z)\right\|_{L_{\frac{p}{2}}(\mathcal{N})}^{1/2}\lesssim \|g\|_{L_{p}(\lm;\,L_2^c(\bx\times \mathbb{R}_+,\,d\mu(y)dt))}.
\end{align*}
To show this inequality, we write
\begin{align*}
g(x,\,y,\,t)&=g(x,\,y,\,t)\chi_{B(\omega,\,\delta^{k-1})}(x)+g(x,\,y,\,t)
\chi_{(B(\omega,\,\delta^{k-1}))^{\complement}}(x)\\
&=:g_1(x,\,y,\,t)+g_2(x,\,y,\,t).
\end{align*}
By \eqref{iopiopiop}, we only need to show the $L_{\frac{p}{2}}(\mathcal{N};\,\ell_\infty)$-norm of expressions
\begin{align}\label{Agoal1}
\fint_{B(\omega,\,\delta^k)}\lf|
\int_{(B(\omega,\,\delta^{k-1}))
^{\complement}}
\iint_{\Gamma_x}g_2(x,\,y,\,t)
\lf[\mathbf{D}_t(z,\,y)-Q_{B(\omega,\,\delta^k)}(y)\r]\,\frac{d\mu(y)dt}{\lf[V_t(x)t\r]^{1/2}}d\mu(x)\r|^2\,d\mu(z)
\end{align}
and
\begin{align}\label{Agoal2}
\fint_{B(\omega,\,\delta^k)}\lf|\int_{\bx}\iint_{\Gamma_x}g_1(x,\,y,\,t)
\overline{\mathbf{D}_t(y,\,z)}\,\frac{d\mu(y)dt}{\lf[V_t(x)t\r]^{1/2}}d\mu(x)\r|^2\,d\mu(z)
\end{align}
are dominated by $\|g\|_{L_{p}(\lm;\,L_2^c(\bx\times \mathbb{R}_+,\,d\mu(y)dt))}^2$, where we recall that $Q_{B}(y):=\fint_{B}\mathbf{D}_t(z,\,y)\,d\mu(z)$ for $y\in\bx$.

To estimate the $L_{\frac{p}{2}}(\mathcal{N};\,\ell_\infty)$-norm of \eqref{Agoal1}, we first follow the proof of \eqref{eeeee000} and then \eqref{3.1a} to deduce that if $z\in B(\omega,\,\delta^{k})$, then
\begin{align*}
&\lf|
\int_{(B(\omega,\,\delta^{k-1}))
^{\complement}}
\iint_{\Gamma_x}g_2(x,\,y,\,t)
\lf[\mathbf{D}_t(z,\,y)-Q_{B(\omega,\,\delta^k)}(y)\r]\,\frac{d\mu(y)dt}{\lf[V_t(x)t\r]^{1/2}}d\mu(x)\r|^2\\
\lesssim&\delta^{k\epsilon_1}\int_{(B(\omega,\,\delta^{k-1}))^{\complement}}V(x,\,z)[d(x,\,z)]^{\epsilon_1}
\fint_{B(\omega,\,\delta^k)}\iint_{\Gamma_x}
\lf|\overline{\mathbf{D}_t(y,\,z)}-\overline{\mathbf{D}_t(y,\,s)}\r|^2
\,\frac{d\mu(y)dt}{V_t(x)t}d\mu(s)\\
&\hspace{1.0cm} \times\iint_{\Gamma_x}|g(x,\,y,\,t)|^2\,d\mu(y)dtd\mu(x)\\
\lesssim& \sum_{j=-\infty}^{k}\delta^{(k-j)\epsilon_1}\int_{B(\omega,\,\delta^{j-1})\backslash B(\omega,\,\delta^{j})}\frac{1}{V(x,\,z)}
\iint_{\Gamma_x}|g(x,\,y,\,t)|^2\,d\mu(y)dtd\mu(x),
\end{align*}
where the last inequality is a variant of the combination of inequalities \eqref{ana1} and \eqref{ana2} with a slight modification that the $L_{\infty}(\lm;\,L_2^c(\bx\times \mathbb{R}_+,\,d\mu(y)dt))$-norm of $g$ is not taken outside. Next, we shall use a duality argument to estimate the $L_{\frac{p}{2}}(\mathcal{N};\,\ell_\infty)$-norm of the right-hand side. To this end, we take a sequence of positive elements $\{a_k\}_{k\in\mathbb{Z}}\in L_{(\frac{p}{2})'}(\mathcal{N})$ such that $\|\sum_k a_k\|_{L_{(\frac{p}{2})'}(\mathcal{N})}=1$. Then, we use H\"{o}lder's inequality and Lemma \ref{hlmi} to see that
\begin{align*}
&\tau\int_\bx\sum_{k\in\mathbb{Z}} \sum_{j=-\infty}^{k}\delta^{(k-j)\epsilon_1}\fint_{B(\omega,\,\delta^{j-1})}
\iint_{\Gamma_x}|g(x,\,y,\,t)|^2\,d\mu(y)dtd\mu(x)a_k(\omega)d\mu(\omega)\\
&=\tau \int_\bx\iint_{\Gamma_x}|g(x,\,y,\,t)|^2\,d\mu(y)dt\sum_{j\in\mathbb{Z}} \fint_{B(x,\,\delta^{j-1})}
\sum_{k=j}^{\infty}\delta^{(k-j)\epsilon_1}a_k(\omega)d\mu(\omega)d\mu(x)\\
&\leq \|g\|_{L_{p}(\lm;\,L_2^c(\bx\times \mathbb{R}_+,\,d\mu(y)dt))}^2
\left\|\sum_{j\in\mathbb{Z}}\fint_{B(x,\,\delta^{j-1})}
\sum_{k=j}^{\infty}\delta^{(k-j)\epsilon_1}a_k(\omega)d\mu(\omega)
\right\|_{L_{(\frac{p}{2})'}(\mathcal{N})}\\
& \lesssim \|g\|_{L_{p}(\lm;\,L_2^c(\bx\times \mathbb{R}_+,\,d\mu(y)dt))}^2
\left\|\sum_{k\in\zz}\sum_{j=-\fz}^k\delta^{(k-j)
\epsilon_1}a_k\right\|_{L_{(\frac{p}{2})'}(\mathcal{N})} \\
&\lesssim \|g\|_{L_{p}(\lm;\,L_2^c(\bx\times \mathbb{R}_+,\,d\mu(y)dt))}^2.
\end{align*}
By taking the supremum in the above inequality over all $\{a_k\}_{k\in\mathbb{Z}}$ satisfying the stating properties, we conclude that the $L_{\frac{p}{2}}(\mathcal{N};\,\ell_\infty)$-norm of \eqref{Agoal1} is dominated by $\|g\|_{L_{p}(\lm;\,L_2^c(\bx\times \mathbb{R}_+,\,d\mu(y)dt))}^2$.

To estimate the $L_{\frac{p}{2}}(\mathcal{N};\,\ell_\infty)$-norm of \eqref{Agoal2}, we shall use the duality argument again. To this end, we take a sequence of positive elements $\{b_k\}_{k\in\mathbb{Z}}\in L_{(\frac{p}{2})'}(\mathcal{N})$ such that $\|\sum_k b_k\|_{L_{(\frac{p}{2})'}(\mathcal{N})}=1$. Note that
\begin{align}\label{wer13}
&\tau\int_{\bx}\sum_{k\in\mathbb{Z}}\fint_{B(\omega,\,\delta^k)}\lf|\int_{\bx}\iint_{\Gamma_x}g_1(x,\,y,\,t)
\overline{\mathbf{D}_t(y,\,z)}\,\frac{d\mu(y)dt}{\lf[V_t(x)t\r]^{1/2}}d\mu(x)\r|^2d\mu(z)\,b_k(\omega)d\mu(\omega)\nonumber\\
&\leq\sum_{k\in\mathbb{Z}}
\int_{\bx}\frac{1}{\mu(B(\omega,\,\delta^{k}))}\nonumber\\
&\hspace{1.0cm}\times\left\|\int_{B(\omega,\,\delta^{k-1})}\iint_{\Gamma_x}g(x,\,y,\,t)b_k(\omega)^{1/2}
\overline{\mathbf{D}_t(y,\,\cdot)}\,\frac{d\mu(y)dt}{\lf[V_t(x)t\r]^{1/2}}d\mu(x)\right\|_{L_2(\mathcal{N})}^2\,d\mu(\omega).
\end{align}
To estimate the $L_2(\mathcal{N})$ in the above expression, we use the duality again, Cauchy-Schwartz inequality and Remark \ref{r2.9} (iii) to deduce that
\begin{align}\label{sde1}
&\left\|\int_{B(\omega,\,\delta^{k-1})}\iint_{\Gamma_x}g(x,\,y,\,t)b_k(\omega)^{1/2}
\overline{\mathbf{D}_t(y,\,z)}\,\frac{d\mu(y)dt}{\lf[V_t(x)t\r]^{1/2}}d\mu(x)\right\|_{L_2(\mathcal{N})}^2\nonumber\\
&=\sup\limits_{\|h\|_{L_2(\mathcal{N})=1}}\left|\tau\int_\bx\int_{B(\omega,\,\delta^{k-1})}\iint_{\Gamma_x}g(x,\,y,\,t)b_k(\omega)^{1/2}
\overline{\mathbf{D}_t(y,\,z)}\,\frac{d\mu(y)dt}{\lf[V_t(x)t\r]
^{1/2}}d\mu(x)h(z)d\mu(z)\right|^2                                       \nonumber\\
&=\sup\limits_{\|h\|_{L_2(\mathcal{N})=1}}\left|\tau\int_{B(\omega,\,\delta^{k-1})}\iint_{\Gamma_x}g(x,\,y,\,t)b_k(\omega)^{1/2}
(\mathbf{D}_t(h^*)(y))^*\,\frac{d\mu(y)dt}{\lf[V_t(x)t\r]
^{1/2}}d\mu(x)\right|^2\nonumber\\
&\lesssim \tau\int_{B(\omega,\,\delta^{k-1})}\iint_{\Gamma_x}|g(x,\,y,\,t)|^2b_k(\omega)
d\mu(y)dtd\mu(x).
\end{align}
Substituting the inequality \eqref{sde1} into \eqref{wer13}, we conclude that
\begin{align*}
&\tau\int_{\bx}\sum_{k\in\mathbb{Z}}\fint_{B(\omega,\,\delta^k)}\lf|\int_{\bx}\iint_{\Gamma_x}g_1(x,\,y,\,t)
\overline{\mathbf{D}_t(y,\,z)}\,\frac{d\mu(y)dt}{\lf[V_t(x)t\r]^{1/2}}d\mu(x)\r|^2d\mu(z)\,b_k(\omega)d\mu(\omega)\nonumber\\
&\lesssim \tau
\int_{\bx}\iint_{\Gamma_x}|g(x,\,y,\,t)|^2d\mu(y)dt
\sum_{k\in\mathbb{Z}}\fint_{B(x,\,\delta^{k-1})}b_k(\omega)d\mu(\omega)
d\mu(x)\\
&\lesssim \|g\|_{L_{p}(\lm;\,L_2^c(\bx\times \mathbb{R}_+,\,d\mu(y)dt))}^2 \left\|\sum_{k\in\mathbb{Z}}\fint_{B(x,\,\delta^{k-1})}
b_k(\omega)d\mu(\omega)\right\|_{L_{(\frac{p}{2})'}(\mathcal{N})}\\
&\lesssim  \|g\|_{L_{p}(\lm;\,L_2^c(\bx\times \mathbb{R}_+,\,d\mu(y)dt))}^2\left\|\sum_{k\in\mathbb{Z}}b_k\right\|_{L_{(\frac{p}{2})'}(\mathcal{N})}\\
&\lesssim \|g\|_{L_{p}(\lm;\,L_2^c(\bx\times \mathbb{R}_+,\,d\mu(y)dt))}^2.
\end{align*}
By taking the supremum in the above inequality over all $\{b_k\}_{k\in\mathbb{Z}}$ satisfying the stating properties, we conclude that the $L_{\frac{p}{2}}(\mathcal{N};\,\ell_\infty)$-norm of \eqref{Agoal2} is dominated by $\|g\|_{L_{p}(\lm;\,L_2^c(\bx\times \mathbb{R}_+,\,d\mu(y)dt))}^2$. This finishes the proof of Lemma \ref{l4.1}.
\end{proof}


\begin{proof}[Proof of Theorem \ref{t4.1}]We first prove Theorem \ref{t4.1} (i).  We shall show that for any $S_{\cm}$-valued simple function $f$  and $g\in L_{p'}\mathcal{MO}^{c}(\bx,\,\cm)$,
\begin{align}\label{duagoal}
\lf|\cl_g(f)\r|\lesssim\|f\|_{\mathcal{H}_{p}^c(\bx,\,\cm)}\|g\|_
{L_{p'}\mathcal{MO}^c(\bx, \,\cm)}.
\end{align}
By equality \eqref{e3.1}, we have
\begin{align}\label{jkl1}
\lf|\cl_g(f)\r|^2
&=C\lf|\tau\int_{\bx}\int_0^{\fz}\int_{B_d(x,\,\frac{t}{2})}\mathbf{D}_t(f)(z)\lf[\mathbf{D}_t(g)(z)\r]^{*}
\,\frac{dtd\mu(z)}{V_{\frac{t}{2}}(x)t}d\mu(x)\r|^2 \nonumber\\
&\leq C\tau\int_{\bx}\int_0^{\fz}\lf[S_{\mathbf{D}}^c(f)(x,\,t)\r]^{p-2}
\int_{B_d(x,\,\frac{t}{2})}\lf|\mathbf{D}_t(f)(z)\r|^2
\,\frac{dtd\mu(z)}{V_{\frac{t}{2}}(x)t}d\mu(x)\nonumber\\
&\hspace{1.0cm} \times \tau\int_{\bx}\int_0^{\fz}\lf[S_{\mathbf{D}}^c(f)(x,\,t)\r]^{2-p}\int_{B_d(x,\,\frac{t}{2})}
\lf|\mathbf{D}_t(g)(z)\r|^2
\,\frac{dtd\mu(z)}{V_{\frac{t}{2}}(x)t}d\mu(x)\nonumber\\
&=:\mathrm{L}_1 \times \mathrm{L}_2,
\end{align}
where $S_{\mathbf{D}}^c(f)(x,\,t)$ and $\widetilde{S}_{\mathbf{D}}^c(f)(x,\,t)$ are given in \eqref{ggg1} and \eqref{ggg2}, respectively. These functions satisfy $[S_{\mathbf{D}}^c(f)(x,\,t)]^{p-2}\leq[\widetilde{S}_{\mathbf{D}}^c(f)(x,\,t)]^{p-2}$ for any $(x,\,t)\in\bx \times\rr_+$ and $p\in(1,\,2)$. Therefore,
\begin{align}\label{jkl2}
\mathrm{L}_1&\leq\tau\int_{\bx}\int_0^{\fz}\lf[\widetilde{S}_{\mathbf{D}}^{c}
(f)(x,\,t)\r]^{p-2}\int_{B_{d}(x,\,\frac{t}{2})}|\mathbf{D}_t(f)(z)|^2\,
\frac{d\mu(z)dt}{V_{\frac{t}{2}}(x)t}d\mu(x)\nonumber\\
&=-\tau\int_{\bx}\int_0^{\fz}\lf[\widetilde{S}_{\mathbf{D}}^{c}
(f)(x,\,t)\r]^{p-2}\frac{\partial}{\partial t}\lf[\widetilde{S}_{\mathbf{D}}^{c}
(f)(x,\,t)\r]^2\,dtd\mu(x)\nonumber\\
&\thicksim-\tau\int_{\bx}\int_0^{\fz}\lf[\widetilde{S}_{\mathbf{D}}^{c}
(f)(x,\,t)\r]^{p-1}\frac{\partial}{\partial t}\widetilde{S}_{\mathbf{D}}^{c}
(f)(x,\,t)\,dtd\mu(x)\nonumber\\
&\lesssim\tau\int_{\bx}\lf[\widetilde{S}_{\mathbf{D}}^{c}
(f)(x,\,0)\r]^{p}\,d\mu(x)\nonumber\\
&\lesssim \|f\|^p_{\mathcal{H}_p^{c}(\bx,\,\cm)}.
\end{align}

Now we estimate the term $\mathrm{L}_{2}$. Let $c_\delta:=1+(1-\delta)^{-1}$,
then
\begin{align}\label{e4.4}
\mathrm{L_2}
&=\tau\sum_{k\in\zz}\sum_{\alpha\in I_k}\int_{Q_{\alpha}^k}\int_{2C_4c_\delta\delta^k}
^{2C_4c_\delta\delta^{k-1}}\lf[S_{\mathbf{D}}^c(f)(x,\,t)\r]^{2-p}\int_{B_d(x,\,\frac{t}{2})}
\lf|\mathbf{D}_t(g)(z)\r|^2
\,\frac{d\mu(z)}{V_{\frac{t}{2}}(x)t}dtd\mu(x). \nonumber\\
&\leq\tau\sum_{k\in\zz}\sum_{\alpha\in I_k}\int_{Q_{\alpha}^k}\int_{2C_4c_\delta\delta^k}
^{2C_4c_\delta\delta^{k-1}}\lf[\mathbb{S}_{\mathbf{D}}^c(f)(x,\,k)\r]^{2-p}\int_{B_d(x,\,\frac{t}{2})}
\lf|\mathbf{D}_t(g)(z)\r|^2
\,\frac{d\mu(z)}{V_{\frac{t}{2}}(x)t}dtd\mu(x)\nonumber\\
   &=\tau\int_{\bx}\sum_{k\in\zz}\lf[\mathbb{S}_{\mathbf{D}}^c(f)(x,\,k)\r]^{2-p}\int_{2C_4c_\delta\delta^k}
^{2C_4c_\delta\delta^{k-1}}\int_{B_d(x,\,\frac{t}{2})}\lf|\mathbf{D}_t(g)(z)\r|^2
\,\frac{d\mu(z)}{V_{\frac{t}{2}}(x)t}dtd\mu(x),
\end{align}
where $Q_{\alpha}^k\in\mathcal{Q}$ and $I_k$, $\mathcal{Q}$, $C_4$, $\delta$ are given in Lemma \ref{l3.1} and where $\mathbb{S}_{\mathbf{D}}^c(f)$ is given in \eqref{mathS}.
%
Let  $\mathbb{V}(x,\,j):=[\mathbb{S}_{\mathbf{D}}^c(f)(x,\,j)]^{2-p}-[\mathbb{S}_{\mathbf{D}}^c(f)(x,\,j-1)]^{2-p}$ for any $x\in\bx$ and $j\in\zz$. By \eqref{e4.4} (which is a slight modification of \eqref{e3.5}), we can repeat the argument in showing \eqref{rept1} to deduce that
\begin{align*}
\mathrm{L}_{2}
&\lesssim\tau\sum_{j\in\zz}\sum_{\alpha\in I_j}\mathbb{V}(z_{\alpha}^j,\,j)
\int_{0}^{2C_4c_\delta\delta^{j-1}}\int_{B_d(z_{\alpha}^j,\,2C_4c_\delta\delta^{j-1})}\lf|\mathbf{D}_t(g)(z)\r|^2
\,\frac{d\mu(z)dt}{t}.
\end{align*}
Next, we apply H\"{o}lder's inequality and Lemma \ref{l3.1} to conclude that
\begin{align}\label{jkl3}
\mathrm{L}_{2}&\lesssim\tau\sum_{j\in\zz}\sum_{\alpha\in I_j}\int_{Q_{\alpha}^j}\mathbb{V}(x,\,j)
\int_{0}^{2C_4c_\delta\delta^{j-1}}\fint_{B_d(z_{\alpha}^j,\,2C_4c_\delta\delta^{j-1})}\lf|\mathbf{D}_t(g)(z)\r|^2
\,\frac{d\mu(z)dt}{t}d\mu(x)\nonumber\\
&\lesssim\tau\sum_{j\in\zz}\int_{\bx}\mathbb{V}(x,\,j)
\int_{0}^{2C_4c_\delta\delta^{j-1}}\fint_{B_d(x,\,4C_4c_\delta\delta^{j-1})}\lf|\mathbf{D}_t(g)(z)\r|^2
\,\frac{d\mu(z)dt}{t}d\mu(x)\nonumber\\
&\lesssim\lf\|\sum_{j\in\zz}\mathbb{V}(\cdot,\,j)\r\|_{L_{(\frac{p'}{2})'}(\lm)}
\lf\|\sup_{j\in\zz}^{\ \ \ \ \  +}
\int_{0}^{2C_4c_\delta\delta^{j-1}}\fint_{B_d(x,\,4C_4c_\delta\delta^{j-1})}\lf|\mathbf{D}_t(g)(z)\r|^2
\,\frac{d\mu(z)dt}{t}\r\|_{L_{\frac{p'}{2}}(\lm)}\\
&\lesssim\lf\|\lf[\mathbb{S}_{\mathbf{D}}^c(\cdot,\,+\infty)\r]^{2-p}\r\|_{L_{(\frac{p'}{2})'}(\lm)}
\|g\|_{L_{p'}\mathcal{MO}^c(\bx,\,\cm)}^2\nonumber\\
&\lesssim\|f\|_{\mathcal{H}_p^{c}(\bx,\,\cm)}
^{2-p}\|g\|_{L_{p'}\mathcal{MO}^c(\bx,\,\cm)}^2.\nonumber
\end{align}
Combining the estimates of $\mathrm{L}_1$ and $\mathrm{L}_2$ finishes the proof of \eqref{duagoal} and then Theorem \ref{t4.1} (i).

Next we prove Theorem \ref{t4.1} (ii). Let $\cl\in(\mathcal{H}_p^{c}(\bx,\,\cm))^{\ast}$.
By the definition of $\Phi$ given in Lemma \ref{d3.1} together with the Hahn-Banach theorem, we know that
$\cl$ can be extended to a continuous functional on $L_{p}(\lm;\,L_2^c(\bx\times \mathbb{R}_+,\,d\mu(y)dt))$ with the same norm.
Moreover, since
$$(L_{p}(\lm;\,L_2^c(\bx\times \mathbb{R}_+,\,d\mu(y)dt)))^{\ast}
=L_{p'}(\lm;\,L_2^c(\bx\times \mathbb{R}_+,\,d\mu(y)dt)),$$
there exists $h\in L_{p'}(\lm;\,L_2^c(\bx\times \mathbb{R}_+,\,d\mu(y)dt))$ such that
$
\|h\|_{L_{p'}(\lm;\,L_2^c(\bx\times \mathbb{R}_+,\,d\mu(y)dt))}
=\|\cl\|_{(\mathcal{H}_p^{c}(\bx,\,\cm))^{\ast}}.
$
This, together with Lemma \ref{l4.1}, implies that $g:=\Psi(h)\in L_{p'}\mathcal{MO}^c(\bx,\,\cm)$, satisfies equality \eqref{euqa} and
\begin{align*}
\|g\|_{L_{p'}\mathcal{MO}^c(\bx,\,\cm)}
=\|\Psi(h)\|_{L_{p'}\mathcal{MO}^c(\bx,\,\cm)}\lesssim
\|h\|_{L_{p'}(\lm;\,L_2^c(\bx\times \mathbb{R}_+,\,d\mu(y)dt))}.
\end{align*}
This completes the proof of Theorem \ref{t4.1} (ii) and hence of Theorem \ref{t4.1}.
\end{proof}
\begin{definition}
Let $p\in(1,\,\fz)$. We say that a $L_p(\cm)$-valued measure $dm$ on $\bx\times\rr_+$ is a {\it{$p$-Carleson measure}}, if
\begin{eqnarray*}
\lf\|\sup_{x\in B\subset\bx}^{\ \ \ \ \  +}\frac{1}{\mu(B)}\int_{ T(B)}dm(x,\,t)\r\|_{L_{p}(\lm)}<\infty,
\end{eqnarray*}
where the supremum is taken over all balls $B\subset\bx$ such that $B\ni x$.
\end{definition}
We have the following $p$-Carleson measure characterization theorem.
\begin{theorem}\label{c4.1}
Let $p\in(2,\,\infty)$. Then the following conditions are equivalent:
\begin{enumerate}
\item[\rm{(i)}] $g\in L_p\mathcal{MO}^c(\bx,\,\cm)$;
\item[\rm{(ii)}]
$dm_g=|\mathbf{D}_t(g)(x)|^2\,d\mu(x)\frac{dt}{t}$ is a $p$-Carleson measure on $\bx\times\rr_+$.
\end{enumerate}
Moreover, there exists a positive constant $C$ such that
$$\frac{1}{C}\lf\|\sup_{x\in B\subset\bx}^{\ \ \ \ \  +}\mu_{g,\,B}^{\sharp}\r\|_{L_{\frac{p}{2}}(\lm)}
\leq\|g\|_{L_p\mathcal{MO}^c(\bx,\,\cm)}^2
\leq{C}\lf\|\sup_{x\in B\subset\bx}^{\ \ \ \ \  +}\mu_{g,\,B}^{\sharp}\r\|_{L_{\frac{p}{2}}(\lm)}.$$
Similar result also holds for the row space.
\end{theorem}
\begin{proof}
By repeating the proof of Theorem \ref{t4.1}, we know from inequalities \eqref{jkl1}, \eqref{jkl2} and \eqref{jkl3} that $g$ induces a linear functional $\mathcal{L}_g$ with the form \eqref{formmm}, initially defined on the set of all $S_{\cm}$-valued simple functions, satisfying
\begin{align}\label{iop00}
\lf|\cl_g(f)\r|^2&\leq C(\mathrm{L}_1 \times \mathrm{L}_2),
\end{align}
where $\mathrm{L}_1\lesssim \|f\|_{\mathcal{H}_p^{c}(\bx,\,\cm)}^p$ and
\begin{align*}
\mathrm{L}_{2}&\lesssim\lf\|\sum_{j\in\zz}\mathbb{V}(\cdot,\,j)\r\|
_{L_{(\frac{p'}{2})'}(\lm)}\lf\|\sup_{j\in\zz}^{\ \ \ \ \  +}
\int_{0}^{2C_4c_\delta\delta^{j-1}}\fint_{B_d(x,\,4C_4c_\delta
\delta^{j-1})}\lf|\mathbf{D}_t(g)(z)\r|^2
\,\frac{d\mu(z)dt}{t}\r\|_{L_{\frac{p'}{2}}(\lm)}\\
&\lesssim\lf\|\lf[\mathbb{S}_{\mathbf{D}}^c(\cdot,\,+\infty)\r]^{2-p}\r
\|_{L_{(\frac{p'}{2})'}(\lm)}
\lf\|\sup_{x\in B\subset\bx}^{\ \ \ \ \  +}\mu_{g,\,B}^{\sharp}\r\|_{L_{\frac{p'}{2}}(\lm)}^2\nonumber\\
&\lesssim\|f\|_{\mathcal{H}_p^{c}(\bx,\,\cm)}
^{2-p}\lf\|\sup_{x\in B\subset\bx}^{\ \ \ \ \  +}\mu_{g,\,B}^{\sharp}\r\|_{L_{\frac{p'}{2}}(\lm)}^2.\nonumber
\end{align*}

Substituting the estimates of $\mathrm{L}_1$ and $\mathrm{L}_2$ into inequality \eqref{iop00}, we conclude that
$$|\cl_g(f)|\lesssim \lf\|\sup_{x\in B\subset\bx}^{\ \ \ \ \  +}\mu_{g,\,B}^{\sharp}\r\|_{L_{\frac{p'}{2}}(\lm)}\|f\|_{\mathcal{H}_p^{c}(\bx,\,\cm)}.$$
This, in combination with Theorem \ref{t4.1}, completes the proof of Theorem \ref{c4.1}.
\end{proof}

\bigskip

\section{Relation between $H_p$ and $L_pMO$}\label{s5}
\setcounter{equation}{0}

In this section, we establish the equivalence between $L_p\mathcal{MO}^{\dagger}(\bx,\,\cm)$
 and $\mathcal{H}_p^{\dagger}(\bx,\,\cm)$ for $p\in(2,\,\infty)$ and $\dagger\in \{c,r,cr\}$.
\begin{theorem}\label{t5.1}
For any $p\in(2,\,\infty)$ and $\dagger\in\{c,r,cr\}$, one has
$$\mathcal{H}_p^{\dagger}(\bx,\,\cm)=L_p\mathcal{MO}^\dagger(\bx,\,\cm)$$
with equivalent norms.
\end{theorem}
To prove Theorem \ref{t5.1}, we need the following lemma, which can be regarded as an operator-valued analogue of the {\it Hardy-Littlewood maximal inequality}.
\begin{lemma}\rm{\cite[Theorem 4.1]{hlw21}}\label{l5.x}
Assume that $(\bx,\,d,\,\mu)$ be a space of homogeneous type.
Then for any $p\in(1,\,\fz)$, there exists a $C>0$ such that for any $f\in L_p(\lm)$,
$$\lf\|\{\mathfrak{M}_r(f)\}_{r>0}\r\|_{L_p(\lm,\,\ell_{\fz})} \leq C\|f\|_{L_p(\lm)},$$
where
$$\mathfrak{M}_r(f)(x):=\fint_{B(x,\,r)}f(y)\,d\mu(y),\ \ x\in\bx\ \ {\rm and}\ \ r>0.$$
\end{lemma}

\begin{proof}[Proof of Theorem \ref{t5.1}] We divide the proof into two steps.

\textbf{Step 1:}
In this step, we show that $\mathcal{H}_p^{c}(\bx,\,\cm)\subset L_p\mathcal{MO}^c(\bx,\,\cm)$ for $p\in(2,\,\infty)$. By Theorem \ref{t4.1}, for any
$f\in\mathcal{H}_p^{c}(\bx,\,\cm)$,
\begin{align*}
\|f\|_{L_p\mathcal{MO}^c(\bx,\,\cm)}&=\sup_{\|g\|_{\mathcal{H}_{p'}^{c}(\bx,\,\cm)}\leq1}
\lf|\tau\int_{\bx}g(x)f^{\ast}(x)\,d\mu(x)\r|\\
&=\sup_{\|g\|_{\mathcal{H}_{p'}^{c}(\bx,\,\cm)}\leq1}
\lf|\tau\int_{\bx}\int_0^{\fz}\int_{B_d(x,\,t)}\mathbf{D}_t(g)(y)
\lf[\mathbf{D}_t(f)(y)\r]^{*}\,\frac{d\mu(y)dt}{V_t(x)t}d\mu(x)\r|\\
&\leq\sup_{\|g\|_{\mathcal{H}_{p'}^{c}(\bx,\,\cm)}\leq1}
\lf[\tau\int_{\bx}\lf(\int_0^{\fz}\int_{B_d(x,\,t)}
\lf|\mathbf{D}_t(f)(y))\r|^2\,\frac{d\mu(y)dt}{V_t(x)t}\r)^{p/2}d\mu(x)\r]^{1/p}\\
&\ \ \ \times\lf[\tau\int_{\bx}\lf(\int_0^{\fz}\int_{B_d(x,\,t)}
\lf|\mathbf{D}_t(g)(y))\r|^2\,\frac{d\mu(y)dt}{V_t(x)t}\r)^{p'/2}d\mu(x)\r]^{1/p'}\\
&\leq\|f\|_{\mathcal{H}_p^{c}(\bx,\,\cm)}.
\end{align*}
This finishes the proof of \textbf{Step 1}.

\textbf{Step 2:} In this step, we prove that
$L_p\mathcal{MO}^c(\bx,\,\cm)\subset\mathcal{H}_p^{c}(\bx,\,\cm)$. To prove this, we need to introduce the {\it tent space}. For any $f\in L_p(\lm;\,L_2^c(\bx\times\rr_+,\,\frac{d\mu(y)dt}{V_t(y)t}))$ with $p\in(1,\,\infty)$, we define
$$\mathcal{A}^c(f)(x):=\lf(\iint_{\Gamma_x}|f(y,\,t)|^2\,\frac{d\mu(y)dt}{V_t(x)t}\r)^{1/2}.$$
Then the tent space $T_p^c(\bx,\,\cm)$ is defined by setting
$$T_p^c(\bx,\,\cm):=\lf\{f\in L_p(\mathcal{N};\,L_2^c(\bx\times\rr_+,\,\frac{d\mu(y)dt}{V_t(y)t}))
:\|f\|_{T_p^c(\bx,\,\cm)}<\infty\r\},
$$
equipped with the norm
$\|f\|_{T_p^c(\bx,\,\cm)}:=\lf\|\mathcal{A}^c(f)\r\|_{L_p(\lm)}.$

Now we claim that for any $g\in L_p\mathcal{MO}^c(\bx,\,\cm)$ with $p\in(2,\,\infty)$, $g$ induces a continuous linear functional $\cl_g$ on $T^c_{p'}(\bx,\, \cm)$ by
\begin{align}\label{e5.1}
\cl_g(f):=\tau\int_0^{\infty}\int_{\bx}f(y,\,t)(\mathbf{D}_t(g)(y))^*\,d\mu(y)\frac{dt}{t}
\end{align}
and
\begin{align}\label{e5.2}
\|g\|_{\mathcal{H}_p^{c}(\bx,\, \cm)}\lesssim\|\cl_g\|
\lesssim\|g\|_{L_p\mathcal{MO}^c(\bx,\,\cm)}.
\end{align}
To prove this claim, we define
$$\mathcal{A}^c(f)(x,\,t):=\lf(\int_t^{\infty}
\int_{B_d(x,\,s-\frac{t}{2})}|f(y,\,s)|^2\,
\frac{d\mu(y)ds}{V_{\frac{s}{2}}(x)s}\r)^{1/2}.$$
By approximation, we assume that $\mathcal{A}^c(f)(x,\,t)$ is invertible for any $(x,\,t)\in\bx\times\rr_+$. Then, by the Cauchy-Schwarz inequality, we have
\begin{align*}
\lf|\cl_g(f)\r|^2&=\lf|\tau\int_0^{\fz}\int_{\bx}f(y,\,t)
(\mathbf{D}_t(g)(y))^*\,d\mu(y)\frac{dt}{t}\r|^2\\
&=C\lf|\tau\int_{\bx}\int_0^{\fz}\int_{B_d(x,\,\frac{t}{2})}
f(y,\,t)(\mathbf{D}_t(g)(y))^*\,\frac{d\mu(y)dt}{V_{\frac{t}{2}}(x)t}d\mu(x)\r|^2\\
&\leq C\tau\int_{\bx}\int_0^{\fz}\lf[\mathcal{A}^c(f)(x,\,t)\r]^{p'-2}
\int_{B_d(x,\,\frac{t}{2})}
|f(y,\,t)|^2\,\frac{d\mu(y)dt}{V_{\frac{t}{2}}(x)t}d\mu(x)\\
&\ \ \ \ \times\tau\int_{\bx}\int_0^{\fz}\lf[\mathcal{A}^c(f)(x,\,t)\r]^{2-p'}
\int_0^{\fz}\int_{B_d(x,\,\frac{t}{2})}
|\mathbf{D}_t(g)(y)|^2\,\frac{d\mu(y)dt}{V_{\frac{t}{2}}(x)t}d\mu(x)\\
&=:C(\mathrm{P}_1\times\mathrm{P}_2).
\end{align*}

For the terms $\mathrm{P}_1$ and $\mathrm{P}_2$, similar to the estimates of
$\mathrm{L}_1$ and $\mathrm{L}_2$ in the proof of Theorem \ref{t4.1}, but with $S_{\mathbf{D}}^c(f)(x,\,t)$ being replaced by $\mathcal{A}^c(f)(x,\,t)$, we conclude that
\begin{align*}
\mathrm{P}_1\lesssim\|f\|_{T_{p'}^c(\bx, \,\cm)}^{p'} \ \ \mathrm{and} \ \
\mathrm{P}_2\lesssim\|f\|_{T_{p'}^c(\bx, \,\cm)}^{2-p'}\|g\|
_{L_p\mathcal{MO}^c(\bx,\,\cm)}^2,
\end{align*}
which implies that
\begin{align}\label{e5.3}
\|\cl_g\|\lesssim\|g\|_{L_p\mathcal{MO}^c(\bx,\,\cm)}.
\end{align}

Next we show that $\|g\|_{\mathcal{H}_p^{c}(\bx,\,\cm)}\lesssim\|\cl_g\|$. By the definition of
 $T_{p'}^c(\bx,\, \cm)$, we can identify $T_{p'}^c(\bx,\, \cm)$
 as a subspace of $L_{p'}(\lm;\,L_2^c(\bx\times\rr_+,\,\frac{d\mu(y)dt}{V_t(y)t}))$ via the following map:
 \begin{align*}
\overline{\Phi}:f(y,\,t)\mapsto f(y,\,t)\chi_{\Gamma_x}(y,\,t).
\end{align*}
Therefore, by the Hahn-Banach theorem, $\cl_g$  can be extended to a continuous linear functional on
$L_{p'}(\lm;\,L_2^c(\bx\times\rr_+,\,\frac{d\mu(y)dt}{V_t(y)t}))$ with the same norm. Then, by the fact that
$$(L_{p'}(\lm;\,L_2^c(\bx\times\rr_+,\,\frac{d\mu(y)dt}{V_t(y)t})))^{*}
=L_{p'}(\lm;\,L_2^c(\bx\times\rr_+,\,\frac{d\mu(y)dt}{V_t(y)t})),$$
there exists $h\in L_p(\lm;\,L_2^c(\bx\times\rr_+,\,\frac{d\mu(y)dt}{V_t(y)t}))$ such that
$$\|h\|_{L_p(\lm;\,L_2^c(\bx\times\rr_+,\,\frac{d\mu(y)dt}{V_t(y)t}))}=\|\cl_g\|$$
and that for any $f\in T^c_{p'}(\bx,\, \cm)$,
\begin{align*}
\cl_g(f)&=\tau\int_{\bx}\iint_{\Gamma_x}f(y,\,t)h^{*}(x,\,y,\,t)\,
\frac{d\mu(y)dt}{V_t(y)t}d\mu(x)\\
&=\tau\int_{\bx}\int_0^{\fz}f(y,\,t)\int_{B_d(y,\,t)}h^{*}(x,\,y,\,t)\,
\frac{d\mu(x)dt}{V_t(y)t}d\mu(y).
\end{align*}
This, combined with \eqref{e5.1}, indicates that
\begin{align*}
\mathbf{D}_t(g)(y)=\int_{B_d(y,\,t)}h(x,\,y,\,t)\,\frac{d\mu(x)}{V_t(y)}.
\end{align*}
By Lemma \ref{e2.5}, we have
\begin{align}\label{e5.4}
\|g\|_{\mathcal{H}_p^{c}(\bx,\,\cm)}
&=\lf[\tau\int_{\bx}\lf(\iint_{\Gamma_z}
\lf|\mathbf{D}_t(g)(y)\r|^2\,\frac{d\mu(y)dt}{V_t(z)t}\r)^{p/2}\,d\mu(z)\r]^{1/p}\\\nonumber
&=\lf[\tau\int_{\bx}\lf(\iint_{\Gamma_z}
\lf|\int_{B_d(y,\,t)}h(x,\,y,\,t)\,\frac{d\mu(x)}{V_t(y)}\r|^2\,
\frac{d\mu(y)dt}{V_t(z)t}\r)^{p/2}\,d\mu(z)\r]^{1/p}\\\nonumber
&\lesssim\lf[\tau\int_{\bx}\lf(\int_0^{\fz}\int_{B_d(z,\,t)}
\int_{B_d(y,\,t)}\lf|h(x,\,y,\,t)\r|^2\,\,\frac{d\mu(x)}{V_t(y)}\,
\frac{d\mu(y)dt}{V_t(z)t}\r)^{p/2}\,d\mu(z)\r]^{1/p}\\\nonumber
&\lesssim\lf[\tau\int_{\bx}\lf(\int_0^{\fz}\int_{\bx}
\int_{B_d(z,\,2t)}\lf|h(x,\,y,\,t)\r|^2\,
\frac{d\mu(x)}{V_t(y)}\,\frac{d\mu(y)dt}{V_t(z)t}\r)^{p/2}\,d\mu(z)\r]^{1/p}\\\nonumber
&\thicksim\lf\|\int_0^{\fz}\int_{\bx}
\int_{B_d(z,\,2t)}\lf|h(x,\,y,\,t)\r|^2\,\frac{d\mu(x)}{V_t(y)}\,\frac{d\mu(y)dt}{V_t(z)t}\r\|
_{L_{\frac{p}{2}}(\lm)}^{1/2}.\nonumber
\end{align}
It follows from \eqref{e2.1}, \eqref{e4.0}, Fubini's theorem and Lemma \ref{l5.x} that for any positive $a\in L_{(\frac{p}{2})'}(\lm)$ with
$\|a\|_{L_{(\frac{p}{2})'}(\lm)}\leq1$,
\begin{align*}
&\tau\int_{\bx}\int_0^{\fz}\int_{\bx}
\int_{B_d(z,\,2t)}\lf|h(x,\,y,\,t)\r|^2\,
\frac{d\mu(x)}{V_t(y)}\,\frac{d\mu(y)dt}{V_t(z)t}\,a(z)\,d\mu(z)\\
\lesssim&\tau\int_{\bx}\int_0^{\fz}\int_{\bx}
\lf|h(x,\,y,\,t)\r|^2\,\frac{d\mu(y)dt}{V_t(y)t}\,\frac{1}{V_{2t}(x)}\int_{B_d(x,\,2t)}a(z)\,d\mu(z)d\mu(x)\\
\lesssim&\lf\|\int_0^{\fz}\int_{\bx}
\lf|h(x,\,y,\,t)\r|^2\,\frac{d\mu(y)dt}{V_t(y)t}\r\|_{L_{\frac{p}{2}}(\lm)}
\lf\|\sup_{x\in B\subset\bx}^{\ \ \ \ \ +}\frac{1}{\mu(B)}\int_{B}a(z)\,d\mu(z)\r\|
_{L_{(\frac{p}{2})'}(\lm)}\\
\lesssim&\lf\|h\r\|_{L_{p}(\lm;\,L_2^c(\bx\times\rr_+,\,\frac{d\mu(y)dt}{V_t(y)t}))}^2\\
\thicksim&\lf\|\cl_g\r\|^2.
\end{align*}
From this, \eqref{e5.4} and taking the supremum for all $a$ as above, we conclude that
$$\|g\|_{\mathcal{H}_p^{c}(\bx,\,\cm)}\lesssim\lf\|\cl_g\r\|,$$
which, combined with \eqref{e5.3}, implies that claim \eqref{e5.2} holds true.

Finally, to complete the proof of $\mathcal{H}_p^{c}(\bx,\,\cm)=L_p\mathcal{MO}^c(\bx,\,\cm)$, we only need to show that the set of
$S_{\cm}$-valued simple functions is dense in $L_p\mathcal{MO}^c(\bx,\,\cm)$. Indeed, from the proof of Theorem \ref{t4.1}, we see that for any $g\in L_p\mathcal{MO}^c(\bx,\,\cm)$, there
exists $h\in L_p(\lm;\,L_2^c(\bx\times\rr_+,\,d\mu(y)dt))$ such that
$$g=\Psi(h)\ \ \ \ \mathrm{and} \ \ \ \ \|\Psi(h)\|_{L_p\mathcal{MO}^c(\bx,\,\cm)}
\lesssim\|h\|_{L_p(\lm;\,L_2^c(\bx\times\rr_+,\,d\mu(y)dt))}.$$
By a density argument, we can choose a sequence of $S_{\lm}$-valued simple functions $\{h_i\}_{i\in\nn}$ such that
$h_i\rightarrow h$ in $L_p(\lm;\,L_2^c(\bx\times\rr_+,\,d\mu(y)dt))$, as $i\rightarrow\infty$. Let $g_i=\Psi(h_i)$. Then we deduce that, for each $i$, $g_i$ is an  $S_{\cm}$-valued simple function
and that
$$g_i\rightarrow g\ \ \ \mathrm{in} \ \ L_p\mathcal{MO}^c(\bx,\,\cm),\ \ \mathrm{as} \ \ i\rightarrow\infty.$$
This shows that the set of
$S_{\cm}$-valued simple functions is dense in $L_p\mathcal{MO}^c(\bx,\,\cm)$ and
therefore, concludes the proof of Theorem \ref{t5.1}.
\end{proof}
\begin{corollary}\label{io44}
Let $\Psi$ be the map given in Lemma \ref{d3.1}. Then for any $p\in(1,\,\infty)$,
 $\Psi$ is bounded from $L_{p}(\lm;\,L_2^c(\bx\times \mathbb{R}_+,\,d\mu(y)dt))$ to $\mathcal{H}_{p}^c(\bx,\,\cm)$.
\end{corollary}
\begin{proof}
By Lemma \ref{l4.1} and Theorem \ref{t5.1}, it remains to prove  the conclusion holds for $p\in(  1,\,2)$. To show this, we apply Theorems \ref{t4.1} and \ref{t5.1} to deduce that
\begin{align*}
\lf\|\Psi(h)\r\|_{\mathcal{H}_{p}^c(\bx,\,\cm)}
=&\sup_{\|g\|_{\mathcal{H}_{p'}^c(\bx,\,\cm)}\leq1}\lf|\tau\int_{\bx}
\Psi(h)(z)g^*(z)\,d\mu(z)\r|\\
=&\sup_{\|g\|_{\mathcal{H}_{p'}^c(\bx,\,\cm)}\leq1}\lf|\tau\int_{\bx}
\int_{\bx}\iint_{\Gamma_x}h(x,\,y,\,t)
\mathbf{D}_t(z,\,y)\,\frac{d\mu(y)dt}{[V_t(x)t]^{1/2}}d\mu(x)g^*(z)\,d\mu(z)\r|\\
=&\sup_{\|g\|_{\mathcal{H}_{p'}^c(\bx,\,\cm)}\leq1}\lf|\tau
\int_{\bx}\int_0^{\fz}\int_{B_d(x,\,t)}h(x,\,y,\,t)
\lf[\mathbf{D}_t(g)(y)\r]^*\,\frac{d\mu(y)dt}{[V_t(x)t]^{1/2}}d\mu(x)\r|\\
\leq&\sup_{\|g\|_{\mathcal{H}_{p'}^c(\bx,\,\cm)}\leq1}\lf\{\tau
\int_{\bx}\lf[\iint_{\Gamma_x}|h(x,\,y,\,t)|^2
\,d\mu(y)dt\r]^{p/2}d\mu(x)\r\}^{1/p}\\
&\ \ \times \lf\{\tau
\int_{\bx}\lf[\iint_{\Gamma_x}
\lf|\mathbf{D}_t(g)(y)\r|^2\,\frac{d\mu(y)dt}{V_t(x)t}\r]^{p'/2}d\mu(x)\r\}^{1/p'}\\
\leq&\|h\|_{L_{p}(\lm;\,L_2^c(\bx\times \mathbb{R}_+,\,d\mu(y)dt))}.
\end{align*}
This finishes the proof of Corollary \ref{io44}.
\end{proof}
\begin{corollary}\label{corrrrr}
Let $p\in(1,\,\infty)$ and $\dagger\in\{c,r,cr\}$. Then
$$(\mathcal{H}_{p}^\dagger(\bx,\,\cm))^*\backsimeq\mathcal{H}_{p'}^\dagger(\bx,\,\cm)$$
with equivalent norms.
\end{corollary}
\begin{proof}[Proof of Corollary \ref{corrrrr}]
From Theorems \ref{t4.1} and \ref{t5.1}, it remains to consider the case $p\in(2,\,\fz)$. To this end, we recall from \eqref{e3.1} that for any $S_{\cm}$-valued simple function $f$ and $g\in\mathcal{H}_{p'}^{c}(\bx,\,\cm)$,
\begin{align*}
\lf|\cl_g(f)\r|^2
&=\lf|\tau\int_{\bx}\int_0^{\fz}\int_{B_d(x,\,t)}\mathbf{D}_t(f)(z)\lf[\mathbf{D}_t(g)(z)\r]^{*}
\,\frac{dtd\mu(z)}{V_{t}(z)t}d\mu(x)\r|^2 \lesssim\|f\|_{\mathcal{H}_{p}^c(\bx,\,\cm)}^2\|g\|_{\mathcal{H}_{p'}^c(\bx,\,\cm)}^2,
\end{align*}
which implies that $\mathcal{H}_{p'}^c(\bx,\,\cm)\subset(\mathcal{H}_{p}^c(\bx,\,\cm))^*$ for $p\in(2,\,\fz)$. To show the converse direction, we repeat the steps in the proof of Theorem \ref{t4.1} (ii), but we use Corollary \ref{io44} instead of Lemma \ref{l4.1} to see that $g:=\Psi(h)$ satisfies
$
\|g\|_{\mathcal{H}_{p'}^c(\bx,\,\cm)}
\lesssim\|\cl\|_{(\mathcal{H}_p^{c}(\bx,\,\cm))^{\ast}}.
$
This ends the proof of Corollary \ref{corrrrr}.
\end{proof}
\begin{remark}{\rm
It follows from Theorems \ref{t3.2}, \ref{t4.1} and \ref{t5.1} that for any $p\in[1,\,\infty)$ and $\dagger\in\{c,r,cr\}$, the definition of $\mathcal{H}_{p}^\dagger(\bx,\,\cm)$
is independent of the choice of $\{\mathbf{D}_t\}_{t>0}$, as long as $\{\mathbf{D}_t\}_{t>0}$ satisfies the assumptions $(\mathbf{H}_1)$--$(\mathbf{H}_4)$.}
\end{remark}

\bigskip

\section{Relation between $H_p$ and $L_p$}\label{s6}
\setcounter{equation}{0}

This section is devoted to establishing the equivalence between the mixture space $\mathcal{H}_p^{cr}(\bx,\,\cm)$ and $L_p(\mathcal{N}) $, where $p\in(1,\,\infty)$. To begin with, we let $\{\mathcal{Q}^{(i)}_k\}_{k\in\zz}$ be the sequence of partitions of $\bx$ given in Lemma \ref{l6.1} and then define the dyadic mean oscillation of $f$ with respect to $\mathcal{Q}^{(i)}_{k}$ by
$$f_{\mathcal{Q}^{(i)}_{k}}^{\sharp}(x)
:=\frac{1}{\mu({Q}^{(i)}_{k,x})}\int_{{Q}^{(i)}_{k,x}}\lf|f(y)-f_{{Q}^{(i)}_{k,x}}\r|^2\,d\mu(y),\ {\rm for}\ {\rm any}\ x\in\bx,$$
where ${Q}^{(i)}_{k,x}\in \mathcal{Q}^{(i)}_{k}$ denotes the unique cube containing $x$.
\begin{definition}
Let $p\in(2,\,\infty]$. We define the dyadic column $L_p\mathcal{MO}$ space
$$L_p\mathcal{MO}^{c,\mathcal{Q}^{(i)}}(\lm)
:=\lf\{f\in  L_p(\cm;\,L_2^{\mathrm{loc},c}(\bx)
:\|f\|_{L_p\mathcal{MO}^{c,\mathcal{Q}^{(i)}}(\lm)}<\infty\r\},$$
where
$$\|f\|_{L_p\mathcal{MO}^{c,\mathcal{Q}^{(i)}}(\lm)}:=
\lf\|\sup_k^{\ \ \ \ \ +}f_{\mathcal{Q}^{(i)}_{k}}^{\sharp}\r\|_{L_{\frac{p}{2}}
(\lm)}^{1/2}.$$
In a similar way as extending the definition of column BMO space to the row and mixture ones, one can define
the row space $L_p\mathcal{MO}^{r,\mathcal{Q}^{(i)}}(\lm)$ and mixture space $L_p\mathcal{MO}^{cr,\mathcal{Q}^{(i)}}(\lm)$, respectively.

\end{definition}

\begin{lemma}\label{l6.2}
Let $\dagger\in\{c,r,cr\}$ and $p\in(2,\,\infty]$. Then we have
$$L_p\mathcal{MO}^\dagger(\bx,\,\cm)=\bigcap_{i=1}^{I_0} L_p\mathcal{MO}^{\dagger,\mathcal{Q}^{(i)}}(\lm)$$
with equivalent norms.
\end{lemma}
\begin{proof}
By Lemma \ref{l6.1}, for any ball $B\subset\bx$, there exist $i\in[1,\,I_0]\cap\nn$, $k\in\zz$ and $Q\in \mathcal{Q}_{k}^{(i)}$ such that
$$B\subset Q \ \ \mathrm{and} \ \ \mu(Q)\lesssim\mu(B).$$
Then
\begin{align*}
\frac{1}{\mu(B)}\int_{B}\lf|f(y)-f_{B}\r|^2\,d\mu(y)
&\lesssim\frac{1}{\mu(B)}\int_{B}\lf|f(y)-f_{Q}\r|^2\,d\mu(y)+
\lf|f_{B}-f_{Q}\r|^2\\
&\lesssim\frac{1}{\mu(B)}\int_{B}\lf|f(y)-f_{Q}\r|^2\,d\mu(y)\\
&\lesssim\frac{1}{\mu(Q)}\int_{Q}\lf|f(y)-f_{Q}\r|^2\,d\mu(y).
\end{align*}
Therefore,
\begin{align*}
\|f\|_{L_p\mathcal{MO}^c(\bx,\,\cm)}
\lesssim\lf\|\sup_k^{\ \ \ \ \ +}\sum_{i=1}^{I_0} f_{\mathcal{Q}^{(i)}_k}^{\sharp}\r\|_
{L_{\frac{p}{2}}(\lm)}^{1/2}\lesssim\max_{1\leq i\leq I_0}\lf\{\|f\|_{L_p\mathcal{MO}^{c,\mathcal{Q}^{(i)}}(\lm)}\r\},
\end{align*}
which implies that $$\bigcap_{i=1}^{I_0} L_p\mathcal{MO}^{c,\mathcal{Q}^{(i)}}(\lm)\subset L_p\mathcal{MO}^c(\bx,\,\cm).$$

Next we show that $$L_p\mathcal{MO}^c(\bx,\,\cm)\subset \bigcap_{i=1}^{I_0} L_p\mathcal{MO}^{c,\mathcal{Q}^{(i)}}(\lm).$$
For any $1\leq i\leq I_0$, if $Q\in \mathcal{Q}_k^{(i)}$, then we apply Lemma \ref{l6.1} to choose a ball $B\subset\bx$ such that $Q\subset B$ and $\mu(B)\lesssim\mu(Q)$. Therefore,
\begin{align*}
\frac{1}{\mu(Q)}\int_{Q}\lf|f(y)-f_{Q}\r|^2\,d\mu(y)
\lesssim&\frac{1}{\mu(Q)}\int_{Q}\lf|f(y)-f_{B}\r|^2\,d\mu(y)+
\lf|f_{B}-f_{Q}\r|^2\\
\lesssim&\frac{1}{\mu(Q)}\int_{Q}\lf|f(y)-f_{B}\r|^2\,d\mu(y)\\
\lesssim&\frac{1}{\mu(B)}\int_{B}\lf|f(y)-f_{B}\r|^2\,d\mu(y).
\end{align*}
From this, we conclude that
$$\max_{1\leq i\leq I_0}\lf\{\|f\|_{L_p\mathcal{MO}^{c,\mathcal{Q}^{(i)}}(\lm)}\r\}\lesssim
\|f\|_{L_p\mathcal{MO}^c(\bx,\,\cm)}.$$
This completes the proof of  Lemma \ref{l6.2}.
\end{proof}
By the duality theory of the non-commutative martingale Hardy spaces (see \cite{jx03,px97}), we obtain the following conclusion immediately.
\begin{lemma}\label{l6.3}
Let $p\in(2,\,\infty)$. Then
$$L_p\mathcal{MO}^{cr}(\bx,\,\cm)=L_p(\lm)$$
with equivalent norms.
\end{lemma}
\begin{proof}
From \cite[(4.5) and (4.7)]{jx03}, we conclude that, for any $1\leq i\leq I_0$,
$$L_p\mathcal{MO}^{c,\mathcal{Q}^{(i)}}(\lm)\cap
L_p\mathcal{MO}^{r,\mathcal{Q}^{(i)}}(\lm)=L_p(\lm)$$
with equivalent norms. This, in combination with Lemma \ref{l6.2}, yields
\begin{align*}
L_p\mathcal{MO}^{cr}(\bx,\,\cm)
&=L_p\mathcal{MO}^c(\bx,\,\cm)\cap L_p\mathcal{MO}^r(\bx,\,\cm)\\
&=\bigcap_{i=1}^{I_0}\lf[L_p\mathcal{MO}^{c,\mathcal{Q}^{(i)}}(\lm)\cap L_p\mathcal{MO}^{r,\mathcal{Q}^{(i)}}(\lm)\r]\\
&=L_p(\lm)
\end{align*}
with equivalent norms. This ends the proof of Lemma \ref{l6.3}.
\end{proof}

\begin{lemma}\label{l6.3}
Let
$p\in(1,\,2)$. Then
$$\mathcal{H}_p^{c}(\bx,\,\cm)=\sum_{i=1}^{I_0} \mathcal{H}_p^{c,\mathcal{Q}^{(i)}}(\lm).$$
Let
$p\in(2,\,\infty)$. Then
$$\mathcal{H}_p^{c}(\bx,\,\cm)=\bigcap_{i=1}^{I_0} \mathcal{H}_p^{c,\mathcal{Q}^{(i)}}(\lm)$$
with equivalent norms.
Similar results also hold true for $\mathcal{H}_p^{r}(\bx,\,\cm)$ and $\mathcal{H}_p^{cr}(\bx,\,\cm)$.
\end{lemma}

\begin{theorem}\label{t6.1}
Let $p\in(1,\,\infty)$. Then we have
$$\mathcal{H}_p^{cr}(\bx,\, \cm)=L_p(\lm)$$
with equivalent norms.
\end{theorem}
\begin{proof}[Proof of Theorem \ref{t6.1}]
Recall from \cite{jx03} that
$$\mathcal{H}_p^{cr,\mathcal{Q}^{(i)}}(\lm)
=L_p(\lm).$$  This, in combination with Lemma \ref{l6.3}, yields that when $1<p<2$,
\begin{align*}
\mathcal{H}_p^{cr}(\bx,\, \cm)=\sum_{i=1}^{I_0}\mathcal{H}_p^{c,\mathcal{Q}^{(i)}}(\lm)+
\sum_{i=1}^{I_0}\mathcal{H}_p^{r,\mathcal{Q}^{(i)}}(\lm)=\sum_{i=1}^{I_0}\mathcal{H}_p^{cr,\mathcal{Q}^{(i)}}(\lm)
=L_p(\lm).
\end{align*}
Moreover, when $p\in[2,\,\infty)$,
\begin{align*}
\mathcal{H}_p^{cr}(\bx,\, \cm)
=\lf[\bigcap_{i=1}^{I_0}\mathcal{H}_p^{c,\mathcal{Q}^{(i)}}(\lm)\r]
\bigcap
\lf[\bigcap_{i=1}^{I_0}\mathcal{H}_p^{r,\mathcal{Q}^{(i)}}(\lm)\r]=\bigcap_{i=1}^{I_0}\mathcal{H}_p^{cr,\mathcal{Q}^{(i)}}(\lm)
=L_p(\lm).
\end{align*}
This completes the proof of Theorem \ref{t6.1}.
\end{proof}
\bigskip
\section{Interpolation}\label{s7}
\setcounter{equation}{0}

In this section, we shall establish two
interpolation theorems between the operator-valued BMO spaces and Hardy spaces.

The first main result of this subsection is the following.
\begin{theorem}\label{t7.1}
Let $1\leq q<p<\infty$. Then
\begin{align}\label{e7.1}
\lf[\mathcal{BMO}^{c}(\bx,\,\cm),\,
\mathcal{H}_q^{c}(\bx,\,\cm)\r]_{\frac{q}{p}}=
\mathcal{H}_p^{c}(\bx,\,\cm)
\end{align}
and
\begin{align}\label{e7.2}
\lf[\mathcal{X},\,\mathcal{Y}\r]_{\frac{1}{p}}=
L_p(\lm),
\end{align}
where
$\mathcal{X}=\mathcal{BMO}^{cr}(\bx,\,\cm)$ or $L_{\infty}(\lm)$,
$\mathcal{Y}=\mathcal{H}_{1}^{cr}(\bx,\,\cm)$ or $L_1(\lm)$.
\end{theorem}
To prove the main result, we need some technical lemmas as follows. The following lemma is just the complex interpolation of
the column and row non-commutative $L_p$-spaces.
\begin{lemma}{\rm\cite[p.9]{m07}}\label{l7.0}
Let $0<\theta<1$ and $1\leq p_0,\,p_1,\,p_{\theta}\leq\fz$ with
$\frac{1}{p_{\theta}}=\frac{1-\theta}{p_0}+\frac{\theta}{p_1}$. Then we have
$$ \lf[L_{p_0}(\cm;\,L_2^c(\Omega)),\,L_{p_1}(\cm;\,L_2^c(\Omega))\r]_{\theta}
=L_{p_{\theta}}(\cm;\,L_2^c(\Omega)).$$
Similar result also holds for row spaces.
\end{lemma}

The following lemma is the Wolff interpolation theorem.
\begin{lemma}{\rm{\cite{w94}}}\label{l7.0x}
Let $\mathcal{X}_0,\,\mathcal{X}_1, \,\mathcal{X}_2$ and $\mathcal{X}_3$ be  four Banach spaces such that $\mathcal{X}_0\cap \mathcal{X}_3$  is a
dense subspace of $\mathcal{X}_1$ and of $\mathcal{X}_2$. Suppose $$[\mathcal{X}_1,\,\mathcal{X}_3]_{\theta}=\mathcal{X}_2 \ \ and \ \
[\mathcal{X}_0,\,\mathcal{X}_2]_{\eta}=\mathcal{X}_1$$
with
$0<\theta,\,\eta<1$.
Then we have $[\mathcal{X}_0,\,\mathcal{X}_3]_{\widetilde{\theta}}=\mathcal{X}_1$,
$[\mathcal{X}_0,\,\mathcal{X}_3]_{\widetilde{\eta}}=\mathcal{X}_2$, where $\widetilde{\theta}=\frac{\eta\theta}{1-\eta+\eta\theta}$ and $\widetilde{\eta}=
\frac{\theta}{1-\eta+\eta\theta}$.
\end{lemma}

\begin{lemma}\label{l7.1}
If \eqref{e7.1} holds true for any $2< q<p<\infty$, then
\eqref{e7.1} also holds true for any $1\leq q<p<\infty$.
\end{lemma}
\begin{proof}

{\bf{Case 1}}: When $1<q<p<\fz$, we first consider the subcase $2<p<\fz$. Let $\theta\in(0,\,1)$ and $s=\frac{p}{\theta}$, then there exists a number $\phi\in(0,\,1)$ such that $\frac{1}{p}=\frac{1-\phi}{s}+\frac{\phi}{q}$.
For any $p\in(1,\,\fz)$, $\mathcal{H}_p^{c}(\bx,\,\cm)$ is complemented in  $L_{p}(\lm,\,L_2^c(\bx\times \mathbb{R}_+,\,d\mu(y)dt))$ (see Lemma \ref{d3.1}). It follows from Lemma \ref{l7.0} that
\begin{align}\label{immmmm}
\mathcal{X}_2:=\mathcal{H}_p^c(\bx,\,\cm)
=[\mathcal{H}_s^c(\bx,\,\cm),\,\mathcal{H}_q^c(\bx,\,\cm)]_{\phi}
=:[\mathcal{X}_1,\,\mathcal{X}_3]_{\phi}.
\end{align}
By the assumption $p>2$, we obtain
$$\mathcal{X}_1:=\mathcal{H}_s^c(\bx,\,\cm)
=[\mathcal{BMO}^c(\bx,\,\cm),\,\mathcal{H}_p^c(\bx,\,\cm)]_{\theta}
=:[\mathcal{X}_0,\,\mathcal{X}_2]_{\theta}.$$
From the above two equalities and Lemma \ref{l7.0x}, we deduce that
\begin{align}\label{sgdg}
\mathcal{X}_2=\mathcal{H}_p^c(\bx,\,\cm)
=[\mathcal{X}_0,\,\mathcal{X}_3]_{\widetilde{\theta}}
=[\mathcal{BMO}^c(\bx,\,\cm),\,\mathcal{H}_q^c(\bx,\,\cm)]_{\widetilde{\theta}}
\end{align}
with $\widetilde{\theta}:=\frac{\phi}{1-\theta-\theta\phi}=\frac{q}{p}$.
Next, we consider the subcase $1<p<2$. In this case, the proof can be shown via \eqref{sgdg}, a duality version of \eqref{immmmm} and then Lemma \ref{l7.0x}.

{\bf{Case 2}}: When $1=q<p<\fz$, let $s>\max\{2,\,p\}$. Then there exists a number $\theta\in(0,\,1)$ such that $\frac{1-\theta}{s}+\theta=\frac{1}{p}$. By the above conclusion and duality, we get
\begin{align*}
\mathcal{X}_2:=&\mathcal{H}_p^c(\bx,\,\cm)
=[\mathcal{H}_s^c(\bx,\,\cm),\,\mathcal{H}_1^c(\bx,\,\cm)]_{\theta}
=:[\mathcal{X}_1,\,\mathcal{X}_3]_{\theta}.
\end{align*}
Let $\phi=\frac{p}{s}$. It follows from {\bf{Case 1}} that
$$\mathcal{X}_1:=\mathcal{H}_s^c(\bx,\,\cm)
=[\mathcal{BMO}^c(\bx,\,\cm),\,\mathcal{H}_p^c(\bx,\,\cm)]_{\phi}
=:[\mathcal{X}_0,\,\mathcal{X}_2]_{\phi}.$$
From the above two equalities and Lemma \ref{l7.0x}, we deduce that
$$\mathcal{H}_p^c(\bx,\,\cm)=\mathcal{X}_2=[\mathcal{X}_0,\,\mathcal{X}_3]
_{\widetilde{\phi}}
=[\mathcal{BMO}^c(\bx,\,\cm),\,\mathcal{H}_1^c(\bx,\,\cm)]_{\widetilde{\phi}}
$$
with $\widetilde{\phi}:=\frac{\phi}{1-\theta-\theta\phi}=\frac{1}{p}$.
\end{proof}

\begin{lemma}{\rm \cite[Proposition 3.7]{m03}}\label{l7.1x0}
Let $2\leq p,\,q,\,s<\fz$ and $\theta\in(0,\,1)$ with $\frac{1}{s}=\frac{1-\theta}{p}+\frac{\theta}{q}$. Then
$$[L_p(\cm;\,\ell_{\fz}^c),\,L_q(\cm;\,\ell_{\fz}^c)]_{\theta}=L_s(\cm;\,\ell_{\fz}^c).$$
Similar result also holds for row spaces.
\end{lemma}

\begin{lemma}{\rm \cite[Theorem 4.2.2]{bl76}}\label{l7.1x}
Let $\theta\in(0,\,1)$ and $(\mathcal{X}_0,\,\mathcal{Y})$ be a
compatible couple of Banach spaces. Then
$\mathcal{X}\cap\mathcal{Y}$ is dense in $[\mathcal{X},\,\mathcal{Y}]_{\theta}$.
\end{lemma}

The following Lemma is an elementary fact from the interpolation theory.
\begin{lemma}\label{l7.2}
Let $\theta\in(0,\,1)$, $(\mathcal{X}_0,\,\mathcal{Y}_0)$ and $(\mathcal{X}_1,\,\mathcal{Y}_1)$ be
compatible couples of Banach spaces. Then
$$[\mathcal{X}_0\cap \mathcal{Y}_0,\,\mathcal{X}_1\cap \mathcal{Y}_1]_{\theta}\subset[\mathcal{X}_0,\, \mathcal{Y}_0]_{\theta}\cap[\mathcal{X}_1,\, \mathcal{Y}_1]_{\theta}.$$
\end{lemma}

The following lemma is the reiteration theorem of the complex interpolation method.
\begin{lemma}{\rm{\cite[Theorem 4.6.1]{bl76}}}\label{l7.3}
Let $(\mathcal{X}_0,\,\mathcal{X}_1)$ be a
compatible couple of Banach spaces and
$$\mathcal{Y}_0= [\mathcal{X}_0,\,\mathcal{X}_1]_{\theta_0}\ \ and \ \
\mathcal{Y}_1=[\mathcal{X}_0,\,\mathcal{X}_1]_{\theta_1},\ \ \ (0\leq\theta_0,\,\theta_1\leq1).$$
Assume that $\mathcal{X}_0\cap \mathcal{X}_1$  is a
dense subspace of $\mathcal{X}_0$, $\mathcal{X}_1$ and
of $\mathcal{Y}_0\cap\mathcal{Y}_1$.
Then for any $\eta\in[0,\,1]$,
$$[\mathcal{Y}_0,\,\mathcal{Y}_1]_{\eta}
=[\mathcal{X}_0,\,\mathcal{X}_1]_{(1-\eta)\theta_0+\eta\theta_1}$$
with equal norms.
\end{lemma}

\begin{proof} [Proof of Theorem \ref{t7.1}]
We first prove \eqref{e7.1}.
To prove this, by Theorem \ref{t5.1} and Lemma \ref{l7.1}, we only need to show that, for any
 $2\leq q<p<\infty$,
\begin{align}\label{e7.61}
\lf[\mathcal{BMO}^{c}(\bx,\,\cm),\,
L_q\mathcal{MO}^{c}(\bx,\,\cm)\r]_{\frac{q}{p}}=
L_p\mathcal{MO}^{c}(\bx,\,\cm).
\end{align}
In order to show this, we first find an isometry embedding $\mathfrak{R}$ from $L_s\mathcal{MO}^c(\bx,\,\cm)$ to $L_s(\lm\overline{\otimes}B(\ell_2);\,\ell_{\fz}^c)$. To this end,
for any ball $B\subset\bx$, let $H_{\cm}:=L_{\fz}(B)\overline{\otimes}\cm$ and $$\langle f,\,g\rangle_{H_{\cm}}
:=\fint_Bf(x)^*g(x)d\mu(x),$$ where $f,\,g:B\rightarrow \cm$ are two operator-valued functions such that the above integral makes sense. It follows from the definition that for any $a,\,b\in\cm$,
$\langle fa,\,gb\rangle_{H_{\cm}}=a^*\langle f,\,g\rangle_{H_{\cm}}b$.
Therefore, $H_{\cm}$ is a Hilbert $\cm$-module equipped with an $\cm$-valued inner product $\langle \cdot,\,\cdot\rangle_{H_{\cm}}$.
Let $e_{11}$ be the matrix unit in $B(\ell_2)$ with a unique nonzero entry at the position $(1,1)$, which is equal to $1$. Denote by $e:=1\otimes e_{11}\in \cm\overline{\otimes}B(\ell_2)$ and $C(\cm):=(\cm\overline{\otimes}B(\ell_2))e$, where $1$ is the unit of $\mathcal{M}$. Then from \cite{j02}, we know that there exists an $\cm$-module isomorphism
$u_B:H_{\cm}\rightarrow C(\cm)$ such that for any $f,\,g\in H_{\cm}$,
$$\langle u_B(f),\,u_B(g)\rangle_{C(\cm)}=\langle f,\,g\rangle_{H_{\cm}},$$
that is, $u_B(f)^*u_B(g)=\fint_Bf^*g \otimes e_{11}$.
For any $s\in(2,\,\fz]$, we define the map $\mathfrak{R}$ from
$L_s\mathcal{MO}^c(\bx,\,\cm)$ to $L_s(\lm\overline{\otimes}B(\ell_2);\,\ell_{\fz}^c)$ as follows:
$$\mathfrak{R}(f)(x):=\lf\{u_B(f-f_B)(x)\chi_B(x)\r\}_B.$$
Then
for any $f\in L_s\mathcal{MO}^c(\bx,\,\cm)$,
\begin{align*}
\|\mathfrak{R}(f)\|_{L_s(\lm\overline{\otimes}B(\ell_2);\,\ell_{\fz}^c)}
=&\lf\|\lf|u_B(f-f_B)\chi_B\r|^2\r\|^{1/2}_{L_\frac{s}{2}(\lm\overline{\otimes}B(\ell_2);\,\ell_{\fz})}\\
=&\lf\|\sup_{x\in B\subset\bx}^{\ \ \ \ \ +}\lf|u_B(f-f_B)\chi_B\r|^2\r\|_{L_{\frac{s}{2}}(\lm\overline{\otimes}B(\ell_2))}^{1/2}\\
=&\lf\|\sup_{x\in B\subset\bx}^{\ \ \ \ \ +}
\fint_B|f(y)-f_B|^2\,d\mu(y)\r\|_{L_{\frac{s}{2}}(\lm)}^{1/2}\\
=&\|f\|_{L_s\mathcal{MO}^c(\bx,\,\cm)}.
\end{align*}
This, in combination with Lemma \ref{l7.1x0} and interpolation, yields
$$\lf\|\mathfrak{R}:[\mathcal{BMO}^c(\bx,\,\cm),\,L_p\mathcal{MO}^c(\bx,\,\cm)]_{\frac{p}{q}}
\rightarrow L_q(\lm\overline{\otimes}B(\ell_2);\,\ell_{\fz}^c)\r\|\leq1.$$
Therefore, for any $f\in\mathcal{BMO}^c(\bx,\,\cm)\cap L_p\mathcal{MO}^c(\bx,\,\cm)$,
\begin{align*}
\|f\|_{L_q\mathcal{MO}^c(\bx,\,\cm)}
=\|\mathfrak{R}(f)\|_{L_q(\lm\overline{\otimes}B(\ell_2);\,\ell_{\fz}^c)}
\leq\|f\|_
{[\mathcal{BMO}^c(\bx,\,\cm),\,L_p\mathcal{MO}^c(\bx,\,\cm)]_{\frac{p}{q}}}.
\end{align*}
From this and Lemma \ref{l7.1x}, we further conclude that
\begin{align*}
[\mathcal{BMO}^c(\bx,\,\cm),\,L_p\mathcal{MO}^c(\bx,\,\cm)]_{\frac{p}{q}}
\subseteq L_q\mathcal{MO}^c(\bx,\,\cm).
\end{align*}

Now we claim that
\begin{align}\label{e7.65}
[\mathcal{H}^c_1(\bx,\,\cm),\,\mathcal{H}_{p'}^c(\bx,\,\cm)]_{\frac{p}{q}}
\subseteq \mathcal{H}_{q'}^c(\bx,\,\cm).
\end{align}
In fact, it follows from the definition of $\Phi$ (see Lemma \ref{d3.1}) that  for any
$s\in(1,\,\fz)$, $\Phi$ is an isometric embedding from $\mathcal{H}_{s}^c(\bx,\,\cm)$ to
$L_{s}(\lm,\,L_2^c(\bx\times \mathbb{R}_+,\,d\mu(y)dt))$. By this, Lemma \ref{l7.0} and the interpolation,
we have
$$\lf\|\Phi:[\mathcal{H}_{1}^c(\bx,\,\cm),\,\mathcal{H}_{p'}^c(\bx,\,\cm)]
_{\frac{p}{q}}\rightarrow L_{q'}(\lm,\,L_2^c(\bx\times \mathbb{R}_+,\,d\mu(y)dt))\r\|\leq1.$$
Thus,
for any $f\in\mathcal{H}_{1}^c(\bx,\,\cm)\cap\mathcal{H}_{p'}^c(\bx,\,\cm)$,
$$\|f\|_{\mathcal{H}_{q'}^c(\bx,\,\cm)}
=\|\Phi(f)\|_{L_{q'}(\lm,\,L_2^c(\bx\times \mathbb{R}_+,\,d\mu(y)dt))}
\leq\|f\|_{[\mathcal{H}_{1}^c(\bx,\,\cm),\,\mathcal{H}_{p'}^c(\bx,\,\cm)]
_{\frac{p}{q}}}.$$
This, together with Lemma \ref{l7.1x}, yields \eqref{e7.65}.
From the claim \eqref{e7.65}, the duality and Theorem \ref{t5.1}, we deduce that
\begin{align*}
[\mathcal{BMO}^c(\bx,\,\cm),\,L_p\mathcal{MO}^c(\bx,\,\cm)]_{\frac{p}{q}}
\supseteq L_q\mathcal{MO}^c(\bx,\,\cm).
\end{align*}
This completes the proof of \eqref{e7.61}.

Next we show \eqref{e7.2}, which will be divided into two cases: $\mathcal{Y}=L_1(\lm)$
 and $\mathcal{Y}=\mathcal{H}_1^{cr}(\bx,\,\cm)$.

\textbf{Case 3:} $\mathcal{Y}=L_1(\lm)$. In this case, if $\mathcal{X}=L_{\fz}(\lm)$, the conclusion is well-known.
If $\mathcal{X}=\mathcal{BMO}^{cr}(\bx,\,\cm)$, then we first assume $2\leq p<\fz$. Under this assumption, by Lemma \ref{l7.2}, Remark \ref{r2.9} (iii) and \eqref{e7.1}, we get
\begin{align*}
&\lf[\mathcal{BMO}^{cr}(\bx,\,\cm),\,L_2(\lm)\r]_{\frac{2}{p}}\\
=&[\mathcal{BMO}^{cr}(\bx,\,\cm),\,\mathcal{H}_2^c(\bx,\,\cm)\cap \mathcal{H}_2^r(\bx,\,\cm)]_{\frac{2}{p}}\\
\subseteq&[\mathcal{BMO}^c(\bx,\,\cm),\,\mathcal{H}_2^c(\bx,\,\cm)]_{\frac{2}{p}}
\cap[\mathcal{BMO}^r(\bx,\,\cm),\, \mathcal{H}_2^r(\bx,\,\cm)]_{\frac{2}{p}}\\
=&\mathcal{H}_p^c(\bx,\,\cm)\cap\mathcal{H}_p^r(\bx,\,\cm)\\
=&\mathcal{H}_p^{cr}(\bx,\,\cm)=L_p(\lm).
\end{align*}
Conversely, since $\mathcal{BMO}^{cr}(\bx,\,\cm)\supseteq L_{\infty}(\lm)$, we have
\begin{align*}
\lf[\mathcal{BMO}^{cr}(\bx,\,\cm),\,L_2(\lm)\r]_{\frac{2}{p}}
\supseteq\lf[L_{\infty}(\lm),\,L_2(\lm)\r]
_{\frac{2}{p}}=L_p(\lm).
\end{align*}
From the above conclusions, we conclude that for any $p\in[2,\,\fz)$,
\begin{align}\label{e7.6}
\mathcal{X}_1:=L_p(\lm)
=\lf[\mathcal{BMO}^{cr}(\bx,\,\cm),\,L_2(\lm)\r]_{\frac{2}{p}}
=:[\mathcal{X}_0,\,\mathcal{X}_2]_{\frac{2}{p}}.
\end{align}
It is obvious that, for $p\in[2,\,\fz)$,
\begin{align}\label{e7.7}
\mathcal{X}_2:=L_2(\lm)
=\lf[L_p(\lm),\,L_1(\lm)\r]_{\frac{p-2}{2(p-1)}}
=:[\mathcal{X}_1,\,\mathcal{X}_3]_{\frac{p-2}{2(p-1)}}.
\end{align}
By \eqref{e7.6}, \eqref{e7.7} and Lemma \ref{l7.0x}, we conclude that for any $p\in[2,\,\fz)$,
\begin{align*}
L_p(\lm)=[\mathcal{BMO}^{cr}(\bx,\,\cm),\,L_1(\lm)]_{\frac{1}{p}}.
\end{align*}
For $p\in[1,\,2)$, it follows from Lemma \ref{l7.3} that
\begin{align*}
L_p(\lm)
=&[L_2(\lm),\,L_1(\lm)]_{\frac{2-p}{p}}\\
=&\lf[\lf[\mathcal{BMO}^{cr}(\bx,\,\cm),\,L_1(\lm)\r]_{\frac{1}{2}}
,\,[\mathcal{BMO}^{cr}(\bx,\,\cm),\,L_1(\lm)]_{1}\r]_{\frac{2-p}{p}}\\
=&[\mathcal{BMO}^{cr}(\bx,\,\cm),\,L_1(\lm)]_{\frac{1}{p}}.
\end{align*}

\textbf{Case 4:} $\mathcal{Y}=\mathcal{H}_1^{cr}(\bx,\,\cm)$.
Let $q\in(1,\,p)$.
Then
\begin{align}\label{e7.8}
[\mathcal{H}_1^{cr}(\bx,\,\cm),\,L_p(\lm)]_{\frac{p'}{q'}}=L_q(\lm).
\end{align}
If $\mathcal{X}=L_{\fz}(\lm)$, then
 \begin{align*}
\widetilde{\mathcal{X}}_1:=L_{p}(\lm)
=[L_{\fz}(\lm),\,L_q(\lm)]_{\frac{q}{p}}
=:\lf[\widetilde{\mathcal{X}}_0,\,\widetilde{\mathcal{X}}_2\r]_{\frac{q}{p}}.
\end{align*}
Moreover, it follows from \eqref{e7.8} that
 \begin{align}\label{e7.9}
\widetilde{\mathcal{X}}_2:=L_{q}(\lm)
=[L_{p}(\lm),\,\mathcal{H}_1^{cr}(\bx,\,\cm)]_{1-\frac{p'}{q'}}
=:\lf[\widetilde{\mathcal{X}}_1,\,\widetilde{\mathcal{X}}_3\r]_{1-\frac{p'}{q'}}.
\end{align}
Therefore, by Lemma \ref{l7.0x}, we have
$$L_{p}(\lm)=\lf[L_{\fz}(\lm),\,\mathcal{H}_1^{cr}(\bx,\,\cm)\r]_{\frac{1}{p}}.
$$
If $\mathcal{X}=\mathcal{BMO}(\bx,\,\cm)$, by \textbf{Case 3} and Lemma \ref{l7.3}, we get
 \begin{align*}
\widetilde{\mathcal{X}}_1:&=L_{p}(\lm)\\
&=[\mathcal{BMO}^{cr}(\bx,\,\cm),\,L_1(\lm)]_{\frac{1}{p}}\\
&=\lf[[\mathcal{BMO}^{cr}(\bx,\,\cm),\,L_1(\lm)]_{0},\,
[\mathcal{BMO}^{cr}(\bx,\,\cm),\,L_1(\lm)]_{\frac{1}{q}}\r]_{\frac{q}{p}}\\
&=[\mathcal{BMO}^{cr}(\bx,\,\cm),\,L_q(\lm)]_{\frac{q}{p}}\\
&=:\lf[\widetilde{\mathcal{X}}_0,\,\widetilde{\mathcal{X}}_2\r]_{\frac{q}{p}}.
\end{align*}
From this, \eqref{e7.9} and Lemma \ref{l7.0x}, we conclude that, for any $p\in(1,\,\fz)$,
$$L_{p}(\lm)=\lf[\mathcal{BMO}^{cr}(\bx,\,\cm),\,\mathcal{H}_1^{cr}(\bx,\,\cm)\r]_{\frac{1}{p}}.
$$
This completes the proof of \textbf{Case 4} and hence the proof of Theorem \ref{t7.1}.
\end{proof}
The second main result of this subsection is the following.
\begin{theorem}\label{t7.2x}
Let $1\leq q<p<\infty$. Then
$$\lf[\mathcal{BMO}^{c}(\bx,\,\cm),\,\mathcal{H}_q^c(\bx,\,\cm)\r]_{\frac{q}{p},\,p}=
\mathcal{H}_p^c(\bx,\,\cm)$$
with equivalent norms. Similar result also holds between row BMO and Hardy spaces.
\end{theorem}
The following lemma establish the connection
between the complex and real interpolation.
\begin{lemma}{\rm{\cite[Theorem 4.7.2]{bl76}}}\label{l7.21}
Let $0<p\leq\fz$, $0<\theta_0<\theta_1<1$, $\theta>\theta_0$ and $(\mathcal{X}_0,\,\mathcal{X}_1)$ be a
compatible couple of Banach spaces. Then we have
$$[\mathcal{X}_0,\,\mathcal{X}_1]_{\theta,\,p}
=\lf[[\mathcal{X}_0,\,\mathcal{X}_1]_{\theta_0}
,\,[\mathcal{X}_0,\,\mathcal{X}_1]_{\theta_1}\r]_{\frac{\theta-\theta_0}
{\theta_1-\theta_0},\,p}$$
with equivalent norms.
\end{lemma}

\begin{proof} [Proof of Theorem \ref{t7.2x}]
For any $1<p,\,u,\,s<\fz$ and $\eta\in(0,\,1)$ such that
$$\frac{1}{p}=\frac{1-\eta}{u}+\frac{\eta}{s},$$
 by Lemma \ref{l7.0} and a similar argument as above for the complex method, we obtain
\begin{align}\label{e7.21}
\lf[\mathcal{H}_u^c(\bx,\,\cm),\,\mathcal{H}_s^c(\bx,\,\cm)\r]_{\eta,\,p}
=\mathcal{H}_p^c(\bx,\,\cm).
\end{align}
Let $\theta:=\frac{q}{p}$, $0<\theta_0<\frac{\theta}{2}<1$ and $\theta_1=\theta-\theta_0$. Then there exists a number $\eta$ such that
$\theta=(1-\eta)\theta_0+\eta\theta_1$.
From this, Lemma \ref{l7.21}, Theorem \ref{t7.1} and \eqref{e7.21}, we deduce that
\begin{align*}
&[\mathcal{BMO}^c(\bx,\,\cm),\,\mathcal{H}_q^c(\bx,\,\cm)]_{\theta,\,p}\\
=&\lf[[\mathcal{BMO}^c(\bx,\,\cm),\,\mathcal{H}_q^c(\bx,\,\cm)]_{\theta_0}
,\,[\mathcal{BMO}^c(\bx,\,\cm),\,\mathcal{H}_q^c(\bx,\,\cm)]_{\theta_1}\r]
_{\eta,\,p}\\
=&\lf[\mathcal{H}_{\frac{q}{\theta_0}}^c(\bx,\,\cm)
,\,\mathcal{H}_{\frac{q}{\theta_1}}^c(\bx,\,\cm)\r]
_{\eta,\,p}\\
=&\mathcal{H}_{p}^c(\bx,\,\cm).
\end{align*}
with equivalent norms.
This finishes the proof of Theorem \ref{t7.2x}.
\end{proof}

\begin{theorem}\label{t7.2}
Let $p\in[1,\,\infty)$. Then
$$\lf[\mathcal{X}_0,\,\mathcal{X}_1\r]_{\frac{1}{p},\,p}=
L_p(\lm),$$
where
$\mathcal{X}_0=\mathcal{BMO}^{cr}(\bx,\,\cm)$ or $L_{\infty}(\lm)$,
$\mathcal{X}_1=\mathcal{H}_{1}^{cr}(\bx,\,\cm)$ or $L_1(\lm)$.
\end{theorem}
\begin{proof} [Proof of Theorem \ref{t7.2}]
For any $p\in(1,\,\fz)$, we choose $p_1$ and $\,p_2$ such that $1\leq p_1<p<p_2<\fz$.
Then there exists a number $\theta\in(0,\,1)$ such that
$\frac{1-\theta}{p_1}+\frac{\theta}{p_2}=\frac{1}{p}$.
This, together with Lemma \ref{l7.21} and Theorem \ref{t7.1}, implies
\begin{align*}
\lf[\mathcal{X}_0,\,\mathcal{X}_1\r]_{\frac{1}{p},\,p}
&=\lf[\lf[\mathcal{X}_0,\,\mathcal{X}_1\r]_{\frac{1}{p_1}}
,\,\lf[\mathcal{X}_0,\,\mathcal{X}_1\r]_{\frac{1}{p_2}}\r]_{\theta,\,p}\\
&=\lf[L_{p_1}(\lm)
,\,L_{p_2}(\lm)\r]_{\theta,\,p}\\
&=L_p(\lm).
\end{align*}
This completes the proof of Theorem \ref{t7.2}.
\end{proof}


\bigskip

\section{Application}\label{s8}
\hskip\parindent
In this section, we obtain the $L_p(\mathcal{N})$-boundedness of non-commutative Calder\'{o}n-Zygmund operators, with operator-valued kernel, by establishing $L_\infty(\mathcal{N})\rightarrow \mathcal{BMO}^{c}(\bx,\,\cm)$ boundedness for such kinds of operators and then applying interpolation theory established in Section \ref{s7}.
\begin{definition}
We say that $T^c$ (resp. $T^r$) is a {left (resp. right) $\mathcal{M}$-valued Calder\'on-Zygmund operators} on $(\bx,\,d,\,\mu)$ for some $\vartheta\in (0,1)$  if $T^c$ (resp. $T^r$) is bounded on $L_2(\mathcal{N})$ and has the associated left (resp. right) $\mathcal{M}$-valued kernel $\mathcal{K}(x,\,y)$ such that for locally integrable function $f:\bx\rightarrow \mathcal{S}_\mathcal{M}$ and $x\notin\supp f$,
$$T^c(f)(x):=\int_{\bx}\mathcal{K}(x,\,y)f(y)\,d\mu(y),\ \ \ x\in\bx,$$
$$T^r(f)(x):=\int_{\bx}f(y)\mathcal{K}(x,\,y)\,d\mu(y),\ \ \ x\in\bx,$$
respectively, and $\mathcal{K}(x,\,y)$ satisfies the following estimates: for any $x,\,y\in \bx$ with $x\neq y$,
$$\lf\|\mathcal{K}(x,\,y)\r\|_{\cm}\leq \frac{C}{V(x,\,y)},$$
and for any $x,\,x',\,y\in\bx$ with $d(x',\,x)<\frac{d(x,\,y)}{2}$,
$$\lf\|\mathcal{K}(x',\,y)-\mathcal{K}(x,\,y)\r\|_{\cm}
+\lf\|\mathcal{K}(y,\,x')-\mathcal{K}(y,\,x)\r\|_{\cm}\leq C\frac{d(x',\,x)^{\vartheta}}{V(x,\,y)d(x,\,y)^{\vartheta}}.$$
\end{definition}
\begin{theorem}\label{t8.1}
Assume that $T^c$ is a {\it $\mathcal{M}$-valued left-Calder\'on-Zygmund operator}  on $(\bx,\,d,\,\mu)$.
Then $T^c$ is bounded from $L_\infty(\mathcal{N})$ to $\mathcal{BMO}^{c}(\bx,\,\cm)$.
Furthermore, $T^c$ is bounded on $L_p(\mathcal{N})$ for all $1<p<\infty$.
Similar result also holds for {\it $\mathcal{M}$-valued right-Calder\'on-Zygmund operator $T^r$}.
\end{theorem}
\begin{proof} [Proof of Theorem \ref{t8.1}]
Let $f\in L_\infty(\mathcal{N})$ and ball $B:=B(x_B,\,r_B)\subset\bx$ with $x_B\in\bx$ and $r_B>0$.
Then we can write
\begin{align*}
f&=(f-f_{4B})\chi_{4B}+(f-f_{4B})
\chi_{(4B)^{\complement}}
+f_{4B}\\
&=:f_1+f_2+f_3.
\end{align*}
Set
$$\mathfrak{C}_{4B}:=\int_{(4B)^{\complement}}
\mathcal{K}(x_B,\,y)\lf[f(y)-f_{4B}\r]\,d\mu(y).$$
Therefore, it is easy to see that
\begin{align*}
&\fint_{B}\lf|T^c(f_1+f_2)(x)-\mathfrak{C}_{4B}\r|^2\,d\mu(x)\\
\leq&2\fint_{B}\lf|T^c(f_1)(x)\r|^2\,d\mu(x)
 +2\fint_{B}
\lf|\int_{(4B)^\complement}\lf[\mathcal{K}(x,\,y)-\mathcal{K}(x_B,\,y)\r]f_2(y)
\,d\mu(y)\r|^2\,d\mu(x)\\
=&:\mathrm{N}_1+\mathrm{N}_2.
\end{align*}
For the term $\mathrm{N}_1$, from the duality theory, \eqref{e2.1} and the boundedness of $T^c$ on
$L_2(\lm)$, we see that
\begin{align*}
\|\mathrm{N}_1\|_{\cm}^{1/2}\lesssim&\lf\|\fint_{B}\lf|T^c(f_1)(x)\r|^2\,d\mu(x)\r\|_{\cm}^{1/2}\\
\thicksim&\mu(B)^{-1/2}\sup_{\|u\|_{L_2(\cm)=1}}\lf[\int_{\bx}
\lf\langle\lf|T^c(f_1)(x)\r|^2u,\,u\r\rangle_{L_2(\cm)}\,d\mu(x)\r]^{1/2}\\
\thicksim&\mu(B)^{-1/2}\sup_{\|u\|_{L_2(\cm)=1}}\lf[\int_{\bx}
\lf\|T^c(f_1u)(x)\r\|^2_{L_2(\cm)}\,d\mu(x)\r]^{1/2}\\
\lesssim&\mu(B)^{-1/2}\sup_{\|u\|_{L_2(\cm)=1}}\lf[\int_{\bx}
\lf\|f_1(x)u\r\|^2_{L_2(\cm)}\,d\mu(x)\r]^{1/2}\\
\thicksim&\mu(B)^{-1/2}\lf\|\int_{\bx}\lf|f_1(x)\r|^2\,d\mu(x)\r\|_{\cm}^{1/2}\\
\lesssim&\lf\|\fint_{4B}\lf|f(x)-
f_{4B}\r|^2\,d\mu(x)\r\|_{\cm}^{1/2}\\
\lesssim&\|f\|_{\mathcal{BMO}^{c}(\bx,\,\cm)}.
\end{align*}
For the term $\mathrm{N}_2$, from Lemma \ref{e2.5} and \eqref{e2.1}, we conclude that
\begin{align*}
\mathrm{N}_2\lesssim&\fint_{B}
\lf|\int_{\bx}\lf[\mathcal{K}(x,\,y)-\mathcal{K}(x_B,\,y)\r]f_2(y)\,d\mu(y)\r|^2\,d\mu(x)\\
\leq&\fint_{B}
\int_{(4B)^{\complement}}
\lf\|\mathcal{K}(x,\,y)-\mathcal{K}(x_B,\,y)\r\|_{\cm}\,d\mu(y)\\
&\ \ \ \times\int_{(4B)^{\complement}}
\lf\|\mathcal{K}(x,\,y)-\mathcal{K}(x_B,\,y)\r\|_{\cm}^{-1}
\lf|\lf[\mathcal{K}(x,\,y)-\mathcal{K}(x_B,\,y)\r]f_2(y)\r|^2\,d\mu(y)d\mu(x)\\
\lesssim& \fint_{B}
\int_{(4B)^{\complement}}
\frac{[d(x,\,x_B)]^{\vartheta}}{V(y,\,x_B)[d(y,\,x_B)]^{\vartheta}}\,d\mu(y)
\int_{(4B)^{\complement}}
\frac{[d(x,\,x_B)]^{\vartheta}}{V(y,\,x_B)[d(y,\,x_B)]^{\vartheta}}|f_2(y)|^2\,d\mu(y)d\mu(x)\\
\lesssim&\sum_{j=0}^{\infty}\frac{1}{2^{\delta j}\mu(B(x_B,\,2^{j+1}r_B))}\int_{2^{j+3}B\setminus
  2^{j+2}B}
\lf|f(y)-f_{4B}\r|^2\,d\mu(y)\\
\lesssim&\sum_{j=0}^{\infty}\frac{1}{2^{\delta j}\mu(B(x_B,\,2^{j+1}r_B))}\int_{2^{j+3}B\setminus
  2^{j+2}B}
\lf|f(y)-f_{2^{j+3}B}\r|^2+\lf|f_{2^{j+3}B}-f_{4B}\r|^2\,d\mu(y)\\
\lesssim&\|f\|_{\mathcal{BMO}^{c}(\bx,\,\cm)}^2,
\end{align*}
where the last inequality used a standard fact that $|f_{2^{j}B}-f_{4B}|\lesssim\log(j+1)\|f\|_{\mathcal{BMO}^{c}(\bx,\,\cm)}$.
Combining the above two estimates yields $\|T^c(f_1+f_2)\|_{\mathcal{BMO}^{c}(\bx,\,\cm)}\lesssim \|f\|_{\mathcal{BMO}^{c}(\bx,\,\cm)}\lesssim \|f\|_{L_\infty(\mathcal{N})}.$ It remains to show that $\|T^c(f_3)\|_{\mathcal{BMO}^{c}(\bx,\,\cm)}\lesssim \|f\|_{L_\infty(\mathcal{N})}.$ Indeed, following the proof of \cite[p.194]{xxx16}, one can show that $T^c(\chi_\bx)$ can be naturally defined as a function in $\mathcal{BMO}^{c}(\bx,\,\cm)$, which implies that
$$\|T^c(f_3)\|_{\mathcal{BMO}^{c}(\bx,\,\cm)}\lesssim \|T^c(\chi_\bx)\|_{\mathcal{BMO}^{c}(\bx,\,\cm)}\|f_B\|_{L_\infty(\mathcal{M})}\lesssim \|f\|_{L_\infty(\mathcal{N})}.$$
Combining the above facts imply that $T^c$ is bounded from $L_\infty(\mathcal{N})$ to $\mathcal{BMO}^{c}(\bx,\,\cm)$. Furthermore, a direct calculation see that $(T^c)^*$ is also a {\it $\mathcal{M}$-valued left-Calder\'on-Zygmund operator} with kernel $\mathcal{K}(y,\,x)^*$. Then, a standard interpolation together with duality argument deduce the $L_p(\mathcal{N})$-boundedness of the targeted operator. This ends the proof of Theorem \ref{t8.1}.
\end{proof}


\textbf{Acknowledgements.} The authors are supported by National Natural Science Foundation of China (No. 12071355, No. 12325105) and Fundamental Research Funds for the Central Universities (No. 2042022kf1185).


\bigskip

\end{document}